\documentclass{article}
\usepackage{amsmath}
\usepackage{amssymb}
\usepackage{graphicx}
\usepackage{epstopdf}
\usepackage{amsthm}
\usepackage{amsfonts}
\usepackage{array}
\usepackage{amsthm}
\usepackage{tabularx}
\usepackage{url}
\usepackage{setspace}
\usepackage{geometry}
\usepackage{subfigure}
\usepackage{authblk}
\usepackage{hyperref}
\usepackage{xcolor}
\usepackage{multirow}
\usepackage{lipsum}
\newtheoremstyle{break}
  {\topsep}{\topsep}%
  {\itshape}{}%
  {\bfseries}{}%
  {\newline}{}%
\theoremstyle{break}
\newtheorem{theorem}{Theorem}
\newtheorem{remark}{Remark}
\newtheorem{definition}{Definition}

\newtheorem{lemma}{Lemma}

\newtheorem{algorithm}{Algorithm}

\begin{document}
\title{Ridge Regression Revisited: \\ Debiasing, Thresholding and Bootstrap}
\author{ Yunyi Zhang
\thanks{Department of Mathematics,
        Univ.~of California--San Diego,
          La Jolla, CA 92093-0112, USA;
             yuz334@ucsd.edu }$\ \text{ and }$ Dimitris N. Politis\thanks{Department of Mathematics
and Halicioglu Data Science Institute,
        Univ.~of California--San Diego,
          La Jolla, CA 92093-0112, USA;
                        dpolitis@ucsd.edu }}
\maketitle
\begin{abstract}
The success of the Lasso  in  the era of high-dimensional data can be attributed to its conducting an implicit
model selection, i.e., zeroing out regression coefficients that are not significant. By contrast,
classical ridge regression can not reveal a potential sparsity of parameters, and
may also introduce a large bias under the high-dimensional setting.
Nevertheless, recent work on the Lasso  involves debiasing and thresholding, the latter in order
to further enhance the model selection.
As a consequence,   ridge regression may be worth another look since --after  debiasing and thresholding--
it may offer some advantages over the Lasso, e.g.,  it can be easily computed using a closed-form expression.
In this paper, we define
 a debiased and thresholded ridge regression method, and prove a  consistency result and a Gaussian approximation theorem. We further introduce a wild
 bootstrap algorithm to construct   confidence regions and perform hypothesis testing for a linear combination  of parameters. In addition to estimation,
we consider the problem of prediction, and present a novel, hybrid
 bootstrap algorithm tailored for  prediction intervals.
Extensive numerical simulations further  show that the debiased and thresholded ridge regression has  favorable
 finite sample performance and may be preferable in some settings.
\end{abstract}

\newpage
\section{Introduction}
Linear regression is a fundamental topic in statistical inference. The classical setting assumes the dimension of parameters in a linear model is constant. However, in the modern era, observations may have a comparable or even larger dimension than the number of samples. To perform a consistent estimation
with high-dimensional data, statisticians often assume the underlying parameters are sparse (i.e., the parameter vector contains lots of zeros), and proceed with statistical inference based on this assumption.

The success of the Lasso  in the setting of high-dimensional data can be attributed to its conducting an implicit
model selection, i.e., zeroing out regression coefficients that are not significant;
see Tibshirani \cite{doi:10.1111/j.2517-6161.1996.tb02080.x}. More recent work includes:
  Meinshausen and B\"{u}hlmann \cite{meinshausen2006}, Meinshausen and Yu \cite{meinshausen2009}, and van de Geer \cite{vandegeer2008} for the Lasso estimator's (model-selection) consistency and applications; Chatterjee and Lahiri \cite{10.2307/41059185}, \cite{10.2307/41416396}, Zhang and Cheng \cite{doi:10.1080/01621459.2016.1166114}, and Dezeure et al. \cite{Dezeure2017} for confidence interval construction and hypothesis testing; and Javanmard and Montanari \cite{javanmard2018}, Fan and Li \cite{doi:10.1198/016214501753382273}, and Chen and Zhou \cite{chen2020} for improvements of the Lasso estimator. Although the Lasso has the desirable property of
zeroing out some regression coefficients, van de Geer et al. \cite{10.1214/11-EJS624}
proposed to further {\it threshold}  the estimated coefficients, leading to a
sparser fitted model. Furthermore, Javanmard and Montanari \cite{10.5555/2031491}, and Dezeure et al. \cite{Dezeure2017},  proposed to {\it  debias} the
Lasso in constructing confidence intervals. See van de Geer \cite{10.1214/19-EJS1599} and Javanmard and Javadi \cite{10.1214/19-EJS1554} for recent works of debiased Lasso.

An alternative approach providing consistent estimators for a high dimensional linear model is  the so-called {\it post-selection inference}. It first applies Lasso to select influential parameters, then fits an ordinary least squares regression on the selected parameters; see Lee et al. \cite{lee2016}, Liu and Yu \cite{liu2013}, and Tibshirani et al. \cite{tibshirani2018}.
We refer to  B\"{u}hlmann and van de Geer \cite{10.5555/2031491} for a comprehensive overview  of the Lasso method for high dimensional data.

Ridge regression is   a classical method, and its estimator has a closed-form expression, making statistical inference easier than Lasso. However, there is relatively little research on the ridge regression under the high-dimensional setting. Shao and Deng \cite{shao2012} proposed a threshold ridge regression method and proved its consistency. Dai et al. \cite{DAI2018334} introduced a broken adaptive ridge estimator to approximate $L_0$ penalized regression. Dobriban and Wager \cite{dobriban2018} derived the limit of high dimensional ridge regression's expected predictive risk. B\"{u}hlmann \cite{buhlmann2013} used Lasso to correct the bias in a ridge regression estimator, while Lopes \cite{Lopes2014ARB} applied a residual-based
 bootstrap to construct confidence intervals.

Three issues have prevented ridge regression from being suitable for a high dimensional linear model:

1. {\it  The ridge regression cannot preserve/recover sparsity.} Typically, a ridge regression  estimator of the parameter vector
 will  not contain any zeros, even though the parameters may be sparse.

2. {\it Bias in the ridge regression estimator can be large.} To illustrate this, suppose the parameter of interest is $a^T\beta$ in a linear model $y = X\beta + \epsilon$; here, the dimension $p < n$ (the sample size), $X$ has rank $p$, and $a$ is a known vector. The ridge estimator  is $a^T\widetilde{\theta}^\star$ with $\widetilde{\theta}^\star = (X^TX + \rho_nI_p)^{-1}X^Ty$, for some $\rho_n > 0$, with $I_p$ denoting the $p$-dimensional identity matrix. Performing a thin singular value decomposition $X = P\Lambda Q^T$ (as in Theorem 7.3.2 in Horn and Johnson \cite{matrix}),
and assuming the error vector $\epsilon$
consists of independent identically distributed (i.i.d.) components,
  the bias and the standard deviation can be calculated (and controlled) as follows:
\begin{equation}
\begin{aligned}
\mathbf{E}a^T\widetilde{\theta}^\star - a^T\beta = -\rho_na^TQ(\Lambda^2 + \rho_n I_p)^{-1} Q^T\beta \ \mbox{ which implies } \   \vert \mathbf{E}a^T\widetilde{\theta}^\star - a^T\beta\vert\leq \frac{\rho_n\Vert a\Vert_2\times \Vert\beta\Vert_2}{\lambda^2_{p} + \rho_n}\\
\mbox{ and } \   \sqrt{Var(a^T\widetilde{\theta}^\star)} = \sqrt{Var(\epsilon_1)\times a^TQ(\Lambda^2 + \rho_n I_p)^{-2}\Lambda^2 Q^Ta}\leq\frac{\sqrt{Var(\epsilon_1)}\times \Vert a\Vert_2}{\lambda_p}.
\end{aligned}
\end{equation}

In the above,  $\lambda_p$ is the smallest singular value of $X$, and $\Vert\cdot\Vert_2$ is the Euclidean norm of a vector. If $\Vert\beta\Vert_2$ does not have a bounded order, the bias may tend to infinity. Another critical problem is that the absolute value of the bias can be significantly larger than the standard deviation, which makes constructing confidence intervals difficult.

3. {\it When the dimension of parameters is larger than the sample size, ridge regression estimates the projection of parameters on the linear space spanned by rows of $X$} (Shao and Deng \cite{shao2012}). The projection (which can now be considered to be the `parameters' of the linear model) is not sparse, bringing extra burdens for statistical inference.

The third issue comes from the nature of ridge regression, and it is not necessarily bad; our section \ref{EXPER} provides an example to illustrate this. The first two issues can be solved by {\it thresholding and debiasing}
respectively, yielding an {\it improved} ridge regression that will be the focus of this paper.
If the Lasso is in need of thresholding and debiasing --as van de Geer et al. \cite{10.1214/11-EJS624}, Dezeure et al. \cite{Dezeure2017}, and  B\"{u}hlmann and van de Geer \cite{10.5555/2031491} seem to suggest-- then it
loses some of its attractiveness, in which case  (improved)  ridge regression
may be worth another look.
If (improved)  ridge regression  turns out to have comparable performance to threshold Lasso, then the
former would be preferable since  it can be easily computed using a closed-form expression.
 Indeed, numerical simulations in section \ref{Numerical} indicate that   improved ridge regression has  favorable  finite-sample performance, and  has a further advantage over the Lasso:
 it is {\it robust} against a non-optimal choice of the   hyperparameters.

Apart from   point estimation using improved  ridge regression, this paper presents a Gaussian approximation theorem for the improved ridge regression estimator. Applying this result,   we
propose a wild bootstrap  algorithm to
  construct  a   confidence region for $\gamma = M\beta$ with $M$ a known matrix  and/or test  the null hypothesis $\gamma = \gamma_0$ with $\gamma_0$ a known vector, versus the alternative hypothesis $\gamma \neq \gamma_0$. The wild bootstrap was developed in the 1980s by Wu \cite{10.1214/aos/1176350142} and Liu \cite{10.1214/aos/1176351062};
its applicability to high-dimensional problems was recognized early on by Mammen \cite{10.1214/aos/1176349025}. Here we will use the wild bootstrap in its Gaussian residuals version that has been found useful in high-dimensional regression; see Chernozhukov et al. \cite{chernozhukov2013}.
 Estimating and testing $\gamma$ are important problems in econometrics, e.g., Dolado and L\"{u}tkepohl \cite{doi:10.1080/07474939608800362}, Sun \cite{https://urldefense.com/v3/__https://doi.org/10.1111/j.1368-423X.2012.00390.x__;!!Mih3wA!Q9aoQVO8cn4eWInRmrOzFAYeiX19a7xKzrpyC3_noHS-Dc6SZcpIQu17FxU-vidS2g$ }, \cite{SUN2011345}, and Gon\c{c}alves and Vogelsang \cite{gv2011}. Besides, estimating $\gamma$ directly contributes to prediction, which is an important topic in modern age statistics.

Finally, we consider statistical  prediction   based on the improved ridge regression estimator for a high-dimensional linear model. For a regression problem, quantifying a predictor's accuracy can be as important as predicting accurately.
To do that, it is useful to be able to construct a  prediction interval to accompany the point
prediction; this is usually done by some form of bootstrap; see
 Stine \cite{doi:10.1080/01621459.1985.10478220} for a classical result,  and Politis \cite{model_free} for a comprehensive treatment of both model-based and model-free prediction intervals in regression.
As an alternative to the bootstrap, conformal prediction may be a tool to yield
prediction intervals; see e.g. Romano et al. \cite{NEURIPS2019_5103c358} and Romano et al. \cite{doi:10.1080/01621459.2019.1660174}.
In our point of view, however,  the bootstrap
 is preferable as it captures the underlying variability of
estimated quantities; Section \ref{BOOTPRD} in what follows gives the details.

The remainder of this paper is organized as follows: Section \ref{prelimary} introduces   frequently used notations and assumptions. Section \ref{CLT} presents the consistency result and the Gaussian approximation theorem for the improved ridge regression estimator. Section \ref{BOOT} constructs a confidence region for $\gamma = M\beta$, and tests the null hypothesis $\gamma = \gamma_0$ versus the alternative hypothesis $\gamma \neq\gamma_0$ via a bootstrap algorithm.
Section \ref{BOOTPRD}   constructs bootstrap prediction intervals in our ridge  regression setting
using a novel, hybrid resampling procedure.
 Finally,   Section \ref{Numerical} provides extensive simulations to illustrate the finite sample performance, while Section \ref{CONCLU} gives some concluding
remarks; technical  proofs are deferred to the Appendix.

\section{Preliminaries}
\label{prelimary}
Our work focuses on the  fixed design  linear model
\begin{equation}
y = X\beta + \epsilon \label{LINModel}
\end{equation}
where the (unknown) parameter vector $\beta$ is $p$-dimensional, and
the $n\times p$ fixed (nonrandom) design matrix $X$  is assumed to have rank $r$. Define the known
matrix of  linear combination coefficients as $M = (m_{ij})_{i = 1,2,...,p_1,j = 1,2,...,p}$  so that  $M$ has $p_1$ rows.
The  linear combination of interest are $\gamma = (\gamma_1,...,\gamma_{p_1})^T = M\beta$.

Perform a thin singular value decomposition $X = P\Lambda Q^T$  as in Theorem 7.3.2 in Horn and Johnson \cite{matrix};
here,  $P$ and $Q$ respectively is $n\times r$ and $p\times r$ orthonormal matrices
that  satisfy $P^TP = Q^TQ = I_r$, where $I_r$ denotes  the $r\times r$ identity matrix.
Furthermore,  $\Lambda = diag(\lambda_1,...,\lambda_r)$, and $\lambda_1\geq \lambda_2\geq...\geq \lambda_r>0$ are positive singular values of $X$. 

Denote $Q_\perp$ as the $p\times (p - r)$ orthonormal complement of $Q$; then we have
\begin{equation}
Q_{\perp}^T Q_{\perp} = I_{p - r},\ Q^TQ_\perp = 0,\ \text{and } QQ^T+Q_\perp Q_\perp^T = I_p ;
\end{equation}
in the above, $0$ is the $r\times (p-r)$ matrix having all elements $0$. Define $\zeta = Q^T\beta$ and $\theta = (\theta_1,...,\theta_p)^T= Q\zeta = QQ^T\beta$, then $X\beta = X\theta$, $\theta^T\theta = \zeta^TQ^TQ\zeta = \zeta^T\zeta$. According to Shao and Deng \cite{shao2012}, the ridge regression estimates $\theta$ rather than $\beta$.

Define $\theta_\perp = Q_\perp Q_\perp^T\beta$, so $\beta = \theta + \theta_\perp$. If the design matrix $X$ has rank $p \leq n$, then $Q_\perp$ does not exist. In this situation, we define $\theta_\perp = 0$, the $p$ dimensional vector with all elements $0$. For a threshold $b_n$, define the set $\mathcal{N}_{b_n} = \{i\ | \vert\theta_i\vert>b_n\}$. After selecting a suitable $b_n$, define
\begin{equation}
c_{ik} = \sum_{j\in\mathcal{N}_{b_n}}m_{ij}q_{jk},\ \forall\ i=1,2,...,p_1,\ k = 1,2,...,r,\ \text{and }\mathcal{M} = \{i\ | \sum_{k=1}^rc_{ik}^2 > 0\}
\label{def_cik}
\end{equation}
Define $\tau_i,\ i=1,2,...,p_1$ as
\begin{equation}
\begin{aligned}
\tau_i = \sqrt{\sum_{k=1}^r c_{ik}^2\left(\frac{\lambda_k}{\lambda_k^2 + \rho_n} + \frac{\rho_n\lambda_k}{(\lambda_k^2 + \rho_n)^2}\right)^2 + \frac{1}{n}}
\end{aligned}
\label{DefVar}
\end{equation}

We will use the standard order notations $O(\cdot),\ o(\cdot),\ O_p(\cdot),$ and $o_p(\cdot)$. For two numerical sequences $a_n, b_n, n = 1,2,...$, we say $a_n = O(b_n)$ if $\exists$ a constant $C>0$ such that $\vert a_n\vert\leq C\vert b_n\vert$ for all $n$, and $a_n = o(b_n)$ if $\lim_{n\to\infty} \frac{a_n}{b_n} = 0$. For two random variable sequences $X_n, Y_n$, we say $X_n = O_p(Y_n)$ if for any $0<\epsilon<1$, $\exists$ a constant $C_\epsilon > 0$ such that $\sup_{n} Prob(\vert X_n\vert\geq C_\epsilon \vert Y_n\vert)\leq \epsilon$; and $X_n = o_p(Y_n)$ if $\frac{X_n}{Y_n}\to_p 0$; see e.g. Definition 1.9 and Chapter 1.5.1 of Shao \cite{Mstat}.
{\it All order notations and convergences in this paper will be understood to hold
as   the sample size $n\to\infty$.}

 For a finite set $A$, $\vert A\vert$ denotes the number of elements in $A$. Notations
 $\exists$ and  $\forall $ denote ``there exists" and ``for all" respectively.
 $Prob^*\left(\cdot\right)$ and $\mathbf{E}^*\cdot$ respectively represent probability and expectation in the {\it ``bootstrap world"},
i.e., they are
 the conditional probability $Prob(\cdot|y)$ and the conditional expectation $\mathbf{E}(\cdot|y)$.

Suppose $H(x)$ is a cumulative distribution function and $0<\alpha < 1$; then the $1-\alpha$ quantile of $H$ is defined as
\begin{equation}
c_{1-\alpha} = \inf\{x\in\mathbf{R}|H(x)\geq 1-\alpha\} .
\label{defQuan}
\end{equation}
In particular, given some order statistics  $X_1\leq X_2\leq...\leq X_B$, the $1-\alpha$ sample quantile $C_{1-\alpha}$ is defined as
\begin{equation}
C_{1-\alpha} = X_{i_*}\ \ \text{such that }i_* = \min\left\{i \ \Big|\
\frac{1}{B}\sum_{j=1}^B\mathbf{1}_{X_j\leq X_i}\geq  1-\alpha \right\} .
\end{equation}
Other notations will be defined before being used. Without being explicitly specified, the convergence results in this paper assume the sample size $n\to\infty$.

The high dimensionality in this paper comes from two aspects: the number of parameters $p$ may increase with the sample size $n$, and (for statistical inference/hypothesis testing) the number of simultaneous linear combinations $p_1$ and $\vert\mathcal{M}\vert$ can also increase with $n$.


Our work adopts the following assumptions:

\textbf{Assumptions}

1. Assume a fixed design, i.e., the design matrix $X$ is deterministic. Also assume that
there exist constants $c_\lambda, C_\lambda > 0$, $0<\eta\leq 1/2$, such that the positive singular values of $X$ satisfy
\begin{equation}
C_\lambda n^{1/2}\geq \lambda_1\geq \lambda_2\geq...\geq\lambda_r\geq c_\lambda n^\eta .
\end{equation}
Furthermore, the Euclidean norm of $\theta$
is assumed to satisfy $\Vert\theta\Vert_2 = \sqrt{\sum_{i=1}^p\theta_i^2}= O(n^{\alpha_\theta})$ with $0<\alpha_\theta < 3\eta$.

2. The ridge parameter satisfies  $\rho_n = O(n^{2\eta - \delta})$ with a
positive constant $\delta$ such that $\frac{\eta+\alpha_\theta}{2}<\delta<2\eta$

3. The errors $\epsilon = (\epsilon_1,...,\epsilon_n)^T$
driving regression (\ref{LINModel}) are assumed to be  i.i.d.,
with $\mathbf{E}\epsilon_1 = 0$, and $\mathbf{E}\vert\epsilon_1\vert^m < \infty$
for some $m>4$.

4. The dimension of the  parameter
vector $\beta$   satisfies $p = O(n^{\alpha_p})$ for some constant $  \alpha_p  \in [0,m\eta )$
where $m, \eta$ are as defined in Assumptions 1--3. Furthermore,
the threshold $b_n$ is chosen as $b_n = C_b\times n^{-\nu_b}$ with constants $C_b,\nu_b > 0$ and $\nu_b + \frac{\alpha_p}{m} - \eta < 0$. We assume $\exists$ a constant $0<c_b<1$ such that $\max_{i\not\in\mathcal{N}_{b_n}}\vert\theta_i\vert\leq c_b\times b_n$, and $\min_{i\in\mathcal{N}_{b_n}}\vert\theta_i\vert\geq \frac{b_n}{c_b}$.

5. $\mathcal{M}$ (defined in \eqref{def_cik}) is not empty
and $\vert\mathcal{M}\vert = O(n^{\alpha_\mathcal{M}})$ with $\alpha_\mathcal{M}< m\eta$
where $m, \eta$ are as defined in Assumptions 1--3.  Besides,  assume $\exists$
 constants $c_\mathcal{M}, C_\mathcal{M}$ such that $0<c_{\mathcal{M}}<\sum_{k=1}^rc_{ik}^2\leq C_\mathcal{M}$ for all $i\in\mathcal{M}$. Also assume
\begin{equation}
\max_{i=1,2,...,p_1}\vert\sum_{j\not\in\mathcal{N}_{b_n}}m_{ij}\theta_j\vert = o\left(\frac{1}{\sqrt{n\log(n)}}\right)\ \text{and}\ \max_{i=1,2,...,p_1}\vert\sum_{j=1}^pm_{ij}\theta_{\perp,j}\vert = o\left(\frac{1}{\sqrt{n\log(n)}}\right)
\label{UnseenBias}
\end{equation}

6. $\exists$ a constant $\alpha_\sigma$ satisfying $\eta\geq \alpha_\sigma > 0$
such that
\begin{equation}
n^{-\nu_b}\sum_{j\not\in\mathcal{N}_{b_n}}\vert\theta_j\vert = O(n^{-\alpha_\sigma}),\ \ \frac{\sqrt{\vert\mathcal{N}_{b_n}\vert}}{n^{\eta}} = O(n^{-\alpha_\sigma})
\end{equation}

7. $\vert\mathcal{M}\vert\leq r$,   the matrix $T = (c_{ik})_{i\in\mathcal{M},k=1,2,...,r}$ has rank $\vert\mathcal{M}\vert$, and one   of the two following conditions holds true:

\qquad 7.1.
\begin{equation}
\max_{i\in\mathcal{M},l=1,2,...,n}\vert\frac{1}{\tau_i}\times \sum_{k=1}^rc_{ik}p_{lk}\left(\frac{\lambda_k}{\lambda_k^2 + \rho_n} + \frac{\rho_n\lambda_k}{(\lambda_k^2 + \rho_n)^2}\right)\vert = o(\min(n^{(\alpha_\sigma - 1)/2}\times \log^{-3/2}(n),\ n^{-1/3}\times\log^{-3/2}(n))
\end{equation}

\qquad 7.2. $\alpha_\sigma < 1/2$ and
\begin{equation}
\begin{aligned}
\vert\mathcal{M}\vert = o(n^{\alpha_\sigma} \times \log^{-3}(n))\\
\max_{i\in\mathcal{M},l=1,2,...,n}\vert\frac{1}{\tau_i}\times \sum_{k=1}^rc_{ik}p_{lk}\left(\frac{\lambda_k}{\lambda_k^2 + \rho_n} + \frac{\rho_n\lambda_k}{(\lambda_k^2 + \rho_n)^2}\right)\vert = O(n^{-\alpha_\sigma}\times\log^{-3/2}(n))
\end{aligned}
\end{equation}

\begin{remark}

The intuitive  meaning of assumption 4 is that the $\theta_i$s that are not being truncated should be significantly larger than the $\theta_i$ being truncated.  Furthermore,
note that
\begin{equation}
\sum_{l = 1}^n \left(\sum_{k=1}^rc_{ik}p_{lk}\left(\frac{\lambda_k}{\lambda_k^2 + \rho_n} + \frac{\rho_n\lambda_k}{(\lambda_k^2 + \rho_n)^2}\right)\right)^2 = \sum_{k = 1}^r c^2_{ik}\left(\frac{\lambda_k}{\lambda^2_k + \rho_n} + \frac{\rho_n\lambda_k}{(\lambda^2 + \rho_n)^2}\right)^2\leq \tau^2_i .
\end{equation}
Therefore, assumption 7 implies that no single element in the matrix $\left(\frac{1}{\tau_i}\sum_{k=1}^rc_{ik}p_{lk}\left(\frac{\lambda_k}{\lambda_k^2 + \rho_n} + \frac{\rho_n\lambda_k}{(\lambda_k^2 + \rho_n)^2}\right)\right)_{i \in\mathcal{M}, l = 1,...,n}$ dominates the others. We use $\tau_i, i = 1,...,p_1$ to prevent the normalizing parameters from being $0$. We do not require that the design matrix has rank $\min(n, p)$ or that $p<n$. However, when these conditions are not satisfied, the sparsity of $\theta$, i.e., assumption 6, can be violated. Section \ref{EXPER} uses a numerical simulation to illustrate this problem.

\label{remark1}
\end{remark}

\section{Consistency and the Gaussian approximation theorem}
\label{CLT}

Throughout, we will use the  notations developed  in section \ref{prelimary}. For a chosen ridge parameter $\rho_n > 0$, define the classical ridge regression estimator $\widetilde{\theta}^\star$ and the de-biased estimator $\widetilde{\theta}$ as
\begin{equation}
\begin{aligned}
\widetilde{\theta}^{\star} = (X^TX + \rho_n I_p)^{-1}X^Ty\\
\widetilde{\theta} = (\widetilde{\theta}_1,...,\widetilde{\theta}_p)^T = \widetilde{\theta}^\star + \rho_n\times Q(\Lambda^2 + \rho_nI_r)^{-1}Q^T\widetilde{\theta}^\star
\end{aligned}
\label{MainStat}
\end{equation}
Then we have
\begin{equation}
\begin{aligned}
\widetilde{\theta} - \theta = - \rho_n^2Q(\Lambda^2 + \rho_nI_r)^{-2}\zeta + Q\left((\Lambda^2 + \rho_nI_r)^{-1}\Lambda+ \rho_n(\Lambda^2 + \rho_nI_r)^{-2}\Lambda\right)P^T\epsilon
\end{aligned}
\label{EXPAN}
\end{equation}

Similar to $\mathcal{N}_{b_n}$, define the set $\widehat{\mathcal{N}}_{b_n}$, the estimator $\widehat{\theta} = (\widehat{\theta}_1,...,\widehat{\theta}_p)^T$ and $\widehat{\gamma}$ as
\begin{equation}
\widehat{\mathcal{N}}_{b_n} = \left\{i\ \Big|\vert\widetilde{\theta}_i\vert > b_n\right\},\ \ \widehat{\theta}_i = \widetilde{\theta}_i \times \mathbf{1}_{i\in\widehat{\mathcal{N}}_{b_n}},\ \ \widehat{\gamma} = M\widehat{\theta}
\label{HAT}
\end{equation}
Then, $\widehat{\theta}$ and $\widehat{\gamma}$ constitute
 the improved, i.e., debiased and thresholded, ridge regression estimator for the parameter vector $\theta$ and $\gamma = M\beta$ respectively. Apart from parameter estimation, we need to estimate the error variance $\sigma^2 = \mathbf{E}\epsilon^2_1$. The estimator for $\sigma^2$ is \begin{equation}
\widehat{\sigma}^2 = \frac{1}{n}\sum_{i=1}^n (y_i - \sum_{j=1}^p x_{ij}\widehat{\theta}_j)^2
\label{defSigma}
\end{equation}

\begin{theorem}
\label{thm1}
1. Suppose assumptions 1 to 5 hold true. Then
\begin{equation}
Prob\left(\widehat{\mathcal{N}}_{b_n}\neq \mathcal{N}_{b_n}\right) = O(n^{\alpha_p + m\nu_b - m\eta})\ \ \text{and  } \max_{i=1,2,...,p_1}\vert\widehat{\gamma}_i - \gamma_i\vert  = O_p(\vert\mathcal{M}\vert^{1/m}\times n^{-\eta})
\label{SELConst}
\end{equation}
where
$\gamma_i, i = 1,...,p_1$ and $\mathcal{N}_{b_n}$ are defined in section \ref{prelimary}.

\noindent 2. Suppose assumptions 1 to 6 hold true.  Then
\begin{equation}
\vert\widehat{\sigma}^2 - \sigma^2\vert = O_p(n^{-\alpha_\sigma}) .
\end{equation}
\label{THM_CON}
\end{theorem}

\noindent
Define $\widehat{\tau}_i,\ i=1,2,...,p_1$ and $H(x), x\in\mathbf{R}$ as
\begin{equation}
\begin{aligned}
\widehat{\tau}_i = \sqrt{\sum_{k=1}^r \left(\sum_{j\in\widehat{\mathcal{N}}_{b_n}}m_{ij}q_{jk}\right)^2 \times \left(\frac{\lambda_k}{\lambda_k^2 + \rho_n} + \frac{\rho_n\lambda_k}{(\lambda_k^2 + \rho_n)^2}\right)^2 + \frac{1}{n}}\\
H(x) = Prob\left(\max_{i\in\mathcal{M}}\frac{1}{\tau_i}\vert\sum_{k=1}^r c_{ik}\left(\frac{\lambda_k}{\lambda_k^2 + \rho_n} + \frac{\rho_n\lambda_k}{(\lambda_k^2 + \rho_n)^2}\right)\xi_k\vert\leq x\right)\\
\text{Here $\xi_k,\ k=1,2,...,r$ are independent normal random variables with mean $0$ and variance $\sigma^2 = \mathbf{E}\epsilon_1^2$. }
\end{aligned}
\label{estTau}
\end{equation}

\noindent
$\vert\mathcal{M}\vert$ (defined in \eqref{def_cik}) and $p_1$ may grow as the sample size increases. In this case, the estimator $\max_{i = 1,2,..., p_1}\frac{\vert\widehat{\gamma}_i - \gamma_i\vert}{\widehat{\tau}_i}$ does not have an asymptotic distribution. However, the cumulative distribution function of $\max_{i = 1,2,..., p_1}\frac{\vert\widehat{\gamma}_i - \gamma_i\vert}{\widehat{\tau}_i}$ still can be approximated by $H(x)$ (whose expression changes as the sample size increases as well). Define $c_{1-\alpha}$ as the $1-\alpha$ quantile of $H$; theorem \ref{Thm2}
 implies that the set
\begin{equation}
\left\{\gamma = (\gamma_1,...,\gamma_{p_1})\Big| \max_{i = 1,...,p_1}\frac{\vert\widehat{\gamma}_i - \gamma_i\vert}{\widehat{\tau}_i}\leq c_{1-\alpha}\right\}
\end{equation}
is an asymptotically valid
$(1-\alpha)\times100$\% confidence region for the parameter of interest $\gamma$.

\begin{theorem}
Suppose assumptions 1 to 7 hold true. Then
\begin{equation}
\lim_{n\to\infty}\sup_{x\geq 0}\vert Prob\left(\max_{i=1,2,...,p_1}\frac{\vert\widehat{\gamma}_i - \gamma_i\vert}{\widehat{\tau}_i}\leq x\right) - H(x)\vert = 0
\label{eqCLT}
\end{equation}
where
$\gamma_i, i = 1,...,p_1$ are defined in section \ref{prelimary}.
\label{Thm2}
\end{theorem}

Gaussian approximation theorems like theorem \ref{Thm2} are useful tools not only in linear models but also in other high dimensional statistics; e.g., Chernozhukov et al. \cite{chernozhukov2013} and Zhang and Wu \cite{zhang2017}.

\section{Bootstrap inference and hypothesis testing}
\label{BOOT}
An obstacle for constructing a practical confidence region or testing a hypothesis via theorem \ref{Thm2} are the unknown $\mathcal{M}$, $\mathcal{N}_{b_n}$, and $\sigma$. Besides, $H$ is too complicated to have a closed-form formula. Fortunately, statisticians can simulate normal random variables on a computer, so they may use Monte-Carlo simulations to find the $1-\alpha$ quantile of $H$. Based on this idea, this section develops a wild bootstrap algorithm similar to Mammen \cite{10.1214/aos/1176349025} and Chernozhukov et al. \cite{chernozhukov2013} for the following tasks: constructing the confidence region for the parameter of interest $\gamma = M\beta$; and testing the null hypothesis $\gamma = \gamma_0$ (for a known $\gamma_0$) versus the alternative hypothesis $\gamma\neq \gamma_0$. Similar to Zhang and Cheng \cite{doi:10.1080/01621459.2016.1166114}, Chernozhukov et al. \cite{chernozhukov2013}, and Zhang and Wu \cite{zhang2017}, we use the maximum
statistic $\max_{i=1,2,...,p_1}\frac{\vert\widehat{\gamma}_i - \gamma_i\vert}{\widehat{\tau}_i}$ to construct a simultaneous confidence region.

\begin{algorithm}[Wild bootstrap inference and hypothesis testing]
	
\textbf{Input: } Design matrix $X$, dependent variables $y = X\beta + \epsilon$, linear combination matrix $M$, ridge parameter $\rho_n$, threshold $b_n$, nominal coverage probability $1-\alpha$, number of bootstrap replicates $B$

\noindent\textbf{Additional input for testing: } $\gamma_0 = (\gamma_{0,1},...,\gamma_{0, p_1})^T$

\noindent 1. Calculate $\widehat{\theta},\ \widehat{\gamma} = (\widehat{\gamma}_1,...,\widehat{\gamma}_{p_1})^T$ defined in \eqref{HAT}, $\widehat{\tau}_i,\ i=1,2,...,p_1$ defined in \eqref{estTau}, and $\widehat{\sigma}$ defined in \eqref{defSigma}.

\noindent 2. Generate i.i.d. errors $\epsilon^* = (\epsilon^*_1,...,\epsilon^*_n)^T$ with $\epsilon^*_i, i = 1,...,n$ having normal distribution with mean $0$ and variance $\widehat{\sigma}^2$, then calculate $y^* = X\widehat{\theta} + \epsilon^*$ and $\widehat{\theta}_{\perp} = Q_\perp Q_\perp^T\widehat{\theta}$ ($Q_\perp$ is defined in section \ref{prelimary}).

\noindent 3. Calculate $\widetilde{\theta}^{\star*} = (X^TX + \rho_n I_p)^{-1}X^Ty^*$ and $\widetilde{\theta}^* = (\widetilde{\theta}^* _1,...,\widetilde{\theta}^* _p)^T=\widetilde{\theta}^{\star *} + \rho_n\times Q(\Lambda^2 + \rho_nI_r)^{-1}Q^T\widetilde{\theta}^{\star*} + \widehat{\theta}_{\perp}$.

\noindent 4. Calculate $\widehat{\mathcal{N}}_{b_n}^* = \left\{i\Big| \vert\widetilde{\theta}_i^*\vert > b_n\right\}$ and $\widehat{\theta}^* = (\widehat{\theta}_1^*,...,\widehat{\theta}_p^*)^T$ with $\widehat{\theta}_i^* = \widetilde{\theta}_i^*\times \mathbf{1}_{i\in\widehat{\mathcal{N}}^*_{b_n}}$ for $i=1,2,...,p$.

\noindent 5. Calculate $\widehat{\gamma}^* = M\widehat{\theta}^*$, $\widehat{\tau}_i^*, i = 1,2,...,p_1$, and $E_b^*$ such that
\begin{equation}
\widehat{\tau}_i^* = \sqrt{\sum_{k=1}^r \left(\sum_{j\in\widehat{\mathcal{N}}^*_{b_n}}m_{ij}q_{jk}\right)^2 \times \left(\frac{\lambda_k}{\lambda_k^2 + \rho_n} + \frac{\rho_n\lambda_k}{(\lambda_k^2 + \rho_n)^2}\right)^2 + \frac{1}{n}},\ \  E_b^* = \max_{i=1,2,...,p_1}\frac{\vert\widehat{\gamma}^*_i - \widehat{\gamma}_i\vert}{\widehat{\tau}_i^*}
\end{equation}

\noindent 6.a (For constructing a confidence region) Repeat steps 2 to 5 for $B$ times to generate $E^*_b,\ b=1,2,...,B$; then calculate the $1-\alpha$ sample quantile $C^*_{1-\alpha}$ of $E^*_b$. The $1-\alpha$ confidence region for the parameter of interest $\gamma = M\beta$ is given by the set
\begin{equation}
\left\{\gamma = (\gamma_1,...,\gamma_{p_1})^T\Big|\max_{i=1,2,...,p_1}\frac{\vert\widehat{\gamma}_i - \gamma_i\vert}{\widehat{\tau}_i}\leq C^*_{1-\alpha}\right\}
\label{confRegion}
\end{equation}

\noindent 6.b (For hypothesis testing) Repeat steps 2 to 5 for $B$ times to generate $E^*_b,\ b=1,2,...,B$; then calculate the $1-\alpha$ sample quantile $C^*_{1-\alpha}$ of $E^*_b$. Reject the null hypothesis  $\gamma = \gamma_0$  when
\begin{equation}
	\max_{i = 1,2,...,p_1}\frac{\vert \widehat{\gamma}_i - \gamma_{0, i}\vert}{\widehat{\tau}_i} > C^*_{1-\alpha} .
\end{equation}
\label{BootALG1}
\end{algorithm}
\noindent
As in section \ref{prelimary}, if $X$ has rank $p \leq n$, we define $\widehat{\theta}_\perp = 0$, the $p$ dimensional vector with all elements $0$.

According to theorem 1.2.1 in Politis et al. \cite{subsampling}, the consistency of algorithm \ref{BootALG1} --either for asymptotic validity of confidence regions or
consistency of the hypothesis test-- is ensured if
\begin{equation}
Prob\left(\max_{i=1,2,...,p_1}\frac{\vert\widehat{\gamma}_i - \gamma_i\vert}{\widehat{\tau}_i}\leq c^*_{1-\alpha}\right)\to 1-\alpha
\label{EstSec}
\end{equation}
where $c^*_{1-\alpha}$ is the $1-\alpha$ quantile of the conditional distribution $Prob^*\left(\max_{i=1,2,...,p_1}\frac{\vert\widehat{\gamma}^*_i - \widehat{\gamma}_i\vert}{\widehat{\tau}_i^*}\leq x\right)$; we
  prove this in theorem \ref{BootCons} below.
\begin{theorem}
Suppose assumptions 1 to 7 hold true. Then
\begin{equation}
\sup_{x\geq 0}\vert Prob^*\left(\max_{i=1,2,...,p_1}\frac{\vert\widehat{\gamma}^*_i - \widehat{\gamma}_i\vert}{\widehat{\tau}_i^*}\leq x\right) - H(x)\vert = o_P(1) .
\label{EstFir}
\end{equation}
In addition, for any given $0< \alpha < 1$,
(\ref{EstSec}) holds true.
\label{BootCons}
\end{theorem}
Theorem \ref{BootCons} has two implications. On the one hand, the confidence region introduced in
step 6.a of  algorithm \ref{BootALG1} is asymptotically valid, i.e., its
coverage tends to $1-\alpha$.
 On the other hand, consider the hypothesis test of step 6.b of  algorithm \ref{BootALG1};
Theorem \ref{BootCons} implies that,
 if the null hypothesis is true, them the
 probability for incorrectly rejecting the null hypothesis is asymptotically $\alpha$,
i.e., the test is consistent.

\section{Bootstrap interval prediction}
\label{BOOTPRD}
Given our data from the linear model $y = X\beta + \epsilon$, consider  a new $p_1\times p$ 
regressor matrix $X_f$, i.e., a collection of regressor (column) vectors that happen to be of   interest; as with $X$ itself, $X_f$ is  assumed given, i.e.,  deterministic.
The prediction  problem involves (a) finding a predictor for the {\it future} (still unobserved) vector $y_f = X_f\beta + \epsilon_f$, and (b) finding a $1-\alpha$ prediction region $A\subset R^{p_1}$ so that $Prob(y_f\in A) \to 1- \alpha$ as the (original) sample size $n\to\infty$. Here $\epsilon_f = (\epsilon_{f,1},...,\epsilon_{f,p_1})^T$ are i.i.d. errors with the same marginal distribution as $\epsilon_1$, and $\epsilon_f$ is independent with $\epsilon$. 

The  predictor of $y_f $ that is optimal with respect to total mean squared error is $X_f\beta$; since $\beta$ is unknown, we can estimate  it by
 $\widehat{\theta}$ as in \eqref{HAT}, yielding the practical predictor  $X_f\widehat{\theta}$. However, finding a $1-\alpha$ prediction region  for  $y_f $ is more challenging.
We adopt  definition 2.4.1  of Politis \cite{model_free}, and define a consistent prediction region  in terms of conditional coverage as follows.
\begin{definition}[Consistent prediction region]
 A set $\Gamma= \Gamma(X, y, X_f)$ is called a $1-\alpha$ consistent prediction region for the future observation $y_f = X_f\beta + \epsilon_f$ if
\begin{equation}
Prob\left(y_f\in \Gamma|y\right)\to_p 1-\alpha\ \ \text{as   }n\to\infty .
\label{eq.DEF_PRD}
\end{equation}
\label{DEF_PRD}
\end{definition}
\noindent Note that the convergence (\ref{DEF_PRD}) is ``in probability"
since $Prob\left(y_f\in \Gamma|y\right)$ is a function of $y$, and therefore random;
see also Lei and Wasserman \cite{10.2307/24772746} for more on the notion of conditional validity.

Other authors, including
Stine \cite{doi:10.1080/01621459.1985.10478220},
Romano et al. \cite{NEURIPS2019_5103c358}, and Chernozhukov et al. \cite{chernozhukov2019distributional},
considered another definition of prediction interval consistency focusing on unconditional coverage,
 i.e., insisting that
\begin{equation}
Prob(y_f\in\Gamma)\to 1-\alpha .
\label{eq.DEF_PRD2}
\end{equation}

However, the conditional coverage of definition \ref{DEF_PRD} is a stronger
property. To see why, define the random variables $U_n=Prob\left(y_f\in \Gamma|y\right)$, noting
that $y$ has dimension $n$. Then, the boundedness of $U_n$  can be invoked
to show that if $U_n\to_p  1-\alpha$, then   $\mathbf{E} U_n\to  1-\alpha$ as well.
Hence, (\ref{eq.DEF_PRD}) implies  (\ref{eq.DEF_PRD2}); see
  Zhang and Politis \cite{zhang2021bootstrap} for a further discussion
on conditional vs.~unconditional coverage.

Consider the prediction error $y_f - X_f\widehat{\theta} = \epsilon_f - X_f(\widehat{\theta} - \beta)$.
If we can put  bounds on the prediction error that are valid with conditional probability
$1-\alpha$ (asymptotically), then a consistent prediction region ensues.
Note that the prediction error has two parts: $\epsilon_f $ and $- X_f(\widehat{\theta} - \beta)$.
Although the latter may be asymptotically negligible, it is important in practice to not
approximate it by zero   as it would yield   finite-sample undercoverage; see e.g.
Ch. 3 of   Politis \cite{model_free} for an extensive discussion.

 Theorem \ref{Thm2} indicates that the asymptotically negligible  estimation error can be approximated by normal random variables. On the other hand, the non-negligible error $\epsilon_f $ may not have a normal distribution;
 so in order to approximate the distribution of $\epsilon_f - X_f(\widehat{\theta} - \beta)$,
we need to estimate the errors' marginal distribution as well.

This section requires some additional assumptions.

\textbf{Additional assumptions}

8. The cumulative distribution function of errors $F(x) = Prob\left(\epsilon_1\leq x\right)$ is continuous

9. The number of regressors of interest 
is bounded, i.e., $p_1 = O(1)$
\\

\noindent
Since $F$ is increasing and bounded, if $F(x)$ is continuous, then $F$ is uniformly continuous on $\mathbf{R}$. this property is useful in the proof of lemma \ref{LEMMARES}.

\begin{lemma}
Suppose  assumption 1 to 6 and 8 hold true. Define the   
residuals $\widehat{\epsilon}^{'} = (\widehat{\epsilon}_1^{'},...,\widehat{\epsilon}_n^{'})^T = y - X\widehat{\theta}$, as well as the centered residuals $\widehat{\epsilon} = (\widehat{\epsilon}_1,...,\widehat{\epsilon}_n)^T$ with $\widehat{\epsilon}_i = \widehat{\epsilon}_i^{'} - \frac{1}{n}\sum_{i=1}^n\widehat{\epsilon}_i^{'}$.
If we let $\widehat{F}(x) = \frac{1}{n}\sum_{i = 1}^n\mathbf{1}_{\widehat{\epsilon}_i\leq x}$, then
\begin{equation}
\sup_{x\in\mathbf{R}}\vert\widehat{F}(x) - F(x)\vert\to_p 0\ \text{as $n\to\infty$. }\
\label{PrdEq}
\end{equation}
\label{LEMMARES}
\end{lemma}
\noindent   We   emphasize that the dimension of parameters $p$ in lemma \ref{LEMMARES} can grow to infinity as long as assumption 4 is satisfied.
Furthermore,   the validity of lemma \ref{LEMMARES}
--as well as that of theorem \ref{THMPD} that follows-- does not require assumption 7.

We will resample the centered residuals $\widehat{\epsilon}_i, i = 1,2,...,n$ (in other words, generate random variables with distribution $\widehat{F}$) in algorithm \ref{BootPrd}. Lemma \ref{LEMMARES}
will ensure  that the centered residuals can capture the
distribution of the non-negligible errors.

For a high dimensional linear model, lemma \ref{LEMMARES} is not an obvious result; see Mammen \cite{10.1214/aos/1033066211} for a  detailed explanation. 
Lemma \ref{LEMMARES} is the foundation for a   new   resampling procedure as
follows; this is a    hybrid bootstrap as it combines the   residual-based bootstrap to replicate the new error
$\epsilon_f $ with the normal approximation to the estimation error $- X_f(\widehat{\theta} - \beta)$.

\begin{algorithm} [Hybrid bootstrap for prediction region]
\textbf{Input: } Design matrix $X$, dependent variables $y = X\beta + \epsilon$, a new $p_1\times p$ linear combination matrix $X_f$, ridge parameter $\rho_n$, threshold $b_n$, nominal coverage probability $0<1-\alpha<1$, the number of bootstrap replicates $B$

1. Calculate $\widehat{\theta}$ defined in \eqref{HAT}, $\widehat{\sigma}$ defined in \eqref{defSigma}, $\widehat{\epsilon}$ defined in lemma \ref{LEMMARES}, $\widehat{y}_f = (\widehat{y}_{f,1},...,\widehat{y}_{f,p_1})^T= X_f\widehat{\theta}$, and $\widehat{\theta}_{\perp} = Q_\perp Q_\perp^T\widehat{\theta}$.

2. Generate i.i.d. errors $\epsilon^* = (\epsilon^*_1,...,\epsilon^*_n)^T$ with $\epsilon^*_i, i = 1,...,n$ having normal distribution with mean $0$ and variance $\widehat{\sigma}^2$. Then generate i.i.d. errors $\epsilon^*_f = (\epsilon^*_{f,1},...,\epsilon^*_{f,p_1})^T$ with $\epsilon_{f,i}, i = 1,...,p_1$ having cumulative distribution function $\widehat{F}$ defined in lemma \ref{LEMMARES}. Calculate $y^* = X\widehat{\theta} + \epsilon^*$.

3. Calculate $\widetilde{\theta}^{\star*} = (X^TX + \rho_n I_p)^{-1}X^Ty^*$ and $\widetilde{\theta}^* = \widetilde{\theta}^{\star *} + \rho_n\times Q(\Lambda^2 + \rho_nI_r)^{-1}Q^T\widetilde{\theta}^{\star*} + \widehat{\theta}_{\perp}$. Then derive $\widehat{\mathcal{N}}_{b_n}^* = \left\{i\Big|\vert\widetilde{\theta}_i^*\vert > b_n\right\}$, $\widehat{\theta}^* = (\widehat{\theta}_1^*,...,\widehat{\theta}_p^*)^T$ with $\widehat{\theta}_i^* = \widetilde{\theta}_i^*\times \mathbf{1}_{i\in\widehat{\mathcal{N}}^*_{b_n}}$ for $i=1,2,...,p$.

4. Calculate $y^*_f = (y^*_{f,1},...,y^*_{f,p_1})^T= X_f\widehat{\theta} + \epsilon^*_f$ and $\widehat{y}^*_f = (\widehat{y}^*_{f,1},...,\widehat{y}^*_{f,p_1})^T = X_f\widehat{\theta}^*$. Define $E^*_b = \max_{i=1,2,...,p_1}\vert y^*_{f,i} - \widehat{y}^*_{f,i}\vert$.

5. Repeat steps 2 to 4 for $B$ times, and generate $E^*_b,\ b=1,2,...,B$. Calculate the $1-\alpha$ sample quantile $C^*_{1-\alpha}$ of $E^*_b$. Then, the $1-\alpha$ prediction region for $y_f = X_f\beta + \epsilon_f$ is given by
\begin{equation}
\left\{y_f = (y_{f,1},...,y_{f,p_1})^T\Big|\max_{i=1,2,...,p_1}\vert y_{f,i} - \widehat{y}_{f,i}\vert\leq C^*_{1-\alpha}\right\} .
\label{prdRegion}
\end{equation}
\label{BootPrd}
\end{algorithm}

\noindent   If the design matrix $X$ has rank $p$, then $\widehat{\theta}_\perp$ is defined to be $0$.

Similar to section \ref{BOOT}, here we define $c^*_{1-\alpha}$ as the $1-\alpha$ quantile of the conditional distribution $Prob^*\left(\max_{i = 1,...,p_1}\vert y^*_{f,i} - \widehat{y}^*_{f,i}\vert\leq x\right)$, which can be approximated by $C^*_{1-\alpha}$ by letting $B\to\infty$.
Theorem \ref{THMPD} below proves $Prob\left(\max_{i=1,2,...,p_1}\vert y_{f,i} - \widehat{y}_{f,i}\vert\leq c^*_{1-\alpha}\right)\to 1-\alpha$ as the sample size $n\to\infty$, which justifies the consistency of the prediction region
(\ref{prdRegion}).  

\begin{theorem}
Suppose assumptions 1 to 6 and 8 to 9 hold true(here consider $M = (m_{ij})_{i = 1,...,p_1, j = 1,...,p}$ in assumption 5 as $X_f$). Then
\begin{equation}
\sup_{x\geq 0}\vert Prob^*\left(\max_{i=1,2,...,p_1}\vert y^*_{f,i} - \widehat{y}^*_{f,i}\vert\leq x\right) - Prob^*\left(\max_{i=1,2,...,p_1}\vert y_{f,i} - \widehat{y}_{f,i}\vert\leq x\right)\vert = o_p(1) .
\label{PDFIR}
\end{equation}
For any fixed $0< \alpha < 1$, it follows that
\begin{equation}
Prob^*\left(\max_{i=1,2,...,p_1}\vert y_{f,i} - \widehat{y}_{f,i}\vert\leq c^*_{1-\alpha}\right)\to_p 1-\alpha\ \text{as }n\to\infty.
\label{PDFCI}
\end{equation}

\label{THMPD}
\end{theorem}

\noindent Note that the   bootstrap probability  $Prob^*(\cdot)$  is  probability conditional on
the data $y$, thus justifying the notion of conditional validity of our definition \ref{DEF_PRD}.

 A version of the algorithm \ref{BootPrd} can be constructed where the residual-based bootstrap part is
conducted by resampling from the empirical distribution of the (centered)  predictive, i.e., leave-one-out,
residuals instead of the fitted residuals $\widehat{\epsilon}_i$; see
Ch. 3 of   Politis \cite{model_free} for a discussion.

\section{Numerical Simulations}
\label{EXPER}
\label{Numerical}
We define $k_n = \sqrt{n\log(n)}$ and the following four terms
\begin{equation}
\begin{aligned}
\mathcal{K}_1 = \max_{i=1,2,...,p_1}k_n\vert\sum_{j\not\in\mathcal{N}_{b_n}}m_{ij}\theta_j\vert,\ \ \ \mathcal{K}_2 = \max_{i=1,2,...,p_1}k_n\vert\sum_{j=1}^rm_{ij}\theta_{\perp,j}\vert, \ \ \
\mathcal{K}_3 = b_n\sum_{j\not\in\mathcal{N}_{b_n}}\vert\theta_j\vert,\ \ \ \mathcal{K}_4 = \frac{\sqrt{\vert\mathcal{N}_{b_n}\vert}}{\lambda_r} ;
\end{aligned}
\end{equation}
see section \ref{prelimary} for the meaning of notations in the above. Assumptions 5 and 6 imply that these terms  converge to 0 as the sample size $n\to\infty$. Indeed, if one of the  $\mathcal{K}_i$ is large, the debiased and threshold ridge regression estimator may have a large bias, which affects the performance of the bootstrap algorithms.

 In this section, we generate the design matrix $X$, the linear combination matrix $M$, and the parameters $\beta$ through the following strategies:
\\
\noindent \textbf{Design matrix $X$: } define $X = (x_1^T,...,x^T_n)^T$ with $x_i = (x_{i1},...,x_{ip})^T\in\mathbf{R}^p, i = 1,...,n$. Generate   $x_1, x_2, \ldots$ as
i.i.d.~normal 
random vectors with mean $0$ and covariance matrix $\Sigma\in\mathbf{R}^{p\times p}$.
We choose $\Sigma$ with diagonal  elements equal to $2.0$ and off-diagonal equal to $0.5$.

\begin{equation}
M = \left[
\begin{matrix}
m_{11} & m_{12} & ... & m_{1\tau} & m_{1\tau+1} & ... & m_{1p}\\
m_{21} & m_{22} & ... & m_{2\tau} & m_{2\tau+1} & ... & m_{2p}\\
\vdots & \vdots & ... & \vdots    & \vdots      & ... & \vdots\\
m_{\vert\mathcal{M}\vert1} & m_{\vert\mathcal{M}\vert2} & ... & m_{\vert\mathcal{M}\vert\tau} & m_{\vert\mathcal{M}\vert\tau+1} & ... & m_{\vert\mathcal{M}\vert p}\\
0 & 0 & ... & 0 &  m_{\vert\mathcal{M}\vert+1\tau+1} & ... & m_{\vert\mathcal{M}\vert+1 p}\\
\vdots & \vdots & ... & \vdots    & \vdots      & ... & \vdots\\
0 & 0 & ... & 0 &  m_{p_1\tau+1} & ... & m_{p_1 p}
\end{matrix}
\right]
\label{defM}
\end{equation}

\noindent \textbf{$M$ and $\beta$ when $p<n$: } choose $\tau = 50$ in \eqref{defM}. Generate $m_{ij}^{'}, i = 1,2,..., \vert\mathcal{M}\vert,\ j = 1,2,...,\tau$
as i.i.d.~normal 
with mean $0.5$ and variance $1.0$, and generate $m_{ij}^{'}, i = 1,2,..., p_1, j = \tau + 1,...,p$ as i.i.d.~normal  
with mean $1.0$ and variance $4.0$. Use $m_{ij} = 2.0\times m_{ij}^{'}/\sqrt{\sum_{j = 1}^\tau m^{'2}_{ij}}$ for $i = 1,2,..., \vert\mathcal{M}\vert,\ j = 1,2,...,\tau$; $m_{ij} = 4.0\times m^{'}_{ij} / \sqrt{\sum_{j = \tau+1}^p m^{'2}_{ij}}$ for $i = 1,2,...,\vert\mathcal{M}\vert,\ j = \tau+1,...,p$; and $m_{ij} = 6.0\times m^{'}_{ij} / \sqrt{\sum_{j = \tau+1}^p m^{'2}_{ij}}$ for $i = \vert\mathcal{M}\vert + 1,...,p_1, j = \tau + 1,...,p$. Choose $\beta = (\beta_1,...,\beta_p)^T$ with $\beta_i = 2.0, i = 1, 2, 3$, $\beta_i = -2.0, i = 4,5,6$, $\beta_i = 1.0, i = 7,8,9$, $\beta_i = -1.0, i = 10, 11, 12$, $\beta_i = 0.01, i = 13, 14, 15, 16$, and $0$ otherwise.

\noindent \textbf{$M$ and $\beta$ when $p>n$: } choose $\tau = 6$ in \eqref{defM}. Generate $m_{ij}^{'}, i = 1,2,...,\vert\mathcal{M}\vert,\ j = 1,2,...,\tau$ as i.i.d.~normal 
with mean $0.5$ and variance $1.0$, and generate $m_{ij}^{'}, i = 1,2,...,p_1,\ j = \tau+1,...,p$  as i.i.d.~normal  
with mean $1.0$ and variance $4.0$. Use $m_{ij} = 2.0 \times m_{ij}^{'} / \sqrt{\sum_{j = 1}^\tau m^{'2}_{ij}}$ for $i = 1,2,...,\vert\mathcal{M}\vert,\ j = 1,2,...,\tau$; and $m_{ij} = m^{'}_{ij} / \sqrt{\sum_{j = \tau+1}^p m^{'2}_{ij}}$ for $i = 1,2,...,p_1,\ j = \tau+1,...,p$. Choose $\beta_i = 1.0, i = 1,2,3$, $\beta_i = -1.0, i = 4,5,6$, and $0$ otherwise. When $p > n$, $\beta$ may not be identifiable \cite{shao2012}, and $\beta$ may not equal $\theta$(defined in section \ref{prelimary}) despite $X\beta = X\theta$. We consider both situations and evaluate the performance of proposed methods on the linear model $y = X\beta+\epsilon$ and $y = X\theta+\epsilon$. We fix $X$ and $M$ in each simulation.

The different regression algorithms considered are the debiased and threshold ridge regression(Deb Thr), the ridge regression, Lasso, threshold ridge regression (Thr Ridge), threshold Lasso (Thr Lasso), and the post-selection algorithms, i.e., Lasso + OLS (Post OLS), and Lasso + Ridge (Post Ridge).
We consider 6 cases for simulation involving a different $p/n$ ratio, and Normal vs.~Laplace (2-sided
exponential) errors;
we present detailed information about each simulation case in table \ref{INFOFD}, compare the performance of different regression algorithms in figure \ref{Figure1}, and record the performance of bootstrap algorithms on estimation/hypothesis testing and interval-prediction in table \ref{TAB1}. The optimal ridge parameter $\rho_n$ and threshold $b_n$ are chosen by 5-fold cross validation. To adapt to assumption 9, we choose $X_f$ as the first 100 lines of $M$ for prediction.

\begin{table}[htbp]
  \centering
  \caption{Information about $X$, $M$ and $\epsilon$ in each simulation case. For the normal distribution we choose variance $4$, for the Laplace distribution we choose the scale $\sqrt{2}$. By doing this, the variance of residuals is $4$. When $p > n$, $\beta\neq \theta$. \textbf{The left(right) side of the slashes represent $\mathcal{K}_2$ calculated by the linear model $y = X\beta + \epsilon$($y = X\theta + \epsilon$).} The difference between $\beta$ and $\theta$ does not change other terms in case 5 and 6.}
  \begin{tabular}{l l l l l l l l l l l l l}
  \hline\hline
  Case            & $n$    & $p$       & Residual      & $p_1$ & $\vert\mathcal{M}\vert$ & $\lambda_r$ & $\rho_n$ & $b_n$  & $\mathcal{K}_1$ & $\mathcal{K}_2$ & $\mathcal{K}_3$ & $\mathcal{K}_4$\\
  1               & 1000   & 500       & Normal        & 800   &  300                    & 12.978      & 56.453   & 0.343  & 1.370           & 0.0           & 0.013 & 1.712\\
  2               & 1000   & 500       & Laplace       & 800   &  300                    & 12.561      & 36.728   & 0.354  & 1.636           & 0.0           & 0.014 & 1.769\\
  3               & 1000   & 650       & Laplace       & 800   &  300                    & 8.226       & 56.432   & 0.396  & 1.553           & 0.0 & 0.016 & 3.085\\
  4               & 1000   & 500       & Laplace       & 800   &  700                    & 12.847      & 55.317   & 0.346  & 1.510           & 0.0 & 0.014 & 1.730\\

  5               & 1000   & 1500      & Normal        & 800   &  300                    & 9.766       & 1.201    & 0.228  & 6.938           & 129 / 0.0 &8.214 & 3.962\\
  6               & 1000   & 1500      & Laplace       & 800   &  300                    & 9.766       & 1.201    & 0.228  & 6.938           & 129 / 0.0 & 8.214 & 3.962\\
  \hline\hline
  \end{tabular}
  \label{INFOFD}
\end{table}
\noindent Case 5 and 6 consider both the linear model $y = X\beta + \epsilon$ and $y = X\theta + \epsilon$, here $\beta \neq \theta = QQ^T\beta$. The difference in $\beta$ and $\theta$ affects the value of $\mathcal{K}_2$(but does not affect others), so we have two values in table \ref{INFOFD}.

\label{EXP1}

\begin{figure}[htbp]
    \flushleft
    \subfigure[Case 1]{
    \includegraphics[width = 2.0in]{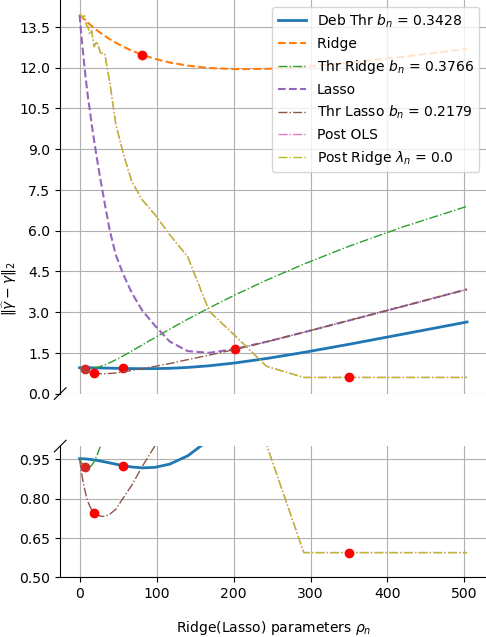}
    \label{1case}
  }
  \subfigure[Case 2]{
    \includegraphics[width = 2.0in]{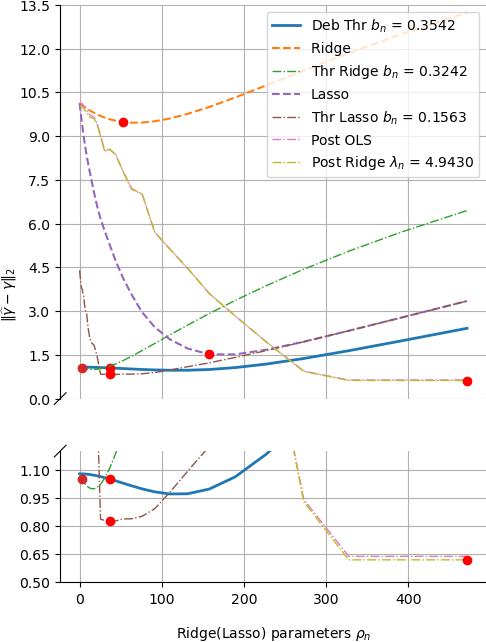}
    \label{2case}
  }
  \subfigure[Case 3]{
    \includegraphics[width = 2.0in]{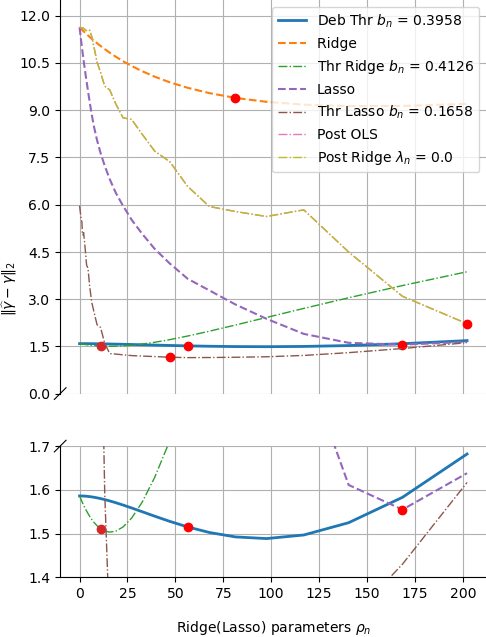}
    \label{2case}
  }
  \subfigure[Case 4]{
    \includegraphics[width = 2.0in]{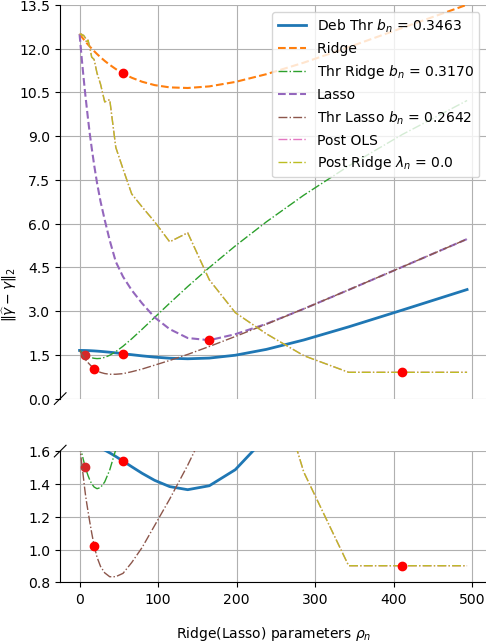}
    \label{2case}
  }
  \subfigure[Case 5(use $\beta$)]{
    \includegraphics[width = 2.0in]{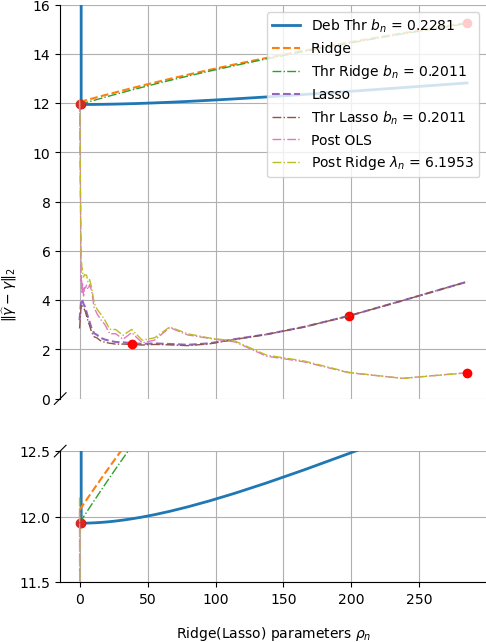}
    \label{2case}
  }
  \subfigure[Case 5(use $\theta$)]{
    \includegraphics[width = 2.0in]{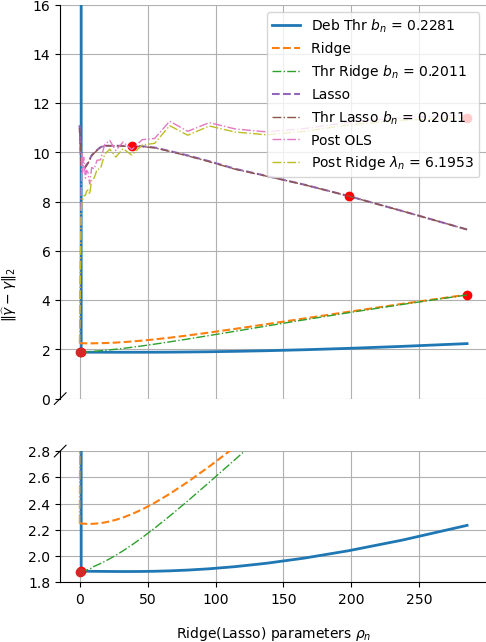}
    \label{2case}
  }
  \subfigure[Case 6(use $\beta$)]{
    \includegraphics[width = 2.0in]{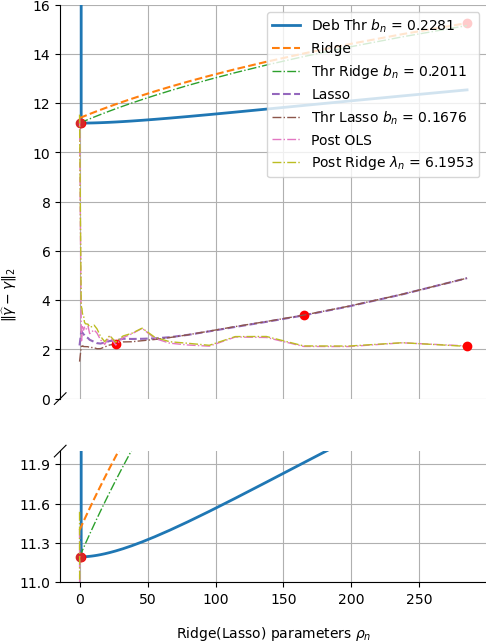}
    \label{2case}
  }
  \subfigure[Case 6(use $\theta$)]{
    \includegraphics[width = 2.0in]{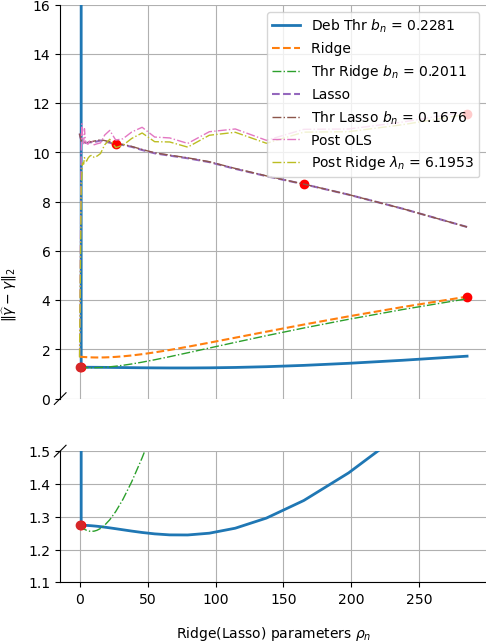}
    \label{2case}
  }
  \caption{Estimation performance of various linear regression methods over the different cases. 'Deb' abbreviates 'Debiased', 'Thr' abbreviates 'Threshold', 'Post' abbreviates 'Post-selection', and 'OLS' abbreviates 'ordinary least square'. Red dots represent the   ridge/Lasso parameters selected by 5-fold cross validation. The optimal  threshold   $b_n$ was also selected  by 5-fold cross validation.
 The vertical axis represents the Euclidean norm of $\widehat{\gamma} - \gamma$ where $\widehat{\gamma}$ is defined in \eqref{HAT}, and $\gamma$ is defined in Section \ref{prelimary}. The little graphs below each of the eight graphs shows a zoomed-in part of the
graph above it.  }
  \label{Figure1}
  \end{figure}

Figure \ref{Figure1} plots the Euclidean norm $\Vert\widehat{\gamma} - \gamma\Vert_2$,
with $\widehat{\gamma}$  defined in \eqref{HAT}, and $\gamma$   defined in Section \ref{prelimary}, for various linear regression methods. When the underlying linear model is sparse, thresholding decreases the ridge regression estimator's error(from around 10 to around 2 in our experiment). However, the performance of the threshold ridge regression method is sensitive to the ridge parameter $\rho_n$, i.e., $\Vert\widehat{\gamma} - \gamma\Vert_2$ can be significantly larger than its minimum despite $\rho_n$ is close to the minimizer of $\Vert\widehat{\gamma} - \gamma\Vert_2$.

In reality,   cross validation does not necessarily guarantee
 selection of the optimal $\rho_n$, so it is risky to use the threshold ridge regression method. Debiasing helps decrease the ridge regression estimator's error; more importantly, it is robust to changes in the
choice of $\rho_n$. Even if a cross validation selects a sub-optimal $\rho_n$, the error of the debiased and threshold ridge regression estimator does not surge, and the estimator's performance does not notably deteriorate. Therefore, we consider the debiased and threshold ridge regression as a practical method to handle real-life data.

Thresholding also helps improve the performance of Lasso, especially when the Lasso parameter is small. However, when the Lasso parameter becomes large, Lasso method already recovers the underlying sparsity of the linear model, and thresholding becomes unnecessary.

When the dimension of parameters $p$ is greater than the sample size $n$, both parameters $\beta$ and $\theta$(see section \ref{prelimary}) could be considered as the `parameters' for the linear model. Lasso methods estimate linear combinations of $\beta$, while ridge regression methods estimate linear combinations of $\theta$. Under this situation, the difference between $\beta$ and $\theta$ is the main factor for the estimators' error. In reality, statisticians cannot distinguish between $\beta$ and $\theta$ based on data. So they need to design which parameters to estimate a priori and select a suitable regression method (e.g., Lasso, ridge regression, or their variations) reflecting their preferences.

As a summary of Figure \ref{Figure1}, apart from having a closed-form formula, the debiased and threshold ridge regression has the smallest estimation error among all ridge regression variations, and has comparable performance to the threshold Lasso. Furthermore, it is not overly sensitive on changes
 in the ridge parameter $\rho_n$. Therefore, even when
  a sub-optimal $\rho_n$ is selected, the performance of the debiased and threshold ridge
regression is not severely affected. When $p>n$, this method (and other ridge regression methods) considers $\theta$ rather than $\beta$ to be the parameter  of the linear model. So, in this case, ridge regression methods are suitable if the underlying linear model is indeed $y = X\theta + \epsilon$(in other words, the projection does not have effect on the parameters of the linear model).

Table \ref{TAB1} records the average errors of the proposed statistics $\widehat{\gamma}$(defined in \eqref{HAT}), $\widehat{\sigma}^2$(defined in \eqref{defSigma}), and the coverage probability of the confidence region \eqref{confRegion} as well as the coverage probability of the prediction region \eqref{prdRegion}, in 1000 numerical simulations. We also record the frequency of model misspecification(i.e., $\widehat{\mathcal{N}}_{b_n}\neq \mathcal{N}_{b_n}$), $P(\widehat{\mathcal{N}}_{b_n}\neq \mathcal{N}_{b_n})$. When the sample size $n$ is greater than the dimension of parameters $p$, thresholding is likely to recover the sparsity of the parameters.
In all  these cases, i.e., Case 1--4, our confidence intervals achieve near-perfect coverage.
The slight under-coverage in prediction intervals is a well-known phenomenon;
see e.g. Ch. 3.7 of Politis \cite{model_free}.

 However, in cases 5 and 6 where $p>n$, $\theta$ is not necessarily sparse, and  model misspecification may happen. Notably, $\widehat{\gamma}$'s error in estimating linear combinations of $\theta$ does not surge even when $p>n$. However, the difference between $\beta$ and $\theta$ introduces a large bias to $\widehat{\gamma}$. Besides, when $p>n$, assumption 6 can be violated. Correspondingly the variance estimator $\widehat{\sigma}^2$ may have a large error. The difference between $\beta$ and $\theta$ invalidates the confidence region
 \eqref{confRegion}. For prediction region \eqref{prdRegion}, this problem still exists. However, the prediction region catches non-negligible errors apart from the asymptotically negligible errors and it is wider than the confidence region. Consequently, as long as the absolute values of difference are small, the prediction interval's performance will not be severely affected.

\begin{table}[htbp]
\centering
\caption{ Frequency of model misspecification; average errors of $\widehat{\gamma}$ and $\widehat{\sigma}^2$; and the coverage probability for the confidence region \eqref{confRegion} and the prediction region \eqref{prdRegion}. The nominal coverage probability is $1-\alpha = 95\%$. The overscore represents calculating the sample mean among $1000$ simulations. We choose the number of bootstrap replicates $B = 500$.}
\begin{tabular}{l l l l l| l}
\hline\hline
\multicolumn{5}{l}{Estimation and Confidence region construction} & \multicolumn{1}{l}{Prediction}\\
Case $\#$ & $P(\widehat{\mathcal{N}}_{b_n}\neq \mathcal{N}_{b_n})$ & $\overline{\max_{i = 1,2,...,p_1}\vert\widehat{\gamma}_i - \gamma_i\vert} $ & $\overline{\vert\widehat{\sigma}^2 - \sigma^2\vert}$ & coverage & coverage\\
1               & $0.0$                                                  & 0.185 & 0.144 & $95.4\%$  & $91.5\%$ \\
2               & $0.0$                                                  & 0.183 & 0.228 & $93.6\%$  & $90.4\%$ \\
3               & $0.0$                                                  & 0.209 & 0.232 & $95.9\%$  & $92.6\%$ \\
4               & $0.0$                                                  & 0.191 & 0.224 & $95.3\%$  & $90.6\%$ \\
5(use $\beta$)  & $0.241$                                                & 1.580 & 1.351 & $0.0\%$   & $97.2\%$ \\
5(use $\theta$) & $0.223$                                                & 0.259 & 1.349 & $97.6\%$  & $98.2\%$ \\
6(use $\beta$)  & $0.225$                                                & 2.259 & 1.353 & $0.0\%$   & $94.6\%$ \\
6(use $\theta$) & $0.254$                                                & 0.260 & 1.352 & $97.3\%$  & $92.8\%$ \\
\hline\hline
\end{tabular}
\label{TAB1}
\end{table}

Figure \ref{Power} plots the power curve of the hypothesis  test
of $\gamma = \gamma_0$ 
 versus $\gamma\neq \gamma_0$; here, we use  $\gamma_0 = \gamma + \delta\times (1,1,...,1)^T$ and $\delta>0$.

\begin{figure}[htbp]
\centering
\includegraphics[width = 3.2in]{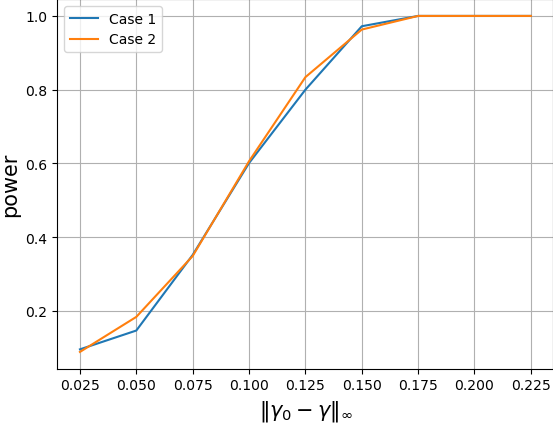}
\caption{Power of the test for cases 1 and 2; the  x-axis represents $\max_{i = 1,...,p_1}\vert\gamma_{0,i} - \gamma_i\vert$. Nominal size for the test is $ 5\%$; see algorithm \ref{BootALG1} for the meaning of notations.}
\label{Power}
\end{figure}

\section{Conclusion}
The paper at hand proposes an improved, i.e.,  debiased and thresholded, ridge regression method that recovers the sparsity of parameters and avoids introducing a large bias. Besides, it derives a consistency result and the Gaussian approximation theorem for the improved ridge estimator. An asymptotically valid confidence region
 for $\gamma = M\beta$ and  a hypothesis test of  $\gamma =\gamma_0$ are also constructed based on
  a wild bootstrap algorithm.
In addition,     a novel, hybrid resampling procedure was proposed
that can be used to perform interval prediction based on the improved ridge regression.

Numerical simulations indicate that improved ridge regression has comparable performance to the threshold Lasso
while having at least two major advantages: (a) Ridge regression is easily computed using a closed-form expression,
and (b) it appears to be quite robust against a non-optimal choice of the ridge parameter $\rho_n$. Therefore,
ridge regression may be found useful again  in applied work using high-dimensional data as long as practitioners
make sure to include   debiasing and thresholding.

\label{CONCLU}

\bibliographystyle{unsrt}
\bibliography{RidgeRegression}

\begin{thebibliography}{10}

\bibitem{doi:10.1111/j.2517-6161.1996.tb02080.x}
Robert Tibshirani.
\newblock Regression shrinkage and selection via the lasso.
\newblock {\em Journal of the Royal Statistical Society: Series B
  (Methodological)}, 58(1):267--288, 1996.

\bibitem{meinshausen2006}
Nicolai Meinshausen and Peter B\"{u}hlmann.
\newblock High-dimensional graphs and variable selection with the lasso.
\newblock {\em Ann. Statist.}, 34(3):1436--1462, 06 2006.

\bibitem{meinshausen2009}
Nicolai Meinshausen and Bin Yu.
\newblock Lasso-type recovery of sparse representations for high-dimensional
  data.
\newblock {\em Ann. Statist.}, 37(1):246--270, 02 2009.

\bibitem{vandegeer2008}
Sara~A. van~de Geer.
\newblock High-dimensional generalized linear models and the lasso.
\newblock {\em Ann. Statist.}, 36(2):614--645, 04 2008.

\bibitem{10.2307/41059185}
A.~Chatterjee and S.~N. Lahiri.
\newblock Asymptotic properties of the residual bootstrap for lasso estimators.
\newblock {\em Proceedings of the American Mathematical Society},
  138(12):4497--4509, 2010.

\bibitem{10.2307/41416396}
A.~Chatterjee and S.~N. Lahiri.
\newblock Bootstrapping lasso estimators.
\newblock {\em Journal of the American Statistical Association},
  106(494):608--625, 2011.

\bibitem{doi:10.1080/01621459.2016.1166114}
Xianyang Zhang and Guang Cheng.
\newblock Simultaneous inference for high-dimensional linear models.
\newblock {\em Journal of the American Statistical Association},
  112(518):757--768, 2017.

\bibitem{Dezeure2017}
Ruben Dezeure, Peter B\"{u}hlmann, and Cun-Hui Zhang.
\newblock High-dimensional simultaneous inference with the bootstrap.
\newblock {\em TEST}, 26(4):685--719, Dec 2017.

\bibitem{javanmard2018}
Adel Javanmard and Andrea Montanari.
\newblock Debiasing the lasso: Optimal sample size for gaussian designs.
\newblock {\em Ann. Statist.}, 46(6A):2593--2622, 12 2018.

\bibitem{doi:10.1198/016214501753382273}
Jianqing Fan and Runze Li.
\newblock Variable selection via nonconcave penalized likelihood and its oracle
  properties.
\newblock {\em Journal of the American Statistical Association},
  96(456):1348--1360, 2001.

\bibitem{chen2020}
Xi~Chen and Wen-Xin Zhou.
\newblock Robust inference via multiplier bootstrap.
\newblock {\em Ann. Statist.}, 48(3):1665--1691, 06 2020.

\bibitem{10.1214/11-EJS624}
Sara van~de Geer, Peter B\"{u}hlmann, and Shuheng Zhou.
\newblock {The adaptive and the thresholded Lasso for potentially misspecified
  models (and a lower bound for the Lasso)}.
\newblock {\em Electronic Journal of Statistics}, 5(none):688 -- 749, 2011.

\bibitem{10.5555/2031491}
Peter B\"{u}hlmann and Sara van~de Geer.
\newblock {\em Statistics for High-Dimensional Data: Methods, Theory and
  Applications}.
\newblock Springer-Verlag Berlin Heidelberg, 1st edition, 2011.

\bibitem{10.1214/19-EJS1599}
Sara van~de Geer.
\newblock {On the asymptotic variance of the debiased Lasso}.
\newblock {\em Electronic Journal of Statistics}, 13(2):2970 -- 3008, 2019.

\bibitem{10.1214/19-EJS1554}
Adel Javanmard and Hamid Javadi.
\newblock {False discovery rate control via debiased lasso}.
\newblock {\em Electronic Journal of Statistics}, 13(1):1212 -- 1253, 2019.

\bibitem{lee2016}
Jason~D. Lee, Dennis~L. Sun, Yuekai Sun, and Jonathan~E. Taylor.
\newblock Exact post-selection inference, with application to the lasso.
\newblock {\em Ann. Statist.}, 44(3):907--927, 06 2016.

\bibitem{liu2013}
Hanzhong Liu and Bin Yu.
\newblock Asymptotic properties of lasso+mls and lasso+ridge in sparse
  high-dimensional linear regression.
\newblock {\em Electron. J. Statist.}, 7:3124--3169, 2013.

\bibitem{tibshirani2018}
Ryan~J. Tibshirani, Alessandro Rinaldo, Rob Tibshirani, and Larry Wasserman.
\newblock Uniform asymptotic inference and the bootstrap after model selection.
\newblock {\em Ann. Statist.}, 46(3):1255--1287, 06 2018.

\bibitem{shao2012}
Jun Shao and Xinwei Deng.
\newblock Estimation in high-dimensional linear models with deterministic
  design matrices.
\newblock {\em Ann. Statist.}, 40(2):812--831, 04 2012.

\bibitem{DAI2018334}
Linlin Dai, Kani Chen, Zhihua Sun, Zhenqiu Liu, and Gang Li.
\newblock Broken adaptive ridge regression and its asymptotic properties.
\newblock {\em Journal of Multivariate Analysis}, 168:334 -- 351, 2018.

\bibitem{dobriban2018}
Edgar Dobriban and Stefan Wager.
\newblock High-dimensional asymptotics of prediction: Ridge regression and
  classification.
\newblock {\em Ann. Statist.}, 46(1):247--279, 02 2018.

\bibitem{buhlmann2013}
Peter B\"{u}hlmann.
\newblock Statistical significance in high-dimensional linear models.
\newblock {\em Bernoulli}, 19(4):1212--1242, 09 2013.

\bibitem{Lopes2014ARB}
Miles Lopes.
\newblock A residual bootstrap for high-dimensional regression with near
  low-rank designs.
\newblock In {\em Advances in Neural Information Processing Systems 27}, pages
  3239--3247, 2014.

\bibitem{matrix}
Roger~A. Horn and Charles~R. Johnson.
\newblock {\em Matrix Analysis}.
\newblock Cambridge University Press, 2012.

\bibitem{10.1214/aos/1176350142}
C.~F.~J. Wu.
\newblock {Jackknife, Bootstrap and Other Resampling Methods in Regression
  Analysis}.
\newblock {\em The Annals of Statistics}, 14(4):1261 -- 1295, 1986.

\bibitem{10.1214/aos/1176351062}
Regina~Y. Liu.
\newblock {Bootstrap Procedures under some Non-I.I.D. Models}.
\newblock {\em The Annals of Statistics}, 16(4):1696 -- 1708, 1988.

\bibitem{10.1214/aos/1176349025}
Enno Mammen.
\newblock {Bootstrap and Wild Bootstrap for High Dimensional Linear Models}.
\newblock {\em The Annals of Statistics}, 21(1):255 -- 285, 1993.

\bibitem{chernozhukov2013}
Victor Chernozhukov, Denis Chetverikov, and Kengo Kato.
\newblock Gaussian approximations and multiplier bootstrap for maxima of sums
  of high-dimensional random vectors.
\newblock {\em Ann. Statist.}, 41(6):2786--2819, 12 2013.

\bibitem{doi:10.1080/07474939608800362}
Juan~J. Dolado and Helmut L\"{u}tkepohl.
\newblock Making wald tests work for cointegrated var systems.
\newblock {\em Econometric Reviews}, 15(4):369--386, 1996.

\bibitem{SUN2011345}
Yixiao Sun.
\newblock Robust trend inference with series variance estimator and
  testing-optimal smoothing parameter.
\newblock {\em Journal of Econometrics}, 164(2):345 -- 366, 2011.

\bibitem{gv2011}
S\'{i}lvia Gon\c{c}alves and Timothy~J. Vogelsang.
\newblock Block bootstrap hac robust tests: the sophistication of the naive
  bootstrap.
\newblock {\em Econometric Theory}, 27(4):745 -- 791, 2011.

\bibitem{doi:10.1080/01621459.1985.10478220}
Robert~A. Stine.
\newblock Bootstrap prediction intervals for regression.
\newblock {\em Journal of the American Statistical Association},
  80(392):1026--1031, 1985.

\bibitem{model_free}
Dimitris~N. Politis.
\newblock {\em Model-Free Prediction and Regression}.
\newblock Springer-Verlag New York, 2015.

\bibitem{NEURIPS2019_5103c358}
Yaniv Romano, Evan Patterson, and Emmanuel Cand\`{e}s.
\newblock Conformalized quantile regression.
\newblock In {\em Advances in Neural Information Processing Systems},
  volume~32, pages 3543--3553. Curran Associates, Inc., 2019.

\bibitem{doi:10.1080/01621459.2019.1660174}
Yaniv Romano, Matteo Sesia, and Emmanuel Cand\`{e}s.
\newblock Deep knockoffs.
\newblock {\em Journal of the American Statistical Association},
  115(532):1861--1872, 2020.

\bibitem{Mstat}
Jun Shao.
\newblock {\em Mathematical Statistics}.
\newblock Springer-Verlag New York, 2003.

\bibitem{zhang2017}
Danna Zhang and Wei~Biao Wu.
\newblock Gaussian approximation for high dimensional time series.
\newblock {\em Ann. Statist.}, 45(5):1895--1919, 10 2017.

\bibitem{subsampling}
Dimitris~N. Politis, Joseph~P. Romano, and Michael Wolf.
\newblock {\em Subsampling}.
\newblock Springer-Verlag New York, 1999.

\bibitem{10.2307/24772746}
Jing Lei and Larry Wasserman.
\newblock Distribution-free prediction bands for non-parametric regression.
\newblock {\em Journal of the Royal Statistical Society. Series B (Statistical
  Methodology)}, 76(1):71--96, 2014.

\bibitem{chernozhukov2019distributional}
Victor Chernozhukov, Kaspar W\"{u}thrich, and Yinchu Zhu.
\newblock Distributional conformal prediction, 2019.

\bibitem{zhang2021bootstrap}
Yunyi Zhang and Dimitris~N. Politis.
\newblock Bootstrap prediction intervals with asymptotic conditional validity
  and unconditional guarantees.
\newblock (arXiv:2005.09145), 2021.

\bibitem{10.1214/aos/1033066211}
Enno Mammen.
\newblock {Empirical process of residuals for high-dimensional linear models}.
\newblock {\em Ann. Statist.}, 24(1):307 -- 335, 1996.

\bibitem{doi:10.1137/1105028}
Peter Whittle.
\newblock Bounds for the moments of linear and quadratic forms in independent
  variables.
\newblock {\em Theory of Probability \& Its Applications}, 5(3):302--305, 1960.

\bibitem{AntiConcentration}
Victor Chernozhukov, Denis Chetverikov, and Kengo Kato.
\newblock Comparison and anti-concentration bounds for maxima of gaussian
  random vectors.
\newblock {\em Probability Theory and Related Fields}, 162:47--70, 2015.

\bibitem{10.1093/biomet/asz020}
Mengyu Xu, Danna Zhang, and Wei~Biao Wu.
\newblock {Pearson's chi-squared statistics: approximation theory and beyond}.
\newblock {\em Biometrika}, 106(3):716--723, 04 2019.

\end{thebibliography}
\clearpage

\appendix
\section{Some important lemmas}
\label{lemmas}
This section introduces three useful lemmas. Lemma \ref{lemma_Se} comes from Whittle \cite{doi:10.1137/1105028}, which directly contributes to the model selection consistency. Lemma \ref{CI_lemma_1} and \ref{thm_1} are similar to Chernozhukov et al. \cite{chernozhukov2013}, they used a joint normal distribution to approximate the distribution of linear combinations of independent random variables.

\begin{lemma}
Suppose random variables $\epsilon_1,...,\epsilon_n$ are i.i.d., $\mathbf{E}\epsilon_1 = 0$, and $\exists$ a constant $m>0$ such that $\mathbf{E}\vert\epsilon_1\vert^m<\infty$. In addition suppose the matrix $\Gamma = (\gamma_{ij})_{i=1,2,...,k,j=1,2,...,n}$ satisfies
\begin{equation}
\max_{i=1,2,...,k}\sum_{j=1}^n\gamma_{ij}^2 \leq D,\ D > 0
\end{equation}
Then $\exists$ a constant $E$ which only depends on $m$ and $\mathbf{E}\vert\epsilon_1\vert^m$ such that for $\forall \delta>0$,
\begin{equation}
Prob\left(\max_{i=1,2,...,k}\vert\sum_{j=1}^n\gamma_{ij}\epsilon_j\vert>\delta\right)\leq \frac{k E D^{m/2}}{\delta^m}
\end{equation}
\label{lemma_Se}
\end{lemma}
\begin{proof}
From theorem 2 in \cite{doi:10.1137/1105028}, for any $i=1,2,...,k$,
\begin{equation}
Prob\left(\vert\sum_{j=1}^n\gamma_{ij}\epsilon_j\vert>\delta\right)\leq \frac{\mathbf{E}\vert \sum_{j=1}^n\gamma_{ij}\epsilon_j\vert^m}{\delta^m}\leq \frac{2^mC(m)\mathbf{E}\vert\epsilon_1\vert^m(\sum_{j=1}^n\gamma_{ij}^2)^{m/2}}{\delta^m}\leq \frac{2^mC(m)\mathbf{E}\vert\epsilon_1\vert^m D^{m/2}}{\delta^m}
\end{equation}
Choose $E = 2^mC(m)\mathbf{E}\vert\epsilon_1\vert^m$,
\begin{equation}
Prob\left(\max_{i=1,2,...,k}\vert\sum_{j=1}^n\gamma_{ij}\epsilon_j\vert>\delta\right)\leq \sum_{i=1}^kProb\left(\vert\sum_{j=1}^n\gamma_{ij}\epsilon_j\vert>\delta\right)\leq \frac{k E D^{m/2}}{\delta^m}
\end{equation}
\end{proof}

\begin{lemma}
Suppose $\epsilon = (\epsilon_1,...,\epsilon_n)^T$ are joint normal random variables with mean $\mathbf{E}\epsilon = 0$, non-singular covariance matrix $\mathbf{E}\epsilon\epsilon^T$, and positive marginal variance $\sigma_i^2 = \mathbf{E}\epsilon_i^2 > 0$, $i=1,2,...,n$. In addition, suppose $\exists$ two constants $0< c_0\leq C_0<\infty$ such that $c_0\leq \sigma_i\leq C_0$ for $i=1,2,...,n$. Then for any given $\delta>0$,
\begin{equation}
\sup_{x\in\mathbf{R}}\left( Prob(\max_{i=1,2,...,n}\vert\epsilon_i\vert\leq x+\delta)-Prob(\max_{i=1,2,...,n}\vert\epsilon_i\vert\leq x)\right)\leq C\delta(\sqrt{\log(n)} + \sqrt{\vert\log(\delta)\vert} + 1)
\end{equation}
$C$ only depends on $c_0$ and $C_0$.
\label{CI_lemma_1}
\end{lemma}

\begin{proof}[Proof of lemma \ref{CI_lemma_1}]
First for any $i=1,2,...,n$,
\begin{equation}
\begin{aligned}
\vert\epsilon_i\vert=\max(\epsilon_i,-\epsilon_i)
\Rightarrow \max_{i=1,...,n}\vert\epsilon_i\vert = \max(\max_{i=1,...,n}\epsilon_i,\max_{i=1,...,n}-\epsilon_i)
\end{aligned}
\label{maxx}
\end{equation}
Therefore, for any $x\in\mathbf{R}$,
\begin{equation}
\begin{aligned}
Prob(\max_{i=1,2,...,n}\vert\epsilon_i\vert\leq x+\delta)-Prob(\max_{i=1,2,...,n}\vert\epsilon_i\vert\leq x) = Prob(0<\max(\max_{i=1,...,n}\epsilon_i,\max_{i=1,...,n}-\epsilon_i)-x\leq \delta)\\
\leq Prob(0<\max_{i=1,...,n}\epsilon_i - x\leq\delta) + Prob(0<\max_{i=1,...,n}-\epsilon_i - x\leq \delta)\\
\leq Prob(\vert \max_{i=1,...,n}\epsilon_i - x\vert\leq \delta) + Prob(\vert \max_{i=1,...,n}-\epsilon_i - x\vert\leq \delta)
\end{aligned}
\end{equation}
$-\epsilon$ is also joint normal with mean $0$ and marginal variance $\mathbf{E}(-\epsilon_j)^2=\sigma_j^2$. From theorem 3 and (18), (19) in \cite{AntiConcentration}, by defining $\underline{\sigma} = \min_{i=1,2,...,n}\sigma_i\leq \max_{i=1,2,...,n}\sigma_i = \overline{\sigma}$,
we have
\begin{equation}
\begin{aligned}
\sup_{x\in\mathbf{R}}Prob\left(\vert\max_{i=1,2,...,n}\epsilon_i - x\vert\leq \delta\right)\leq \frac{\sqrt{2}\delta}{\underline{\sigma}}\left(\sqrt{\log(n)} + \sqrt{\max(1, \log(\underline{\sigma}) - \log(\delta))}\right)\\
+ \frac{4\sqrt{2}\delta}{\underline{\sigma}}\times\left(\frac{\overline{\sigma}}{\underline{\sigma}}\sqrt{\log(n)} + 2 + \frac{\overline{\sigma}}{\underline{\sigma}}\sqrt{\max(0,\log(\underline{\sigma}) - \log(\delta))}\right)\\
\leq \frac{\sqrt{2}\delta}{c_0}\left(\sqrt{\log(n)} + \sqrt{1 + \vert\log(c_0)\vert + \vert\log(C_0)\vert} + \sqrt{\vert\log(\delta)\vert}\right)\\
+ \frac{4\sqrt{2}\delta C_0}{c_0^2}\left(\sqrt{\log(n)} + 2 + \sqrt{\vert\log(c_0)\vert + \vert\log(C_0)\vert} + \sqrt{\vert\log(\delta)\vert}\right)\\
\leq \left(\frac{\sqrt{2\times (1 + \vert\log(c_0)\vert + \vert \log(C_0)\vert)}}{c_0} + \frac{4\sqrt{2}C_0}{c_0^2}(2 + \sqrt{\vert\log(c_0)\vert+\vert\log(C_0)\vert})\right)\times \delta\left(\sqrt{\log(n)} + 1 + \sqrt{\vert\log(\delta)\vert}\right)
\end{aligned}
\label{C_0}
\end{equation}
Choose $C = \frac{\sqrt{2\times (1 + \vert\log(c_0)\vert + \vert \log(C_0)\vert)}}{c_0} + \frac{4\sqrt{2}C_0}{c_0^2}(2 + \sqrt{\vert\log(c_0)\vert+\vert\log(C_0)\vert})$, which only depends on $c_0,\ C_0$. Then
\begin{equation}
\sup_{x\in\mathbf{R}}(Prob(\max_{i=1,2,...,n}\vert\epsilon_i\vert\leq x+\delta)-Prob(\max_{i=1,2,...,n}\vert\epsilon_i\vert\leq x))\leq 2C\delta(1 + \sqrt{\log(n)} + \sqrt{\vert\log(\delta)\vert})
\end{equation}
\end{proof}

\begin{lemma}
Suppose $\epsilon = (\epsilon_1,...\epsilon_n)^T$ are i.i.d. random variables with $\mathbf{E}\epsilon_1=0$, $\mathbf{E}\epsilon_1^2=\sigma^2$ and $\mathbf{E}\vert\epsilon_1\vert^3<\infty$. $\Gamma=(\gamma_{ij})_{i=1,2,...,n,j=1,2,...,k}$ is an $n\times k$ ($1\leq k\leq n$) rank $k$ matrix. And $\exists$ constants $0<c_\Gamma\leq C_\Gamma<\infty$ such that
$c_\Gamma^2\leq\sum_{j=1}^n \gamma_{ji}^2\leq C_\Gamma^2$ for $i=1,2,...,k$. $\widehat\sigma^2=\widehat\sigma^2(\epsilon)$ is an estimator of $\sigma^2$ and random variables $\epsilon^*|\epsilon=(\epsilon_1^*,...,\epsilon_n^*)^T|\epsilon$ are i.i.d. with $\epsilon^*_1$ having normal distribution $\mathcal{N}(0,\widehat\sigma^2)$. $\frac{\epsilon_i^*}{\widehat\sigma}$ is independent of $\epsilon$ for $i=1,2,...,n$. In addition, suppose one of the following conditions:

C1. $\exists$ a constant $0<\alpha_\sigma\leq 1/2$ such that
\begin{equation}
\vert\sigma^2 - \widehat{\sigma}^2\vert = O_p(n^{-\alpha_\sigma})\ \ \text{and } \max_{j=1,2,...,n,\ i=1,2,...,k}\vert\gamma_{ji}\vert = o(\min(n^{(\alpha_\sigma - 1)/2}\times \log^{-3/2}(n),\ n^{-1/3}\times\log^{-3/2}(n))
\end{equation}

C2. $\exists$ a constant $0<\alpha_\sigma< 1/2$ such that
\begin{equation}
\vert\sigma^2-\widehat\sigma^2\vert = O_p(n^{-\alpha_\sigma}),\ k = o(n^{\alpha_\sigma} \times \log^{-3}(n)),\ \max_{j=1,...,n,i=1,...,k}\vert\gamma_{ji}\vert = O(n^{-\alpha_\sigma}\times\log^{-3/2}(n))
\end{equation}

Then we have
\begin{equation}
\sup_{x\in[0,\infty)}\vert Prob(\max_{i=1,2,...,k}\vert\sum_{j=1}^n \gamma_{ji}\epsilon_j\vert\leq x)-Prob^*(\max_{i=1,2,...,k}\vert\sum_{j=1}^n\gamma_{ji}\epsilon_j^*\vert\leq x)\vert = o_P(1)
\label{res_11}
\end{equation}

In particular, if $\widehat{\sigma} = \sigma$, by assuming one of the following conditions,

$C_1^{'}.$
\begin{equation}
\max_{j=1,2,...,n,i=1,2,...,k}\vert\gamma_{ji}\vert = o(n^{-1/3}\times \log^{-3/2}(n))
\end{equation}

$C_2^{'}.$
\begin{equation}
k\times \max_{j=1,2,...,n,i=1,2,...,k}\vert\gamma_{ji}\vert = o(\log^{-9/2}(n))
\end{equation}
Then we have
\begin{equation}
\sup_{x\in[0,\infty)}\vert Prob(\max_{i=1,2,...,k}\vert\sum_{j=1}^n \gamma_{ji}\epsilon_j\vert\leq x)-Prob(\max_{i=1,2,...,k}\vert\sum_{j=1}^n\gamma_{ji}\epsilon_j^*\vert\leq x)\vert = o(1)
\label{RESTHM2}
\end{equation}
\label{thm_1}
\end{lemma}

\begin{proof}[Proof of lemma \ref{thm_1}]
In this proof we define $\Gamma = (\gamma_1,...,\gamma_k)$ with $\gamma_i = (\gamma_{1i}, \gamma_{2i},...,\gamma_{ni})^T\in\mathbf{R}^n$. For $i=1,2,...,k$, $\gamma_i^T\epsilon = \sum_{j=1}^n\gamma_{ji}\epsilon_j$.
From lemma A.2 and (8) in \cite{chernozhukov2013}, and (S1) to (S5) in \cite{10.1093/biomet/asz020}, for $x=(x_1,...,x_n)$ and $y,z\in\mathbf{R}$, define
\begin{equation}
F_\beta(x)=\frac{1}{\beta}\log\left(\sum_{i=1}^n\exp(\beta x_i)\right),\ g_0(y) = (1-\min(1,\max(y,0))^4)^4,\ g_{\psi,z}(y) = g_0(\psi(y-z))
\label{notion}
\end{equation}
Here $\beta,\psi>0$. Then $g_{\psi,z}\in\mathbf{C}^3$ is nonincreasing function. $g_0 =1$ with $y\leq 0$, $0$ with $y\geq 1$, and
\begin{equation}
\begin{aligned}
g_*=\max_{y\in\mathbf{R}}(\vert g_0^{'}(y)\vert+\vert g_0^{''}(y)\vert+\vert g_0^{'''}(y)\vert)<\infty,\ \mathbf{1}_{y\leq z}\leq g_{\psi,z}(y)\leq \mathbf{1}_{y\leq z+\psi^{-1}}\\
\sup_{y,z\in\mathbf{R}}\vert g_{\psi,z}^{'}(y)\vert\leq g_*\psi,\ \sup_{y,z\in\mathbf{R}}\vert g_{\psi,z}^{''}(y)\vert\leq g_*\psi^2,\ \sup_{y,z\in\mathbf{R}}\vert g_{\psi,z}^{'''}(y)\vert\leq g_*\psi^3\\
\frac{\partial F_\beta}{\partial x_i} = \frac{\exp(\beta x_i)}{\sum_{j=1}^n \exp(\beta x_j)}\Rightarrow \frac{\partial F_\beta}{\partial x_i}\geq 0,\ \sum_{i=1}^n \frac{\partial F_\beta}{\partial x_i}=1,\ \sum_{i=1}^n\sum_{j=1}^n\vert\frac{\partial^2 F_\beta}{\partial x_i\partial x_j}\vert\leq 2\beta,\ \sum_{i=1}^n\sum_{j=1}^n\sum_{k=1}^n\vert\frac{\partial^3 F_\beta}{\partial x_i\partial x_j\partial x_k}\vert\leq 6\beta^2\\
F_\beta(x_1,...,x_n)-\frac{\log(n)}{\beta}\leq\max_{i=1,...,n} x_i\leq F_\beta(x_1,...,x_n)
\end{aligned}
\label{prop}
\end{equation}
For any given $x=(x_1,...,x_n)\in\mathbf{R}^n$, define function

\noindent$G_\beta(x) = \frac{1}{\beta}\log(\sum_{i=1}^n\exp(\beta x_i)+\sum_{i=1}^n\exp(-\beta x_i)) = F_\beta(x_1,...,x_n,-x_1,...,-x_n)$.
From \eqref{prop} and \eqref{maxx}, for $i,j,k=1,...,n$
\begin{equation}
\begin{aligned}
G_\beta(x)-\frac{\log(2n)}{\beta}\leq\max_{i=1,...,n}\vert x_i\vert\leq G_\beta(x),\ \frac{\partial G_\beta}{\partial x_i} = \frac{\partial F_\beta}{\partial x_i}-\frac{\partial F_\beta}{\partial x_{i+n}}\Rightarrow \sum_{i=1}^n\vert\frac{\partial G_\beta}{\partial x_i}\vert \leq \sum_{i=1}^n \frac{\partial F_\beta}{\partial x_i} +\frac{\partial F_\beta}{\partial x_{i+n}} = 1\\
\frac{\partial^2 G_\beta}{\partial x_i\partial x_j} = \frac{\partial^2 F_\beta}{\partial x_i\partial x_j} - \frac{\partial^2 F_\beta}{\partial x_i\partial x_{j+n}} - \frac{\partial^2 F_\beta}{\partial x_{i+n}\partial x_j} + \frac{\partial^2 F_\beta}{\partial x_{i+n}\partial x_{j+n}}\Rightarrow \sum_{i=1}^n\sum_{j=1}^n\vert\frac{\partial^2 G_\beta}{\partial x_i\partial x_j}\vert\leq \sum_{i=1}^{2n}\sum_{j=1}^{2n}\vert\frac{\partial^2 F_\beta}{\partial x_i\partial x_j}\vert\leq 2\beta\\
\frac{\partial ^3 G_\beta}{\partial x_i\partial x_j\partial x_k} = \frac{\partial^3 F_\beta}{\partial x_i\partial x_j\partial x_k}-\frac{\partial^3 F_\beta}{\partial x_i\partial x_j\partial x_{k+n}} - \frac{\partial^3 F_\beta}{\partial x_i\partial x_{j+n}\partial x_k} +\frac{\partial^3 F_\beta}{\partial x_i\partial x_{j+n}\partial x_{k+n}} - \frac{\partial^3 F_\beta}{\partial x_{i+n}\partial x_j\partial x_k} + \frac{\partial^3 F_\beta}{\partial x_{i+n}\partial x_j\partial x_{k+n}} \\
+ \frac{\partial^3 F_\beta}{\partial x_{i+n}\partial x_{j+n}\partial x_k} - \frac{\partial^3 F_\beta}{\partial x_{i+n}\partial x_{j+n}\partial x_{k+n}}
\Rightarrow \sum_{i=1}^n\sum_{j=1}^n\sum_{k=1}^n\vert\frac{\partial ^3 G_\beta}{\partial x_i\partial x_j\partial x_k}\vert\leq \sum_{i=1}^{2n}\sum_{j=1}^{2n}\sum_{k=1}^{2n}\vert\frac{\partial^3 F_\beta}{\partial x_i\partial x_j\partial x_k}\vert\leq 6\beta^2
\end{aligned}
\label{property of G}
\end{equation}
Define $h_{\beta, \psi, x}(x_1,...,x_n) = g_{\psi, x}(G_\beta(x_1,...,x_n))$. Direct calculation shows $\frac{\partial h_{\beta, \psi, x}(x_1,...,x_n)}{\partial x_i} = g^{'}_{\psi,x}(G_\beta(x_1,...,x_n))\frac{\partial G_\beta}{\partial x_i}\Rightarrow \sum_{i=1}^n \vert \frac{\partial h_{\beta, \psi, x}(x_1,...,x_n)}{\partial x_i}\vert\leq g_*\psi$;
\begin{equation}
\begin{aligned}
\frac{\partial^2 h_{\beta, \psi, x}(x_1,...,x_n)}{\partial x_i\partial x_j} = g^{''}_{\psi,x}(G_\beta(x_1,...,x_n))\frac{\partial G_\beta}{\partial x_i}\frac{\partial G_\beta}{\partial x_j} + g^{'}_{\psi,x}(G_\beta(x_1,...,x_n))\frac{\partial^2 G_\beta}{\partial x_i\partial x_j}\\
\Rightarrow \sum_{i=1}^n\sum_{j=1}^n\vert \frac{\partial^2 h_{\beta, \psi, x}(x_1,...,x_n)}{\partial x_i\partial x_j}\vert\leq g_*\psi^2\left(\sum_{i=1}^n\vert\frac{\partial G_\beta}{\partial x_i}\vert\right)^2 + g_*\psi\sum_{i=1}^n\sum_{j=1}^n\vert\frac{\partial^2 G_\beta}{\partial x_i\partial x_j}\vert\leq g_*\psi^2+2g_*\psi\beta\\
\text{and }\frac{\partial^3 h_{\beta, \psi, x}(x_1,...,x_n)}{\partial x_i\partial x_j\partial x_k} = g^{'''}_{\psi,x}(G_\beta(x_1,...,x_n))\frac{\partial G_\beta}{\partial x_i}\frac{\partial G_\beta}{\partial x_j}\frac{\partial G_\beta}{\partial x_k} + g^{''}_{\psi,x}(G_\beta(x_1,...,x_n))\frac{\partial^2 G_\beta}{\partial x_i\partial x_k}\frac{\partial G_\beta}{\partial x_j}\\ +g^{''}_{\psi,x}(G_\beta(x_1,...,x_n))\frac{\partial G_\beta}{\partial x_i}\frac{\partial^2 G_\beta}{\partial x_j\partial x_k}
+g^{''}_{\psi,x}(G_\beta(x_1,...,x_n))\frac{\partial^2 G_\beta}{\partial x_i\partial x_j}\frac{\partial G_\beta}{\partial x_k}
+g^{'}_{\psi,x}(G_\beta(x_1,...,x_n))\frac{\partial^3 G_\beta}{\partial x_i\partial x_j\partial x_k}\\
\Rightarrow \sum_{i=1}^n\sum_{j=1}^n\sum_{k=1}^n\vert\frac{\partial^3 h_{\beta, \psi, x}(x_1,...,x_n)}{\partial x_i\partial x_j\partial x_k}\vert\leq g_*\psi^3\left(\sum_{i=1}^n\vert\frac{\partial G_\beta}{\partial x_i}\vert\right)^3 + 3g_*\psi^2\left(\sum_{i=1}^n\sum_{j=1}^n\vert\frac{\partial^2G_\beta}{\partial x_i\partial x_j}\vert\right)\times\left(\sum_{k=1}^n\vert\frac{\partial G_\beta}{\partial x_k}\vert\right)\\
+g_*\psi\sum_{i=1}^n\sum_{j=1}^n\sum_{k=1}^n\vert\frac{\partial ^3 G_\beta}{\partial x_i\partial x_j\partial x_k}\vert\leq g_*\psi^3+6g_*\psi^2\beta+6g_*\psi\beta^2
\end{aligned}
\label{deri}
\end{equation}
Define $\xi=(\xi_1,...,\xi_n)$ as i.i.d. random variables with the same marginal distribution as $\epsilon_1$, and is independent of $\epsilon,\epsilon^*$. Therefore, $Prob(\max_{i=1,2,...,k}\vert\gamma_i^T\epsilon\vert\leq x)=Prob^*(\max_{i=1,2,...,k}\vert\sum_{j=1}^n\gamma_i^T\xi\vert\leq x)$ for any $x$. Since $c_\Gamma^2\leq \mathbf{E}^*\left(\sum_{l=1}^n \frac{\gamma_{il}\epsilon^*_l}{\widehat{\sigma}}\right)^2 = \sum_{l=1}^n\gamma_{il}^2\leq C_\Gamma^2\ \text{for $i=1,2,...,k$}$. According to \eqref{maxx}, \eqref{property of G} and lemma \ref{CI_lemma_1},
$\exists$ a constant $C$ which only depends on $c_\Gamma$ and $C_\Gamma$ such that for any given $\psi,\beta,\widehat\sigma>0$,
\begin{equation}
\begin{aligned}
\sup_{x\in\mathbf{R}}\left(Prob^*\left(\max_{i=1,2,...,k}\vert\gamma_i^T\epsilon^*\vert\leq x + \frac{1}{\psi} + \frac{\log(2k)}{\beta}\right) - Prob^*\left(\max_{i=1,2,...,k}\vert\gamma_i^T\epsilon^*\vert\leq x\right)\right)\\
=\sup_{x\in\mathbf{R}}\left(Prob^*\left(\max_{i=1,2,...,k}\vert\frac{\gamma_i^T\epsilon^*}{\widehat{\sigma}}\vert\leq\frac{x}{\widehat{\sigma}} + \frac{1}{\psi\widehat{\sigma}} + \frac{\log(2k)}{\beta\widehat{\sigma}}\right) - Prob^*\left(\max_{i=1,2,...,k}\vert\frac{\gamma_i^T\epsilon^*}{\widehat{\sigma}}\vert
\leq \frac{x}{\widehat{\sigma}}\right)\right)\\
\leq C\times\left(\frac{1}{\psi\widehat{\sigma}} + \frac{\log(2k)}{\beta\widehat{\sigma}}\right)\times\left(1 + \sqrt{\log(n)} + \sqrt{\vert\log\left(\frac{1}{\psi\widehat{\sigma}} + \frac{\log(2k)}{\beta\widehat{\sigma}}\right)\vert}\right)
\end{aligned}
\end{equation}
Define $z = C\times\left(\frac{1}{\psi\widehat{\sigma}} + \frac{\log(2k)}{\beta\widehat{\sigma}}\right)\times\left(1 + \sqrt{\log(n)} + \sqrt{\vert\log\left(\frac{1}{\psi\widehat{\sigma}} + \frac{\log(2k)}{\beta\widehat{\sigma}}\right)\vert}\right)$. For any $x\geq 0$,
\begin{equation}
\begin{aligned}
Prob(\max_{i=1,2,...,k}\vert\gamma_i^T\epsilon\vert\leq x)-Prob^*(\max_{i=1,2,...,k}\vert\gamma_i^T\epsilon^*\vert\leq x)\\
\leq  Prob^*(\max_{i=1,2,...,k}\vert\gamma_i^T\xi\vert\leq x)-Prob^*(\max_{i=1,2,...,k}\vert\gamma_i^T\epsilon^*\vert\leq x +\frac{1}{\psi}+\frac{\log(2k)}{\beta})
+ z\\
\leq Prob^*(G_\beta(\gamma_1^T\xi,...,\gamma_k^T\xi)\leq x+\frac{\log(2k)}{\beta}) - Prob^*(G_\beta(\gamma_1^T\epsilon^*,...,\gamma_k^T\epsilon^*)\leq x+\frac{1}{\psi}+\frac{\log(2k)}{\beta}) + z\\
\leq \mathbf{E}^*h_{\beta, \psi, x+\frac{\log(2k)}{\beta}}(\gamma_1^T\xi,...,\gamma_k^T\xi)-h_{\beta, \psi, x+\frac{\log(2k)}{\beta}}(\gamma_1^T\epsilon^*,...,\gamma_k^T\epsilon^*)
+ z\\
Prob(\max_{i=1,2,...,k}\vert\gamma_i^T\epsilon\vert\leq x)-Prob^*(\max_{i=1,2,...,k}\vert\gamma_i^T\epsilon^*\vert\leq x)\\
\geq Prob^*(\max_{i=1,2,...,k}\vert\gamma_i^T\xi\vert\leq x)-Prob^*(\max_{i=1,2,...,k}\vert\gamma_i^T\epsilon^*\vert\leq x -\frac{1}{\psi}-\frac{\log(2k)}{\beta}) - z\\
\geq Prob^*(G_\beta(\gamma_1^T\xi,...,\gamma_k^T\xi)\leq x) - Prob^*(G_\beta(\gamma_1^T\epsilon^*,...,\gamma_k^T\epsilon^*)\leq x-\frac{1}{\psi})
-z\\
\geq \mathbf{E}^*h_{\beta,\psi,x-\frac{1}{\psi}}(\gamma_1^T\xi,...,\gamma_k^T\xi)-h_{\beta,\psi,x-\frac{1}{\psi}}(\gamma_1^T\epsilon^*,...,\gamma_k^T\epsilon^*)
- z
\end{aligned}
\end{equation}
Therefore, we have
\begin{equation}
\begin{aligned}
\sup_{x\in[0,\infty)}\vert Prob(\max_{i=1,2,...,k}\vert\gamma_i^T\epsilon\vert\leq x)-Prob^*(\max_{i=1,2,...,k}\vert\gamma_i^T\epsilon^*\vert\leq x)\vert
\leq z +\sup_{x\in\mathbf{R}}\vert \mathbf{E}^*h_{\beta,\psi,x}(\gamma_1^T\xi,...,\gamma_k^T\xi)- h_{\beta,\psi,x}(\gamma_1^T\epsilon^*,...,\gamma_k^T\epsilon^*)\vert
\end{aligned}
\label{first}
\end{equation}
For any $i=1,2,...,k,j=1,2,...,n$, define $H_{ij}=\sum_{s=1}^{j-1}\gamma_{si}\xi_s+\sum_{s=j+1}^n\gamma_{si}\epsilon_s^*$, $m_{ij}=\gamma_{ji}\xi_j$ and $m^*_{ij}=\gamma_{ji}\epsilon^*_j$, we have $H_{ij}+m_{ij}=H_{ij+1}+m^*_{ij+1}$, and
\begin{equation}
\begin{aligned}
\sup_{x\in\mathbf{R}}\vert \mathbf{E}^*h_{\beta,\psi,x}(\gamma_1^T\xi,...,\gamma_k^T\xi) - h_{\beta,\psi,x}(\gamma_1^T\epsilon^*,...,\gamma_k^T\epsilon^*)\vert\\
=\sup_{x\in\mathbf{R}}\vert \sum_{s=1}^n\mathbf{E}^*h_{\beta,\psi,x}(H_{1s}+m_{1s},...,H_{ks}+m_{ks})
-h_{\beta,\psi,x}(H_{1s}+m_{1s}^*,...,H_{ks}+m_{ks}^*)\vert\\
\leq \sum_{s=1}^n \sup_{x\in\mathbf{R}}\vert\mathbf{E}^*h_{\beta,\psi,x}(H_{1s}+m_{1s},...,H_{ks}+m_{ks})
-h_{\beta,\psi,x}(H_{1s}+m_{1s}^*,...,H_{ks}+m_{ks}^*)\vert
\end{aligned}
\end{equation}
Since $\mathbf{E}(\xi_s|\epsilon,\xi_b,\epsilon^*_b,b\neq s)=\mathbf{E}(\epsilon^*_s|\epsilon,\xi_b,\epsilon^*_b,b\neq s)=0$, $\mathbf{E}(\xi_s^2-\epsilon^{*2}_s|\epsilon,\xi_b,\epsilon^*_b,b\neq s)=\sigma^2-\widehat\sigma^2$, from multivariate Taylor's theorem and \eqref{deri}, for any $s=1,2,...,n$ and $x\in\mathbf{R}$,
\begin{equation}
\begin{aligned}
\vert\mathbf{E}\left(h_{\beta, \psi, x}(H_{1s}+m_{1s},...,H_{ks}+m_{ks})
-h_{\beta,\psi,x}(H_{1s}+m_{1s}^*,...,H_{ks}+m_{ks}^*)\right)\Big|\epsilon,\xi_b,\epsilon^*_b,b\neq s\vert\\
\leq\vert\sum_{i=1}^k \frac{\partial h_{\beta,\psi,x}(H_{1s},...,H_{ks})}{\partial x_i}\gamma_{si}\mathbf{E}(\xi_s-\epsilon^*_s|\epsilon,\xi_b,\epsilon^*_b,b\neq s)\vert+\frac{1}{2}\vert\sum_{i=1}^k\sum_{j=1}^k\frac{\partial^2 h_{\beta,\psi,x}(H_{1s},...,H_{ks})}{\partial x_i\partial x_j}\gamma_{si}\gamma_{sj}\mathbf{E}(\xi_s^2-\epsilon^{*2}_s|\epsilon,\xi_b,\epsilon^*_b,b\neq s)\vert\\
+(g_*\psi^3+g_*\psi^2\beta+g_*\psi\beta^2)\max_{i = 1,2,...,k}\vert\gamma_{si}\vert^3\times(\mathbf{E}\vert\epsilon_1\vert^3 + \widehat{\sigma}^3D)\\
\Rightarrow \sup_{x\in\mathbf{R}}\vert\mathbf{E}h_{\beta, \psi, x}(H_{1s}+m_{1s},...,H_{ks}+m_{ks})
-h_{\beta,\psi,x}(H_{1s}+m_{1s}^*,...,H_{ks}+m_{ks}^*)|\epsilon,\xi_b,\epsilon^*_b,b\neq s\vert\\
\leq g_*(\psi^2 + \psi\beta)\vert\sigma^2-\widehat\sigma^2\vert\times\max_{i=1,...,k}\gamma_{si}^2 + (\mathbf{E}\vert\epsilon_1\vert^3+\widehat\sigma^3D)\times g_*(\psi^3+\psi^2\beta+\psi\beta^2)\times \max_{i=1,...,k}\vert\gamma_{si}\vert^3
\end{aligned}
\end{equation}
Here $D = \mathbf{E}\vert Y\vert^3$ with $Y$ having normal distribution with mean $0$ and variance $1$. Then
\begin{equation}
\begin{aligned}
\sup_{x\in[0,\infty)}\vert Prob(\max_{i=1,2,...,k}\vert\gamma_i^T\epsilon\vert\leq x)-Prob^*(\max_{i=1,2,...,k}\vert\gamma_i^T\epsilon^*\vert\leq x)\vert\\
\leq z +(g_*\psi^2+g_*\psi\beta)\vert\sigma^2-\widehat\sigma^2\vert\times\sum_{s=1}^n\max_{i=1,...,k}\gamma_{si}^2 + (\mathbf{E}\vert\epsilon_1\vert^3+\widehat\sigma^3D)\times g_*(\psi^3+\psi^2\beta+\psi\beta^2)\times \sum_{s=1}^n\max_{i=1,...,k}\vert\gamma_{si}\vert^3
\end{aligned}
\label{small}
\end{equation}
In particular, for any given $\delta > 0$, choose $\psi = \beta = \log^{3/2}(n)/\delta^{1/4}$ and suppose $\frac{3\sigma}{2}>\widehat{\sigma} > \frac{\sigma}{2}$. For sufficiently large $n$ we have $\frac{1}{\psi\widehat{\sigma}} + \frac{\log(2k)}{\beta\widehat{\sigma}}\leq \frac{4\log(n)}{\psi\sigma}\leq \frac{4\delta^{1/4}}{\sigma\sqrt{\log(n)}} < 1$ and
\begin{equation}
z \leq \frac{4C\log(n)}{\psi\sigma}\times \left(2\sqrt{\log(n)} + \sqrt{\log(\psi\widehat{\sigma})}\right)\leq \frac{4C\delta^{1/4}}{\sigma}\left(2 +\sqrt{\frac{\frac{3}{2}\log(\log(n)) + \log(3\sigma/2\delta^{1/4})}{\log(n)}}\right)\leq C^{'}\delta^{1/4},\ C^{'} = \frac{12C}{\sigma}
\end{equation}

Suppose condition C1. For any $1>\delta>0$, $\exists D_\delta >0$ such that for sufficiently large $n$,
\begin{equation}
\begin{aligned}
Prob\left(\vert\sigma^2 - \widehat{\sigma}^2\vert \leq D_\delta\times n^{-\alpha_\sigma}\right)> 1- \delta,\ \ \   \max_{j=1,2,...,n,i=1,2,...,k}\vert\gamma_{ji}\vert<\delta\times n^{(\alpha_\sigma -  1)/2}\times\log^{-3/2}(n),\\ \max_{j=1,2,...,n,i=1,2,...,k}\vert\gamma_{ji}\vert< \delta\times n^{-1/3}\times\log^{-3/2}(n)
\end{aligned}
\label{CondC1}
\end{equation}
Choose $\psi = \beta = \log^{3/2}(n)/\delta^{1/4}$. According to \eqref{small}, for sufficiently large $n$, \eqref{CondC1} happens and $\frac{1}{2}\sigma<\widehat{\sigma}<\frac{3}{2}\sigma$ with probability $1-\delta$. If \eqref{CondC1} happens,
\begin{equation}
\begin{aligned}
\sup_{x\in[0,\infty)}\vert Prob(\max_{i=1,2,...,k}\vert\gamma_i^T\epsilon\vert\leq x)-Prob^*(\max_{i=1,2,...,k}\vert\gamma_i^T\epsilon^*\vert\leq x)\vert\\
\leq C^{'}\delta^{1/4} + 2g_*\psi^2\times D_\delta\times n^{-\alpha_\sigma}\times \frac{\delta^2\times n^{\alpha_\sigma}}{\log^{3}(n)} + (\mathbf{E}\vert\epsilon_1\vert^3+\frac{27D}{8}\sigma^3)\times3g_*\psi^3\times \delta^3\times n\times \frac{1}{n\log^{9/2}(n)}\\
=C^{'}\delta^{1/4} + 2g_*D_\delta\delta^{3/2} + 3g_*(\mathbf{E}\vert\epsilon_1\vert^3+\frac{27D}{8}\sigma^3)\times \delta^{9/4}\\
\end{aligned}
\end{equation}
For $\delta > 0$ can be arbitrarily small, we prove \eqref{res_11}.

Suppose condition C2. For any $\delta > 0$, there exists $D_\delta > 0$ such that
\begin{equation}
\begin{aligned}
Prob\left(\vert\sigma^2 - \widehat{\sigma}^2\vert \leq D_\delta \times n^{-\alpha_\sigma}\right)\geq 1-\delta,\ \ \ k\leq \frac{\delta n^{\alpha_\sigma}}{\log^{3}(n)},\ \ \
\max_{i=1,2,...,k}\sum_{j=1}^n\gamma_{ji}^2\leq D_\delta,\
\max_{j=1,2,...,n,i=1,2,...,k}\vert\gamma_{ji}\vert\leq\frac{D_\delta\times n^{-\alpha_\sigma}}{\log^{3/2}(n)}
\end{aligned}
\label{CondC2}
\end{equation}
Since
\begin{equation}
\begin{aligned}
\sum_{j=1}^n\max_{i=1,...,k}\gamma_{ji}^2\leq \sum_{j=1}^n\sum_{i=1}^k\gamma_{ji}^2\leq kD_\delta\\
\sum_{j=1}^n\max_{i=1,...,k}\gamma_{ji}^3\leq \max_{j=1,2,...,n,i=1,2,...,k}\vert\gamma_{ji}\vert\times \sum_{j=1}^n\max_{i=1,...,k}\gamma_{ji}^2\leq kD_\delta\times \max_{j=1,2,...,n,i=1,2,...,k}\vert\gamma_{ji}\vert
\end{aligned}
\label{KSSS}
\end{equation}
If \eqref{CondC2} happens, by choosing $\psi = \beta = \log^{3/2}(n)/\delta^{1/4}$
\begin{equation}
\begin{aligned}
\sup_{x\in[0,\infty)}\vert Prob(\max_{i=1,2,...,k}\vert\gamma_i^T\epsilon\vert\leq x)-Prob^*(\max_{i=1,2,...,k}\vert\gamma_i^T\epsilon^*\vert\leq x)\vert\\
\leq C^{'}\delta^{1/4} + 2g_*\psi^2D_\delta n^{-\alpha_\sigma}\times kD_\delta + (\mathbf{E}\vert\epsilon_1\vert^3 + \frac{27D}{8}\sigma^3)\times 3g_*\psi^3\times kD_\delta\max_{j=1,2,...,n,i=1,2,...,k}\vert\gamma_{ji}\vert\\
\leq C^{'}\delta^{1/4} + 2g_*D_\delta^2\times \frac{\log^3(n)}{\delta^{1/2}}\times\frac{\delta n^{\alpha_\sigma}}{\log^{3}(n)}\times n^{-\alpha_\sigma} + 3(\mathbf{E}\vert\epsilon_1\vert^3 + \frac{27D}{8}\sigma^3)g_*D_\delta^2\times \frac{\log^{9/2}(n)}{\delta^{3/4}}\times \frac{\delta n^{\alpha_\sigma}}{\log^{3}(n)}\times \frac{n^{-\alpha_\sigma}}{\log^{3/2}(n)}\\
=C^{'}\delta^{1/4} + 2g_*D_\delta^2\delta^{1/2} + 3(\mathbf{E}\vert\epsilon_1\vert^3 + \frac{27D}{8}\sigma^3)g_*D_\delta^2\times\delta^{1/4}
\end{aligned}
\end{equation}
and we prove \eqref{res_11}.

If $\widehat{\sigma} = \sigma$. We choose $\psi = \beta = \log^{3/2}(n)/\delta^{1/4}$, \eqref{small} can be modified to
\begin{equation}
\begin{aligned}
\sup_{x\in[0,\infty)}\vert Prob(\max_{i=1,2,...,k}\vert\gamma_i^T\epsilon\vert\leq x)-Prob(\max_{i=1,2,...,k}\vert\gamma_i^T\epsilon^*\vert\leq x)\vert
\leq C^{'}\delta^{1/4} + (\mathbf{E}\vert\epsilon_1\vert^3+D\sigma^3) g_*\psi(\psi^2+\psi\beta+\beta^2) \sum_{s=1}^n\max_{i=1,...,k}\vert\gamma_{si}\vert^3
\end{aligned}
\end{equation}

Suppose condition $C1^{'}$. For any $\delta > 0$ and sufficiently large $n$, $\max_{j=1,2,...,n, i=1,2,...,k}\vert\gamma_{ji}\vert\leq \delta\times n^{-1/3}\log^{-3/2}(n)$,
\begin{equation}
\begin{aligned}
\sup_{x\in[0,\infty)}\vert Prob(\max_{i=1,2,...,k}\vert\sum_{j=1}^n \gamma_{ji}\epsilon_j\vert\leq x)-Prob(\max_{i=1,2,...,k}\vert\sum_{j=1}^n\gamma_{ji}\epsilon_j^*\vert\leq x)\vert\leq C^{'}\delta^{1/4} + 3(\mathbf{E}\vert\epsilon_1\vert^3+D\sigma^3)g_*\times\delta^{9/4}
\end{aligned}
\end{equation}
and we prove \eqref{RESTHM2}.

Suppose condition $C2^{'}$. For any $\delta > 0$ and sufficiently large $n$, $k\times \max_{j=1,2,...,n,i=1,2,...,k}\vert\gamma_{ji}\vert\leq \delta\log^{-9/2}(n)$. According to \eqref{KSSS}, for sufficiently large $n$ we have
\begin{equation}
\begin{aligned}
\sup_{x\in[0,\infty)}\vert Prob(\max_{i=1,2,...,k}\vert\sum_{j=1}^n \gamma_{ji}\epsilon_j\vert\leq x)-Prob(\max_{i=1,2,...,k}\vert\sum_{j=1}^n\gamma_{ji}\epsilon_j^*\vert\leq x)\vert\leq C^{'}\delta^{1/4}
+3(\mathbf{E}\vert\epsilon_1\vert^3+D\sigma^3)g_*D_\delta\times\delta^{1/4}
\end{aligned}
\end{equation}
and we prove \eqref{RESTHM2}.
\end{proof}

Condition C1 implies $C1^{'}$, and condition C2 implies $C2^{'}$. The additional proportions in C1 and C2 accommodate the error introduced in estimating errors' variance. Condition C1 is designed for the situation when the number of linear combinations $k$ is as large as the sample size $n$; and condition C2 can be used when $k$ is significantly smaller than $n$.

The difference between lemma \ref{thm_1} and the classical central limit theorem is that $k$ can grow as $n$ increases. The maximum $\max_{i=1,2,...,k}\vert\sum_{j=1}^n \gamma_{ji}\epsilon_j\vert$ does not have an asymptotic distribution if $k\to\infty$. However, if the random variables are mixed well, approximating the distribution of $\max_{i=1,2,...,k}\vert\sum_{j=1}^n \gamma_{ji}\epsilon_j\vert$ by the distribution of the maximum of normal random variables is still applicable.
With the help of lemma \ref{thm_1}, we can establish the normal approximation theorem and construct the simultaneous confidence region for $\widehat{\gamma}$(defined in \eqref{HAT}).

\section{Proofs of theorems in section \ref{CLT}}
\label{APPENX}
This section applies notations in section \ref{prelimary}.
\begin{proof}[Proof of theorem \ref{thm1}]
From \eqref{EXPAN},
\begin{equation}
\begin{aligned}
Prob\left(\widehat{\mathcal{N}}_{b_n}\neq \mathcal{N}_{b_n}\right)\leq Prob\left(\min_{i\in\mathcal{N}_{b_n}}\vert\widetilde{\theta}_i\vert\leq b_n\right) + Prob\left(\max_{i\not\in\mathcal{N}_{b_n}}\vert\widetilde{\theta}_i\vert>b_n\right)\\
\leq Prob\left(\min_{i\in\mathcal{N}_{b_n}}\vert\theta_i\vert - \max_{i\in\mathcal{N}_{b_n}}\rho_n^2\vert\sum_{j=1}^r\frac{q_{ij}\zeta_j}{(\lambda_j^2 + \rho_n)^2}\vert - \max_{i\in\mathcal{N}_{b_n}}\vert\sum_{j=1}^rq_{ij}\left(\frac{\lambda_j}{\lambda_j^2 + \rho_n} + \frac{\rho_n\lambda_j}{(\lambda_j^2 + \rho_n)^2}\right)\sum_{l=1}^np_{lj}\epsilon_l\vert\leq b_n\right)\\
+Prob\left(\max_{i\not\in\mathcal{N}_{b_n}}\vert\theta_i\vert + \max_{i\not\in\mathcal{N}_{b_n}}\rho_n^2\vert\sum_{j=1}^r\frac{q_{ij}\zeta_j}{(\lambda_j^2 + \rho_n)^2}\vert+\max_{i\not\in\mathcal{N}_{b_n}}\vert\sum_{j=1}^rq_{ij}\left(\frac{\lambda_j}{\lambda_j^2 + \rho_n} + \frac{\rho_n\lambda_j}{(\lambda_j^2 + \rho_n)^2}\right)\sum_{l=1}^np_{lj}\epsilon_l\vert > b_n\right)
\end{aligned}
\end{equation}
From Cauchy inequality,
\begin{equation}
\begin{aligned}
\max_{i=1,2,...,p}\rho_n^2\vert\sum_{j=1}^r\frac{q_{ij}\zeta_j}{(\lambda_j^2 + \rho_n)^2}\vert\leq \max_{i=1,2,...,p}\rho_n^2\sqrt{\sum_{j=1}^rq_{ij}^2}\times \sqrt{\sum_{j=1}^r\frac{\zeta_j^2}{(\lambda_j^2 + \rho_n)^4}} = O(n^{\alpha_\theta-2\delta})\\
\max_{i=1,2,...,p}\sum_{l=1}^n\left(\sum_{j=1}^rq_{ij}\left(\frac{\lambda_j}{\lambda_j^2 + \rho_n} + \frac{\rho_n\lambda_j}{(\lambda_j^2 + \rho_n)^2}\right)p_{lj}\right)^2 = \max_{i=1,2,...,p}\sum_{j=1}^rq_{ij}^2\left(\frac{\lambda_j}{\lambda_j^2 + \rho_n} + \frac{\rho_n\lambda_j}{(\lambda_j^2 + \rho_n)^2}\right)^2\leq \max_{i=1,2,...,p}\frac{4\sum_{j=1}^rq_{ij}^2}{\lambda_r^2}
\end{aligned}
\label{ETA}
\end{equation}
Therefore, for sufficiently large $n$, from assumption 4 and lemma \ref{lemma_Se}
\begin{equation}
\begin{aligned}
\min_{i\in\mathcal{N}_{b_n}}\vert\theta_i\vert - \max_{i\in\mathcal{N}_{b_n}}\rho_n^2\vert\sum_{j=1}^r\frac{q_{ij}\zeta_j}{(\lambda_j^2 + \rho_n)^2}\vert - b_n>\frac{1}{2}(\frac{1}{c_b} - 1)b_n\\
b_n - \max_{i\not\in\mathcal{N}_{b_n}}\vert\theta_i\vert -  \max_{i\in\mathcal{N}_{b_n}}\rho_n^2\vert\sum_{j=1}^r\frac{q_{ij}\zeta_j}{(\lambda_j^2 + \rho_n)^2}\vert> \frac{1}{2}(1 - c_b)b_n\\
\Rightarrow Prob\left(\widehat{\mathcal{N}}_{b_n}\neq \mathcal{N}_{b_n}\right)\leq \frac{\vert\mathcal{N}_{b_n}\vert\times E\times 2^m}{\lambda_r^m\times (\frac{1}{2}(\frac{1}{c_b} - 1)b_n)^m} + \frac{(p - \vert\mathcal{N}_{b_n}\vert)\times E\times 2^m}{\lambda_r^m\times (\frac{1}{2}(1 - c_b)b_n)^m} = O(n^{\alpha_p + m\nu_b-m\eta})
\end{aligned}
\label{Gamma_Decom}
\end{equation}
Define $\widehat{\gamma} = M\widehat{\theta} = (\widehat{\gamma}_1,...,\widehat{\gamma}_{p_1})^T$ and $\gamma = M\beta = (\gamma_1,...,\gamma_{p_1})^T$. For $\beta = \theta + \theta_\perp$, if $\widehat{\mathcal{N}}_{b_n} = \mathcal{N}_{b_n}$, \eqref{EXPAN} and \eqref{def_cik} imply
\begin{equation}
\begin{aligned}
\max_{i=1,2,...,p_1}\vert\widehat{\gamma}_i - \gamma_i\vert = \max_{i=1,2,...,p_1}\vert\sum_{j\in\mathcal{N}_{b_n}}m_{ij}\widetilde{\theta}_j - \sum_{j\in\mathcal{N}_{b_n}}m_{ij}\theta_j - \sum_{j\not\in\mathcal{N}_{b_n}}m_{ij}\theta_j - \sum_{j=1}^pm_{ij}\theta_{\perp,j}\vert\\
\leq \max_{i=1,2,...,p_1}\rho_n^2\vert\sum_{k=1}^r\frac{c_{ik}\zeta_k}{(\lambda_k^2 + \rho_n)^2}\vert + \max_{i = 1,2,...,p_1}\vert\sum_{k=1}^rc_{ik}\left(\frac{\lambda_k}{\lambda_k^2 + \rho_n} + \frac{\rho_n\lambda_k}{(\lambda^2+\rho_n)^2}\right)\sum_{l=1}^np_{lk}\epsilon_l\vert\\
+ \max_{i=1,2,...,p_1}\vert\sum_{j\not\in\mathcal{N}_{b_n}}m_{ij}\theta_j\vert + \max_{i=1,2,...,p_1}\vert\sum_{j=1}^pm_{ij}\theta_{\perp,j}\vert
\end{aligned}
\label{GammaExpan}
\end{equation}
From \eqref{def_cik} and assumption 5, if $i\not\in\mathcal{M}$, then $c_{ik} = 0$ for $k=1,2,...,r$, so from Cauchy inequality and lemma \ref{lemma_Se},
\begin{equation}
\begin{aligned}
\max_{i=1,2,...,p_1}\rho_n^2\vert\sum_{k=1}^r\frac{c_{ik}\zeta_k}{(\lambda_k^2 + \rho_n)^2}\vert\leq \max_{i\in\mathcal{M}}\rho_n^2\sqrt{\sum_{k=1}^rc_{ik}^2}\times\sqrt{\sum_{k=1}^r\frac{\zeta_k^2}{(\lambda_k^2+\rho_n)^4}}\leq \sqrt{C_{\mathcal{M}}}\rho_n^2\times \frac{\Vert\theta\Vert_2}{\lambda_r^4} = O(n^{\alpha_\theta - 2\delta})\\
\max_{i\in\mathcal{M}}\sum_{l = 1}^n\left(\sum_{k=1}^rc_{ik}p_{lk}\left(\frac{\lambda_k}{\lambda_k^2 + \rho_n} + \frac{\rho_n\lambda_k}{(\lambda_k^2 +\rho_n)^2}\right)\right)^2 = \max_{i\in\mathcal{M}}\sum_{k=1}^rc^2_{ik}\left(\frac{\lambda_k}{\lambda_k^2 + \rho_n} + \frac{\rho_n\lambda_k}{(\lambda_k^2 + \rho_n)^2}\right)^2\leq \frac{4C_\mathcal{M}}{\lambda_r^2}\\
\Rightarrow Prob\left(\max_{i = 1,2,...,p_1}\vert\sum_{k=1}^rc_{ik}\left(\frac{\lambda_k}{\lambda_k^2 + \rho_n} + \frac{\rho_n\lambda_k}{(\lambda^2+\rho_n)^2}\right)\sum_{l=1}^np_{lk}\epsilon_l\vert > \delta\right)\leq \frac{\vert\mathcal{M}\vert\times E\times 2^mC_\mathcal{M}^{m/2}}{\lambda_r^m\delta^m}\ \ \text{for $\forall \delta>0$}\\
\Rightarrow \max_{i = 1,2,...,p_1}\vert\sum_{k=1}^rc_{ik}\left(\frac{\lambda_k}{\lambda_k^2 + \rho_n} + \frac{\rho_n\lambda_k}{(\lambda^2+\rho_n)^2}\right)\sum_{l=1}^np_{lk}\epsilon_l\vert = O_p(\vert\mathcal{M}\vert^{1/m}\times n^{-\eta})
\end{aligned}
\label{Delta2}
\end{equation}
Here $E$ is the constant defined in lemma \ref{lemma_Se}. Combine with assumption 2, assumption 5, and \eqref{Gamma_Decom}, we prove \eqref{SELConst}.

If $\widehat{\mathcal{N}}_{b_n} = \mathcal{N}_{b_n}$, since $X\beta = X\theta$, we have
\begin{equation}
\begin{aligned}
\widehat{\sigma}^2 - \sigma^2 = \frac{1}{n}\sum_{i=1}^n\left(\epsilon_i - \sum_{j\in\mathcal{N}_{b_n}}x_{ij}(\widetilde{\theta}_j - \theta_j) + \sum_{j\not\in\mathcal{N}_{b_n}}x_{ij}\theta_j\right)^2 - \sigma^2\\
=\frac{1}{n}\sum_{i=1}^n\epsilon_i^2 - \sigma^2 + \frac{1}{n}\sum_{i=1}^n\left(\sum_{j\in\mathcal{N}_{b_n}}x_{ij}(\widetilde{\theta}_j - \theta_j)\right)^2 + \frac{1}{n}\sum_{i=1}^n\left(\sum_{j\not\in\mathcal{N}_{b_n}}x_{ij}\theta_j\right)^2 - \frac{2}{n}\sum_{i=1}^n\sum_{j\in\mathcal{N}_{b_n}}\epsilon_ix_{ij}(\widetilde{\theta}_j - \theta_j)\\
 +\frac{2}{n}\sum_{i=1}^n\sum_{j\not\in\mathcal{N}_{b_n}}\epsilon_ix_{ij}\theta_j - \frac{2}{n}\sum_{i=1}^n\left(\sum_{j\in\mathcal{N}_{b_n}}x_{ij}(\widetilde{\theta}_j - \theta_j)\right)\times\left(\sum_{j\not\in\mathcal{N}_{b_n}}x_{ij}\theta_j\right)
\end{aligned}
\end{equation}
From assumption 3, $\mathbf{E}\left(\frac{1}{n}\sum_{i=1}^n\epsilon_i^2 - \sigma^2\right)^2 \leq \frac{2}{n}(\mathbf{E}\epsilon_1^4 +\sigma^4) = O(1/n)\Rightarrow\frac{1}{n}\sum_{i=1}^n\epsilon_i^2 - \sigma^2 = O_p(1/\sqrt{n})$. For the second term, from assumption 1 and \eqref{ETA},
\begin{equation}
\begin{aligned}
\frac{1}{n}\sum_{i=1}^n\left(\sum_{j\in\mathcal{N}_{b_n}}x_{ij}(\widetilde{\theta}_j - \theta_j)\right)^2\leq C_\lambda^2\sum_{j\in\mathcal{N}_{b_n}}(\widetilde{\theta}_j - \theta_j)^2\\
\leq 2C_\lambda^2\sum_{j\in\mathcal{N}_{b_n}}\left(\rho_n^4\left(\sum_{k=1}^r\frac{q_{jk}\zeta_k}{(\lambda_k^2+\rho_n)^2}\right)^2 + \left(\sum_{k=1}^rq_{jk}\left(\frac{\lambda_k}{\lambda_k^2 + \rho_n} + \frac{\rho_n\lambda_k}{(\lambda_k^2 + \rho_n)^2}\right)\sum_{l=1}^np_{lk}\epsilon_l\right)^2\right)\\
= O(\vert\mathcal{N}_{b_n}\vert\times n^{2\alpha_\theta - 4\delta}) + 2C_\lambda^2\sum_{j\in\mathcal{N}_{b_n}}\left(\sum_{k=1}^rq_{jk}\left(\frac{\lambda_k}{\lambda_k^2 + \rho_n} + \frac{\rho_n\lambda_k}{(\lambda_k^2 + \rho_n)^2}\right)\sum_{l=1}^np_{lk}\epsilon_l\right)^2
\end{aligned}
\label{SIGMA}
\end{equation}
Since
\begin{equation}
\begin{aligned}
\mathbf{E}\sum_{j\in\mathcal{N}_{b_n}}\left(\sum_{k=1}^rq_{jk}\left(\frac{\lambda_k}{\lambda_k^2 + \rho_n} + \frac{\rho_n\lambda_k}{(\lambda_k^2 + \rho_n)^2}\right)\sum_{l=1}^np_{lk}\epsilon_l\right)^2 = \sigma^2\sum_{j\in\mathcal{N}_{b_n}}\sum_{l=1}^n\left(\sum_{k=1}^rq_{jk}\left(\frac{\lambda_k}{\lambda_k^2 + \rho_n} + \frac{\rho_n\lambda_k}{(\lambda_k^2 + \rho_n)^2}\right)p_{lk}\right)^2\\
=\sigma^2\sum_{j\in\mathcal{N}_{b_n}}\sum_{k=1}^rq_{jk}^2\left(\frac{\lambda_k}{\lambda_k^2 + \rho_n} + \frac{\rho_n\lambda_k}{(\lambda_k^2 + \rho_n)^2}\right)^2\leq \frac{4\sigma^2\vert\mathcal{N}_{b_n}\vert}{\lambda_r^2}
\end{aligned}
\label{SIGMA2}
\end{equation}
We have $\frac{1}{n}\sum_{i=1}^n\left(\sum_{j\in\mathcal{N}_{b_n}}x_{ij}(\widetilde{\theta}_j - \theta_j)\right)^2 = O_p(\vert\mathcal{N}_{b_n}\vert\times n^{2\alpha_\theta - 4\delta} + \vert\mathcal{N}_{b_n}\vert\times n^{-2\eta})$. For the third term, from assumption 6. we have
\begin{equation}
\begin{aligned}
\frac{1}{n}\sum_{i=1}^n\left(\sum_{j\not\in\mathcal{N}_{b_n}}x_{ij}\theta_j\right)^2\leq C_\lambda^2\sum_{j\not\in\mathcal{N}_{b_n}}\theta_j^2\leq C_\lambda^2\times b_n\sum_{j\not\in\mathcal{N}_{b_n}}\vert\theta_j\vert = O(n^{-\alpha_\sigma})
\end{aligned}
\label{XTHETA}
\end{equation}
For the fourth term, from Cauchy inequality and \eqref{SIGMA},
\begin{equation}
\begin{aligned}
\mathbf{E}\frac{1}{n}\vert\sum_{i=1}^n\sum_{j\in\mathcal{N}_{b_n}}\epsilon_ix_{ij}(\widetilde{\theta}_j - \theta_j)\vert\leq \frac{1}{n}\mathbf{E}\sqrt{\sum_{i=1}^n\epsilon_i^2}\times \sqrt{\sum_{i=1}^n(\sum_{j\in\mathcal{N}_{b_n}}x_{ij}(\widetilde{\theta}_j-\theta_j))^2}\\
\leq \sqrt{\frac{\mathbf{E}\sum_{i=1}^n\epsilon_i^2}{n}}\times\sqrt{\frac{1}{n}\mathbf{E}\sum_{i=1}^n(\sum_{j\in\mathcal{N}_{b_n}}x_{ij}(\widetilde{\theta}_j-\theta_j))^2}
=\sigma \times O(\sqrt{\vert\mathcal{N}_{b_n}\vert\times n^{2\alpha_\theta - 4\delta} + \vert\mathcal{N}_{b_n}\vert\times n^{-2\eta}})\\
\Rightarrow \frac{1}{n}\vert\sum_{i=1}^n\sum_{j\in\mathcal{N}_{b_n}}\epsilon_ix_{ij}(\widetilde{\theta}_j - \theta_j)\vert = O_p(\sqrt{\vert\mathcal{N}_{b_n}\vert}\times n^{\alpha_\theta - 2\delta} + \sqrt{\vert\mathcal{N}_{b_n}\vert}\times n^{-\eta})
\end{aligned}
\end{equation}
For the fifth term,
\begin{equation}
\begin{aligned}
\mathbf{E}\vert\frac{1}{n}\sum_{i=1}^n\sum_{j\not\in\mathcal{N}_{b_n}}\epsilon_ix_{ij}\theta_j\vert^2 = \frac{\sigma^2}{n^2}\sum_{i=1}^n\left(\sum_{j\not\in\mathcal{N}_{b_n}}x_{ij}\theta_j\right)^2\leq \frac{\sigma^2C_\lambda^2}{n}\sum_{j\not\in\mathcal{N}_{b_n}}\theta_j^2\Rightarrow \frac{1}{n}\sum_{i=1}^n\sum_{j\not\in\mathcal{N}_{b_n}}\epsilon_ix_{ij}\theta_j = O_p(n^{-(1+\alpha_\sigma)/2})
\end{aligned}
\end{equation}
For the last term,
\begin{equation}
\begin{aligned}
\frac{1}{n}\vert\sum_{i=1}^n\left(\sum_{j\in\mathcal{N}_{b_n}}x_{ij}(\widetilde{\theta}_j - \theta_j)\right)\times\left(\sum_{j\not\in\mathcal{N}_{b_n}}x_{ij}\theta_j\right)\vert\leq C_\lambda^2\sqrt{\sum_{j\in\mathcal{N}_{b_n}}(\widetilde{\theta}_j - \theta_j)^2}\times\sqrt{\sum_{j\not\in\mathcal{N}_{b_n}}\theta_j^2}\\
= O_p(\sqrt{\vert\mathcal{N}_{b_n}\vert}\times n^{\alpha_\theta - 2\delta-\alpha_\sigma/2} + \sqrt{\vert\mathcal{N}_{b_n}\vert}\times n^{-\eta-\alpha_\sigma/2})
\end{aligned}
\end{equation}
From \eqref{SELConst}, $Prob(\widehat{\mathcal{N}}_{b_n}\neq \mathcal{N}_{b_n})\to 0$. So we have
\begin{equation}
\widehat{\sigma}^2 - \sigma^2 = O_p\left(\frac{1}{\sqrt{n}} + \sqrt{\vert\mathcal{N}_{b_n}\vert}\times n^{\alpha_\theta - 2\delta} + \sqrt{\vert\mathcal{N}_{b_n}\vert}\times n^{-\eta} + n^{-\alpha_\sigma}\right)
\end{equation}
From assumption 2 and 6, we prove the second result.
\end{proof}

Define $T = (c_{ik})_{i\in\mathcal{M},k=1,2,...,r}$. From assumption 7, since the matrix $\left(\frac{1}{\tau_i}c_{ik}\left(\frac{\lambda_k}{\lambda_k^2 + \rho_n} + \frac{\rho_n\lambda_k}{(\lambda_k^2 + \rho_n)^2}\right)\right)_{i\in\mathcal{M},j=1,2,...,r} = D_1 T D_2$ with $D_1 = diag(1/\tau_i,i\in\mathcal{M})$ and $D_2 = diag\left(\frac{\lambda_1}{\lambda_1^2 + \rho_n} + \frac{\rho_n\lambda_1}{(\lambda_1^2 + \rho_n)^2},...,\frac{\lambda_r}{\lambda_r^2 + \rho_n} + \frac{\rho_n\lambda_r}{(\lambda_r^2 + \rho_n)^2}\right)$,
the matrix

\noindent$\left(\frac{1}{\tau_i}c_{ik}\left(\frac{\lambda_k}{\lambda_k^2 + \rho_n} + \frac{\rho_n\lambda_k}{(\lambda_k^2 + \rho_n)^2}\right)\right)_{i\in\mathcal{M},j=1,2,...,r}$ also has rank $\vert\mathcal{M}\vert$. The proof of theorem \ref{Thm2} uses this result.

\begin{proof}[Proof of theorem \ref{Thm2}]
From Cauchy inequality and assumption 2, suppose $\delta = \frac{\eta + \alpha_\theta+\delta_1}{2}$ with $\delta_1 > 0$. For $i\in\mathcal{M}$,
\begin{equation}
\begin{aligned}
\vert\sum_{k=1}^r\frac{c_{ik}\zeta_k}{(\lambda_k^2 + \rho_n)^2}\vert\leq \sqrt{\sum_{k=1}^r\frac{c_{ik}^2\lambda_k^2}{(\lambda_k^2+\rho_n)^2}}\times \sqrt{\sum_{k=1}^r\frac{\zeta_k^2}{\lambda_k^2(\lambda_k^2 + \rho_n)^2}}\leq \tau_i\times \frac{\Vert\theta\Vert_2}{\lambda_r^3}
\Rightarrow \max_{i\in\mathcal{M}}\frac{\rho_n^2}{\tau_i}\vert\sum_{k=1}^r\frac{c_{ik}\zeta_k}{(\lambda_k^2 + \rho_n)^2}\vert = O(n^{-\delta_1})
\end{aligned}
\label{DIVI}
\end{equation}
Define $t_{il} = \frac{1}{\tau_i}\times \sum_{k=1}^rc_{ik}p_{lk}\left(\frac{\lambda_k}{\lambda_k^2 + \rho_n} + \frac{\rho_n\lambda_k}{(\lambda_k^2 + \rho_n)^2}\right)$ for $i\in\mathcal{M}$ and $l=1,2,...,n$. From \eqref{EXPAN}, \eqref{DefVar}, \eqref{GammaExpan} and assumption 5, if $\widehat{\mathcal{N}}_{b_n} = \mathcal{N}_{b_n}$, we have $\widehat{\tau}_i = \tau_i\geq 1/\sqrt{n}$ and $\exists$ a constant $C>0$, for any $a > 0$ and sufficiently large $n$,
\begin{equation}
\begin{aligned}
\max_{i=1,2,...,p_1}\frac{\vert\widehat{\gamma}_i - \gamma_i\vert}{\widehat{\tau}_i}
\leq \max_{i\in\mathcal{M}}\frac{\rho_n^2}{\tau_i}\vert\sum_{k=1}^r \frac{c_{ik}\zeta_k}{(\lambda_k^2 + \rho_n)^2}\vert + \max_{i\in\mathcal{M}}\vert\sum_{l=1}^nt_{il}\epsilon_l\vert +\max_{i=1,2,...,p_1}\frac{\vert\sum_{j\not\in\mathcal{N}_{b_n}}m_{ij}\theta_j\vert}{\tau_i} + \max_{i=1,2,...,p_1}\frac{\vert\sum_{j=1}^p m_{ij}\theta_{\perp,j}\vert}{\tau_i}\\
\leq \max_{i\in\mathcal{M}}\vert\sum_{l=1}^nt_{il}\epsilon_l\vert + Cn^{-\delta_1} + \frac{a}{\sqrt{\log(n)}}\\
\max_{i=1,2,...,p_1}\frac{\vert\widehat{\gamma}_i - \gamma_i\vert}{\widehat{\tau}_i}\geq \max_{i\in\mathcal{M}}\vert\sum_{l=1}^nt_{il}\epsilon_l\vert -\max_{i\in\mathcal{M}}\frac{\rho_n^2}{\tau_i}\vert\sum_{k=1}^r \frac{c_{ik}\zeta_k}{(\lambda_k^2 + \rho_n)^2}\vert - \max_{i=1,2,...,p_1}\frac{\vert\sum_{j\not\in\mathcal{N}_{b_n}}m_{ij}\theta_j\vert}{\tau_i} -  \max_{i=1,2,...,p_1}\frac{\vert\sum_{j=1}^p m_{ij}\theta_{\perp,j}\vert}{\tau_i}\\
\geq \max_{i\in\mathcal{M}}\vert\sum_{l=1}^nt_{il}\epsilon_l\vert - Cn^{-\delta_1} - \frac{a}{\sqrt{\log(n)}}\\
\end{aligned}
\end{equation}

According to theorem \ref{THM_CON} and lemma \ref{thm1}, $\exists$ a constant $C$ and for any given $a>0$, for sufficiently large $n$ and any $x\geq 0$,
\begin{equation}
\begin{aligned}
Prob\left(\max_{i=1,2,...,p_1}\frac{\vert\widehat{\gamma}_i - \gamma_i\vert}{\widehat{\tau}_i}\leq x\right)\leq Prob\left(\max_{i=1,2,...,p_1}\frac{\vert\widehat{\gamma}_i - \gamma_i\vert}{\widehat{\tau}_i}\leq x\cap\widehat{\mathcal{N}}_{b_n} = \mathcal{N}_{b_n}\right) + Prob\left(\widehat{\mathcal{N}}_{b_n} \neq \mathcal{N}_{b_n}\right)\\
\leq Prob\left(\max_{i\in\mathcal{M}}\vert\sum_{l=1}^nt_{il}\epsilon_l\vert\leq x +  Cn^{-\delta_1} + \frac{a}{\sqrt{\log(n)}}\right) + Cn^{\alpha_p + m\nu_b - m\eta}\\
\leq Prob\left(\max_{i\in\mathcal{M}}\vert\sum_{l=1}^nt_{il}\epsilon_l^*\vert\leq x\right)
+ Cn^{\alpha_p + m\nu_b - m\eta} + \sup_{x\geq 0}\vert Prob\left(\max_{i\in\mathcal{M}}\vert\sum_{l=1}^nt_{il}\epsilon_l\vert\leq x\right) - Prob\left(\max_{i\in\mathcal{M}}\vert\sum_{l=1}^nt_{il}\epsilon_l^*\vert\leq x\right)\vert\\
+\sup_{x\in\mathbf{R}}\left( Prob\left(\max_{i\in\mathcal{M}}\vert\sum_{l=1}^nt_{il}\epsilon_l^*\vert\leq x +  Cn^{-\delta_1} + \frac{a}{\sqrt{\log(n)}}\right) - Prob\left(\max_{i\in\mathcal{M}}\vert\sum_{l=1}^nt_{il}\epsilon_l^*\vert\leq x\right)\right)\\
Prob\left(\max_{i=1,2,...,p_1}\frac{\vert\widehat{\gamma}_i - \gamma_i\vert}{\widehat{\tau}_i}\leq x\right)\geq Prob\left(\max_{i=1,2,...,p_1}\frac{\vert\widehat{\gamma}_i - \gamma_i\vert}{\widehat{\tau}_i}\leq x\cap\widehat{\mathcal{N}}_{b_n} = \mathcal{N}_{b_n}\right)\\
\geq Prob\left(\max_{i\in\mathcal{M}}\vert\sum_{l=1}^nt_{il}\epsilon_l\vert\leq x -  Cn^{-\delta_1} - \frac{a}{\sqrt{\log(n)}}\right) - Prob\left(\widehat{\mathcal{N}}_{b_n} \neq \mathcal{N}_{b_n}\right)\\
\geq Prob\left(\max_{i\in\mathcal{M}}\vert\sum_{l=1}^nt_{il}\epsilon_l^*\vert\leq x \right) - Cn^{\alpha_p + m\nu_b - m\eta} \\
- \sup_{x\in\mathbf{R}}\left(Prob\left(\max_{i\in\mathcal{M}}\vert\sum_{l=1}^nt_{il}\epsilon_l^*\vert\leq x\right) - Prob\left(\max_{i\in\mathcal{M}}\vert\sum_{l=1}^nt_{il}\epsilon_l^*\vert\leq x -  Cn^{-\delta_1} - \frac{a}{\sqrt{\log(n)}}\right)\right) \\
- \vert Prob\left(\max_{i\in\mathcal{M}}\vert\sum_{l=1}^nt_{il}\epsilon_l\vert\leq x -  Cn^{-\delta_1} - \frac{a}{\sqrt{\log(n)}}\right) - Prob\left(\max_{i\in\mathcal{M}}\vert\sum_{l=1}^nt_{il}\epsilon_l^*\vert\leq x -  Cn^{-\delta_1} - \frac{a}{\sqrt{\log(n)}}\right)\vert
\end{aligned}
\label{Jp1}
\end{equation}
From assumption 1, 2, 5 and 7, for sufficiently large $n$ we have
\begin{equation}
\begin{aligned}
\max_{i\in\mathcal{M}}\mathbf{E}\left(\sum_{l=1}^nt_{il}\epsilon_l^*\right)^2 = \sigma^2\max_{i\in\mathcal{M}}\sum_{l=1}^n t_{il}^2 = \sigma^2\max_{i\in\mathcal{M}}\frac{\sum_{k=1}^rc_{ik}^2\left(\frac{\lambda_k}{\lambda_k^2 + \rho_n} + \frac{\rho_n\lambda_k}{(\lambda_k^2 + \rho_n)^2}\right)^2}{\tau_i^2}\leq \sigma^2\\
\min_{i\in\mathcal{M}}\mathbf{E}\left(\sum_{l=1}^nt_{il}\epsilon_l^*\right)^2 =  \sigma^2\min_{i\in\mathcal{M}}\frac{1}{1 + \frac{1}{n\sum_{k=1}^rc_{ik}^2\left(\frac{\lambda_k}{\lambda_k^2 + \rho_n} + \frac{\rho_n\lambda_k}{(\lambda_k^2 + \rho_n)^2}\right)^2}}\geq \sigma^2\min_{i\in\mathcal{M}}\frac{1}{1 + \frac{1}{n\sum_{k=1}^rc_{ik}^2\frac{\lambda_k^2}{(\lambda_k^2 + \rho_n)^2}}}\geq \frac{\sigma^2}{1 + \frac{4C_\lambda^2}{c_\mathcal{M}}} > 0\\
\end{aligned}
\label{SigmaBound}
\end{equation}
and $(t_{il})_{i\in\mathcal{M},l=1,2,...,n} = D_1 T D_2 P^T$, here $T = (c_{ik})_{i\in\mathcal{M},k=1,2,...,r}$, $D_1 = diag(1/\tau_i,i\in\mathcal{M})$, and

\noindent$D_2 = diag\left(\frac{\lambda_1}{\lambda_1^2 + \rho_n} + \frac{\rho_n\lambda_1}{(\lambda_1^2 + \rho_n)^2},...,\frac{\lambda_r}{\lambda_r^2 + \rho_n} + \frac{\rho_n\lambda_r}{(\lambda_r^2 + \rho_n)^2}\right)$. So $(t_{il})_{i\in\mathcal{M},l=1,2,...,n}$ has full rank(rank $\vert\mathcal{M}\vert$). From lemma \ref{CI_lemma_1}, $\exists$ a constant $C^{'}$ which only depends on $\sigma,c_\mathcal{M},C_\lambda$ such that
\begin{equation}
\begin{aligned}
\sup_{x\in\mathbf{R}}\left(Prob\left(\max_{i\in\mathcal{M}}\vert\sum_{l=1}^nt_{il}\epsilon_l^*\vert\leq x + Cn^{-\delta_1} + \frac{a}{\sqrt{\log(n)}}\right) - Prob\left(\max_{i\in\mathcal{M}}\vert\sum_{l=1}^nt_{il}\epsilon_l^*\vert\leq x\right)\right)\\
\leq C^{'}\left(Cn^{-\delta_1} + \frac{a}{\sqrt{\log(n)}}\right)\times\left(1 + \sqrt{\log(\vert\mathcal{M}\vert)} + \sqrt{\vert\log(Cn^{-\delta_1} + \frac{a}{\sqrt{\log(n)}})\vert}\right)
\end{aligned}
\end{equation}
For sufficiently large $n$, we have $Cn^{-\delta_1} + \frac{a}{\sqrt{\log(n)}} < 1$ and
\begin{equation}
\begin{aligned}
\vert\log(Cn^{-\delta_1} + \frac{a}{\sqrt{\log(n)}})\vert\leq \log(\frac{\sqrt{\log(n)}}{a}) = \frac{\log(\log(n))}{2} - \log(a) \leq \log(\log(n))\\
\Rightarrow \sup_{x\in\mathbf{R}}\left(Prob\left(\max_{i\in\mathcal{M}}\vert\sum_{l=1}^nt_{il}\epsilon_l^*\vert\leq x + Cn^{-\delta_1} + \frac{a}{\sqrt{\log(n)}}\right) - Prob\left(\max_{i\in\mathcal{M}}\vert\sum_{l=1}^nt_{il}\epsilon_l^*\vert\leq x\right)\right)\\
\leq C^{'}\left(Cn^{-\delta_1} + \frac{a}{\sqrt{\log(n)}}\right)\times \left(1 + \sqrt{\log(n)} + \sqrt{\log(\log(n))}\right)\leq 6C^{'}a
\end{aligned}
\end{equation}
From assumption 7, \eqref{SigmaBound} and lemma \ref{thm_1}, we have
\begin{equation}
\sup_{x\geq 0}\vert Prob\left(\max_{i\in\mathcal{M}}\vert\sum_{l=1}^nt_{il}\epsilon_l\vert\leq x\right) - Prob\left(\max_{i\in\mathcal{M}}\vert\sum_{l=1}^nt_{il}\epsilon_l^*\vert\leq x\right)\vert < a\ \text{for sufficiently large $n$}
\label{Jp6}
\end{equation}
If $x < Cn^{-\delta_1} + \frac{a}{\sqrt{\log(n)}}$, then $Prob\left(\max_{i\in\mathcal{M}}\vert\sum_{l=1}^nt_{il}\epsilon_l\vert\leq x -  Cn^{-\delta_1} - \frac{a}{\sqrt{\log(n)}}\right) = 0$ and

\noindent $Prob\left(\max_{i\in\mathcal{M}}\vert\sum_{l=1}^nt_{il}\epsilon_l^*\vert\leq x -  Cn^{-\delta_1} - \frac{a}{\sqrt{\log(n)}}\right) = 0$. Combine with \eqref{Jp1} to \eqref{Jp6}, we have
\begin{equation}
\sup_{x\geq 0}\vert Prob\left(\max_{i=1,2,...,p_1}\frac{\vert\widehat{\gamma}_i - \gamma_i\vert}{\widehat{\tau}_i}\leq x\right) - Prob\left(\max_{i\in\mathcal{M}}\vert\sum_{l=1}^nt_{il}\epsilon_l^*\vert\leq x\right)\vert\leq Cn^{\alpha_p + m\nu_b - m\eta} + 6C^{'}a + a
\end{equation}
and we prove \eqref{eqCLT}.
\end{proof}
Define $c_{1-\alpha}$ as the $1-\alpha$ quantile of $H$. The density of a multivariate normal random variable with a full rank covariance matrix is positive for $\forall x\in\mathcal{R}^{\vert\mathcal{M}\vert}$. And $\forall x\geq 0,\ \delta > 0$, the set $\{t = (t_i , i\in\mathcal{M})|\ x<\max_{i=1,2,...,\vert\mathcal{M}\vert}\vert t_i\vert\leq x + \delta\}$ has positive Lebesgue measure. Therefore, $H(x)$ is strictly increasing, and for any $0<\alpha < 1$, $H(c_{1-\alpha}) = 1 -\alpha$. From theorem \ref{Thm2}, for any given $0<\alpha_0<\alpha_1<1$,
\begin{equation}
\sup_{\alpha_0\leq\alpha\leq \alpha_1}\vert Prob\left(\max_{i=1,2,...,p_1}\frac{\vert\widehat{\gamma}_i - \gamma_i\vert}{\widehat{\tau}_i}\leq c_{1-\alpha}\right) - (1-\alpha)\vert\leq \sup_{x\geq 0}\vert Prob\left(\max_{i=1,2,...,p_1}\frac{\vert\widehat{\gamma}_i - \gamma_i\vert}{\widehat{\tau}_i}\leq x\right) - H(x)\vert\to 0
\end{equation}
as $n\to\infty$.

\section{Proofs of theorems in section \ref{BOOT}}

\begin{proof}[Proof of theorem \ref{BootCons}]
According to theorem \ref{thm1}, $Prob\left(\widehat{\mathcal{N}}_{b_n}\neq \mathcal{N}_{b_n}\right) = O(n^{\alpha_p + m\nu_b - m\eta})$. If $\widehat{\mathcal{N}}_{b_n} = \mathcal{N}_{b_n}$, from \eqref{EXPAN}
\begin{equation}
\begin{aligned}
\Vert\widehat{\theta}\Vert_2^2 = \sum_{i\in\mathcal{N}_{b_n}}\widetilde{\theta}_i^2\leq 3\sum_{i\in\mathcal{N}_{b_n}}\vert\theta_i\vert^2 + 3\rho_n^4\sum_{i\in\mathcal{N}_{b_n}}\left(\sum_{j=1}^r\frac{q_{ij}\zeta_j}{(\lambda_j^2 + \rho_n)^2}\right)^2 + 3\sum_{i\in\mathcal{N}_{b_n}}\left(\sum_{j=1}^r\sum_{l=1}^nq_{ij}\left(\frac{\lambda_j}{\lambda_j^2 + \rho_n} + \frac{\rho_n\lambda_j}{(\lambda_j^2 + \rho_n)^2}\right)p_{lj}\epsilon_l\right)^2\\
\end{aligned}
\label{Var1}
\end{equation}
From assumption 2, $\sum_{i\in\mathcal{N}_{b_n}}\vert\theta_i\vert^2\leq \Vert\theta\Vert_2^2 = O(n^{2\alpha_\theta})$. Similarly
\begin{equation}
\begin{aligned}
\rho_n^4\sum_{i\in\mathcal{N}_{b_n}}\left(\sum_{j=1}^r\frac{q_{ij}\zeta_j}{(\lambda_j^2 + \rho_n)^2}\right)^2\leq\frac{\rho_n^4}{\lambda_r^8}\sum_{i\in\mathcal{N}_{b_n}}\sum_{j=1}^rq_{ij}^2\sum_{j=1}^r\zeta_j^2 = \frac{\rho_n^4\times\vert\mathcal{N}_{b_n}\vert\times\Vert\theta\Vert_2^2}{\lambda_r^8} = o(n^{-2\alpha_\sigma})
\end{aligned}
\end{equation}

From assumption 6,
\begin{equation}
\begin{aligned}
\mathbf{E}\sum_{i\in\mathcal{N}_{b_n}}\left(\sum_{j=1}^r\sum_{l=1}^nq_{ij}\left(\frac{\lambda_j}{\lambda_j^2 + \rho_n} + \frac{\rho_n\lambda_j}{(\lambda_j^2 + \rho_n)^2}\right)p_{lj}\epsilon_l\right)^2 = \sigma^2\sum_{i\in\mathcal{N}_{b_n}}\sum_{l=1}^n\left(\sum_{j=1}^r q_{ij}\left(\frac{\lambda_j}{\lambda_j^2 + \rho_n} + \frac{\rho_n\lambda_j}{(\lambda_j^2 + \rho_n)^2}\right)p_{lj}\right)^2\\
=\sigma^2 \sum_{i\in\mathcal{N}_{b_n}}\sum_{j=1}^r q_{ij}^2\left(\frac{\lambda_j}{\lambda_j^2 + \rho_n} + \frac{\rho_n\lambda_j}{(\lambda_j^2 + \rho_n)^2}\right)^2\leq \frac{4\sigma^2\vert\mathcal{N}_{b_n}\vert}{\lambda_r^2}
\Rightarrow \sum_{i\in\mathcal{N}_{b_n}}\left(\sum_{j=1}^r\sum_{l=1}^nq_{ij}\left(\frac{\lambda_j}{\lambda_j^2 + \rho_n} + \frac{\rho_n\lambda_j}{(\lambda_j^2 + \rho_n)^2}\right)p_{lj}\epsilon_l\right)^2 = O_p\left(n^{-2\alpha_\sigma}\right)
\end{aligned}
\label{Var3}
\end{equation}

Since $\alpha_\theta,\alpha_\sigma\geq 0$, $\Vert\widehat{\theta}\Vert_2 = O_p(n^{\alpha_\theta})$.
According to \eqref{MainStat} and \eqref{EXPAN}, define $\widehat{\zeta} = Q^T\widehat{\theta}$,
\begin{equation}
\begin{aligned}
\widetilde{\theta}^* - \widehat{\theta} = \left(I_p + \rho_nQ(\Lambda^2 + \rho_nI_r)^{-1}Q^T\right)Q(\Lambda^2 + \rho_n I_r)^{-1}\left(\Lambda^2 Q^T\widehat{\theta} + \Lambda P^T\epsilon^*\right) + \widehat{\theta}_\perp - QQ^T\widehat{\theta} - Q_\perp Q^T_\perp \widehat{\theta}\\
\Rightarrow \widetilde{\theta}^*_i - \widehat{\theta}_i = -\rho_n^2\sum_{j=1}^r\frac{q_{ij}\widehat{\zeta}_j}{(\lambda_j^2 + \rho_n)^2} + \sum_{j = 1}^r\sum_{l=1}^n q_{ij}\left(\frac{\lambda_j}{\lambda_j^2 + \rho_n} + \frac{\rho_n\lambda_j}{(\lambda_j^2 + \rho_n)^2}\right)p_{lj}\epsilon^*_l
\end{aligned}
\end{equation}
Similar to \eqref{Delta2}, rewrite $\delta$ in assumption 2 as $\delta = \frac{\eta + \alpha_\theta + \delta_1}{2}$ with $\delta_1 > 0$, we have
\begin{equation}
\max_{i=1,2,...,p}\vert\rho_n^2\sum_{j=1}^r\frac{q_{ij}\widehat{\zeta}_j}{(\lambda_j^2 + \rho_n)^2}\vert\leq \max_{i=1,2,...,p}\frac{\rho_n^2}{\lambda_r^4}\sqrt{\sum_{j=1}^r q_{ij}^2}\times \sqrt{\sum_{j=1}^r\widehat{\zeta}_j^2}\leq \frac{\rho_n^2\Vert\widehat{\theta}\Vert_2}{c_\lambda^4n^{4\eta}} = O_p\left(n^{-\eta-\delta_1}\right)
\end{equation}
$\epsilon^*_i|\epsilon,i=1,2,...,n$ are normal random variables with mean $0$ and variance $\widehat{\sigma}^2$. Therefore $\mathbf{E}^*\vert\epsilon^*_1\vert^m = \widehat{\sigma}^mD$, $D = \mathbf{E}\vert Y\vert^m$ with $Y$ a normal random variable with mean $0$ and variance $1$. If $\widehat{\sigma} > 0$, from \eqref{ETA} and lemma \ref{lemma_Se}, $\exists$ a constant $E$ which depends on $m$ and $D$ such that for any $a>0$,
\begin{equation}
\begin{aligned}
Prob^*\left(\max_{i=1,2,...,p}\vert \sum_{j = 1}^r\sum_{l=1}^n q_{ij}\left(\frac{\lambda_j}{\lambda_j^2 + \rho_n} + \frac{\rho_n\lambda_j}{(\lambda_j^2 + \rho_n)^2}\right)p_{lj}\frac{\epsilon^*_l}{\widehat{\sigma}}\vert > \frac{a}{\widehat{\sigma}}\right)\leq \frac{pE\widehat{\sigma}^m}{\lambda_r^ma^m}
\end{aligned}
\end{equation}
Suppose $\widehat{\mathcal{N}}_{b_n} = \mathcal{N}_{b_n}$, $\frac{\sigma}{2}<\widehat{\sigma} < \frac{3\sigma}{2}$, and $\max_{i=1,2,...,p}\vert\rho_n^2\sum_{j=1}^r\frac{q_{ij}\widehat{\zeta}_j}{(\lambda_j^2 + \rho_n)^2}\vert\leq C\times n^{-\eta-\delta_1}$ for a constant $C$. Since $\widehat{\theta}_i = 0$ if $i\not\in\widehat{\mathcal{N}}_{b_n}$,
\begin{equation}
\begin{aligned}
Prob^*\left(\widehat{\mathcal{N}}_{b_n}^*\neq\mathcal{N}_{b_n}\right)\leq Prob^*\left(\min_{i\in\mathcal{N}_{b_n}}\vert\widetilde{\theta}_i^*\vert\leq b_n\right) + Prob^*\left(\max_{i\not\in\mathcal{N}_{b_n}}\vert\widetilde{\theta}_i^*\vert > b_n\right)\\
\leq Prob^*\left(\min_{i\in\mathcal{N}_{b_n}}\vert\widehat{\theta}_i\vert - \max_{i\in\mathcal{N}_{b_n}}\vert\rho_n^2\sum_{j=1}^r\frac{q_{ij}\widehat{\zeta}_j}{(\lambda_j^2 + \rho_n)^2}\vert - b_n \leq \max_{i\in\mathcal{N}_{b_n}}\vert \sum_{j = 1}^r\sum_{l=1}^n q_{ij}\left(\frac{\lambda_j}{\lambda_j^2 + \rho_n} + \frac{\rho_n\lambda_j}{(\lambda_j^2 + \rho_n)^2}\right)p_{lj}\epsilon^*_l\vert\right)\\
+ Prob^*\left(\max_{i\not\in\mathcal{N}_{b_n}}\vert \sum_{j = 1}^r\sum_{l=1}^n q_{ij}\left(\frac{\lambda_j}{\lambda_j^2 + \rho_n} + \frac{\rho_n\lambda_j}{(\lambda_j^2 + \rho_n)^2}\right)p_{lj}\epsilon^*_l\vert > b_n - \rho_n^2\max_{i\not\in\mathcal{N}_{b_n}}\vert \sum_{j=1}^r\frac{q_{ij}\widehat{\zeta}_j}{(\lambda_j^2 + \rho_n)^2}\vert\right)
\end{aligned}
\end{equation}
From assumption 4, for sufficiently large $n$,
\begin{equation}
\begin{aligned}
b_n - \rho_n^2\max_{i\not\in\mathcal{N}_{b_n}}\vert \sum_{j=1}^r\frac{q_{ij}\widehat{\zeta}_j}{(\lambda_j^2 + \rho_n)^2}\vert \geq C_bn^{-\nu_b} - Cn^{-\eta-\delta_1}\geq \frac{b_n}{2}
\end{aligned}
\end{equation}
From \eqref{ETA}, lemma \ref{lemma_Se}, assumption 1 and 4, we have
\begin{equation}
\begin{aligned}
\max_{i=1,2,...,p}\vert \sum_{j=1}^rq_{ij}\left(\frac{\lambda_j}{\lambda_j^2 + \rho_n} + \frac{\rho_n\lambda_j}{(\lambda_j^2 + \rho_n)^2}\right)\sum_{l=1}^np_{lj}\epsilon_l\vert = O_p\left(n^{\alpha_p/m - \eta}\right)
\end{aligned}
\end{equation}
Suppose a constant $C$ such that $ \max_{i=1,2,...,p}\vert \sum_{j=1}^rq_{ij}\left(\frac{\lambda_j}{\lambda_j^2 + \rho_n} + \frac{\rho_n\lambda_j}{(\lambda_j^2 + \rho_n)^2}\right)\sum_{l=1}^np_{lj}\epsilon_l\vert\leq Cn^{\alpha_p/m - \eta}$(from lemma \ref{lemma_Se}), and (since $\frac{\rho_n^2\Vert\theta\Vert_2}{\lambda_r^4} = O(n^{-\eta-\delta_1})$) $\frac{\rho_n^2\Vert\theta\Vert_2}{\lambda_r^4}\leq Cn^{-\eta-\delta_1}$. From assumption 4, for sufficiently large $n$,
\begin{equation}
\begin{aligned}
\min_{i\in\mathcal{N}_{b_n}}\vert\widehat{\theta}_i\vert\geq \min_{i\in\mathcal{N}_{b_n}}\vert\theta_i\vert - \max_{i\in\mathcal{N}_{b_n}}\rho_n^2\vert\sum_{j=1}^r\frac{q_{ij}\zeta_j}{(\lambda_j^2 + \rho_n)^2}\vert - \max_{i\in\mathcal{N}_{b_n}}\vert\sum_{j=1}^rq_{ij}\left(\frac{\lambda_j}{\lambda_j^2 + \rho_n} + \frac{\rho_n\lambda_j}{(\lambda_j^2 + \rho_n)^2}\right)\sum_{l=1}^np_{lj}\epsilon_l\vert\\
\geq \frac{b_n}{c_b} - \frac{\rho_n^2\Vert\theta\Vert_2}{\lambda_r^4} - Cn^{\alpha_p/m - \eta}
\Rightarrow \min_{i\in\mathcal{N}_{b_n}}\vert\widehat{\theta}_i\vert - \max_{i\in\mathcal{N}_{b_n}}\vert\rho_n^2\sum_{j=1}^r\frac{q_{ij}\widehat{\zeta}_j}{(\lambda_j^2 + \rho_n)^2}\vert - b_n\\
\geq \left(\frac{1}{c_b} - 1\right)b_n - Cn^{-\eta-\delta_1} - Cn^{\alpha_p/m-\eta} - Cn^{-\eta-\delta_1}
> \frac{b_n}{2}\left(\frac{1}{c_b} - 1\right)
\end{aligned}
\end{equation}
Correspondingly
\begin{equation}
\begin{aligned}
Prob^*\left(\widehat{\mathcal{N}}_{b_n}^*\neq\mathcal{N}_{b_n}\right)\leq Prob^*\left(\max_{i\in\mathcal{N}_{b_n}}\vert \sum_{j = 1}^r\sum_{l=1}^n q_{ij}\left(\frac{\lambda_j}{\lambda_j^2 + \rho_n} + \frac{\rho_n\lambda_j}{(\lambda_j^2 + \rho_n)^2}\right)p_{lj}\epsilon^*_l\vert> \frac{b_n}{2}(\frac{1}{c_b} - 1)\right)\\
+Prob^*\left(\max_{i\not\in\mathcal{N}_{b_n}}\vert \sum_{j = 1}^r\sum_{l=1}^n q_{ij}\left(\frac{\lambda_j}{\lambda_j^2 + \rho_n} + \frac{\rho_n\lambda_j}{(\lambda_j^2 + \rho_n)^2}\right)p_{lj}\epsilon^*_l\vert > b_n/2\right)
\leq \frac{pE\widehat{\sigma}^m}{c_\lambda^mn^{m\eta}b_n^m}\times \left(\frac{2^m}{(1/c_b - 1)^m} + 2^m\right)
\end{aligned}
\label{StarVar}
\end{equation}
which has order $O_p(n^{\alpha_p + m\nu_b-m\eta})$. If $\widehat{\mathcal{N}}^*_{b_n} = \mathcal{N}_{b_n}$, then $\widehat{\tau}^*_i = \tau_i$ for $i=1,2,...,p_1$. Similar to \eqref{DIVI},
\begin{equation}
\begin{aligned}
\max_{i=1,2,...,p_1}\frac{\vert\widehat{\gamma}^*_i - \widehat{\gamma}_i\vert}{\widehat{\tau}_i^*} = \max_{i=1,2,...,p_1}\frac{\vert -\rho_n^2\sum_{k=1}^r\frac{c_{ik}\widehat{\zeta}_k}{(\lambda_k^2 + \rho_n)^2}  + \sum_{l=1}^n\sum_{k=1}^rc_{ik}\left(\frac{\lambda_k}{\lambda_k^2 + \rho_n} + \frac{\rho_n\lambda_k}{(\lambda_k^2 + \rho_n)^2}\right)p_{lk}\epsilon^*_l\vert}{\tau_i}\\
\leq \max_{i\in\mathcal{M}}\rho_n^2\frac{\vert\sum_{k=1}^r\frac{c_{ik}\widehat{\zeta}_k}{(\lambda_k^2 + \rho_n)^2}\vert}{\tau_i} + \max_{i\in\mathcal{M}}\frac{\vert \sum_{l=1}^n\sum_{k=1}^rc_{ik}\left(\frac{\lambda_k}{\lambda_k^2 + \rho_n} + \frac{\rho_n\lambda_k}{(\lambda_k^2 + \rho_n)^2}\right)p_{lk}\epsilon^*_l\vert}{\tau_i}\\
\leq \frac{\rho_n^2\Vert\widehat{\theta}\Vert_2}{\lambda_r^3} + \max_{i\in\mathcal{M}}\frac{\vert \sum_{l=1}^n\sum_{k=1}^rc_{ik}\left(\frac{\lambda_k}{\lambda_k^2 + \rho_n} + \frac{\rho_n\lambda_k}{(\lambda_k^2 + \rho_n)^2}\right)p_{lk}\epsilon^*_l\vert}{\tau_i}\\
\max_{i=1,2,...,p_1}\frac{\vert\widehat{\gamma}^*_i - \widehat{\gamma}_i\vert}{\widehat{\tau}_i^*}\geq \max_{i\in\mathcal{M}}\frac{\vert \sum_{l=1}^n\sum_{k=1}^rc_{ik}\left(\frac{\lambda_k}{\lambda_k^2 + \rho_n} + \frac{\rho_n\lambda_k}{(\lambda_k^2 + \rho_n)^2}\right)p_{lk}\epsilon^*_l\vert}{\tau_i} - \frac{\rho_n^2\Vert\widehat{\theta}\Vert_2}{\lambda^3_r}
\end{aligned}
\label{StarJ1}
\end{equation}

From theorem \ref{thm1}, for any $a > 0$ , $\exists$ constant $D_a$ such that $\vert\widehat{\sigma}^2 - \sigma^2\vert\leq D_an^{-\alpha_\sigma}$ and $\frac{1}{2}\sigma <\widehat{\sigma}<\frac{3}{2} \sigma$ with probability $1 - a$,
\begin{equation}
\vert\sigma - \widehat{\sigma}\vert = \frac{\vert\sigma^2 - \widehat{\sigma}^2\vert}{\sigma + \widehat{\sigma}}\leq \frac{D_an^{-\alpha_\sigma}}{\sigma}
\end{equation}
If $0< x \leq n^{\alpha_\sigma / 2}$, according to lemma \ref{CI_lemma_1}, assumption 7 and \eqref{SigmaBound}, $\exists$ a constant $C^{'}$ which only depends on $\sigma,c_\mathcal{M},C_\lambda$ such that
\begin{equation}
\begin{aligned}
\vert Prob^*\left(\max_{i\in\mathcal{M}}\frac{\vert \sum_{l=1}^n\sum_{k=1}^rc_{ik}\left(\frac{\lambda_k}{\lambda_k^2 + \rho_n} + \frac{\rho_n\lambda_k}{(\lambda_k^2 + \rho_n)^2}\right)p_{lk}\epsilon^*_l\vert}{\tau_i}\leq x\right) - H(x)\vert = \vert H(\frac{x\sigma}{\widehat{\sigma}}) - H(x)\vert\\
\leq C^{'}\left(1 + \sqrt{\log(\vert\mathcal{M}\vert)} + \sqrt{\vert\log(\frac{x\vert\sigma - \widehat{\sigma}\vert}{\widehat{\sigma}})\vert}\right)\frac{x\vert\sigma - \widehat{\sigma}\vert}{\widehat{\sigma}}\\
\leq \frac{2D_aC^{'}}{\sigma^2}\left(1 + \sqrt{\log(n)}\right)n^{-\alpha_\sigma / 2} + C^{'}\sqrt{\frac{x\vert\sigma - \widehat{\sigma}\vert}{\widehat{\sigma}}\vert\log( \frac{x\vert\sigma - \widehat{\sigma}\vert}{\widehat{\sigma}})\vert}\times \sqrt{\frac{2D_a}{\sigma^2}}n^{-\alpha_\sigma/4}
\end{aligned}
\label{Sep1}
\end{equation}
Function $x\log(x)$ is continuous when $x > 0$, $x\log(x)\to 0$ as $x\to 0$, and $\frac{x\vert\sigma - \widehat{\sigma}\vert}{\widehat{\sigma}}\leq \frac{2D_an^{-\alpha_\sigma/2}}{\sigma^2}\to 0$ as $n\to\infty$. So $\sqrt{\frac{x\vert\sigma - \widehat{\sigma}\vert}{\widehat{\sigma}}\vert\log( \frac{x\vert\sigma - \widehat{\sigma}\vert}{\widehat{\sigma}})\vert}\leq \sup_{x\in(0,1]}\sqrt{\vert x\log(x)\vert} < \infty$ for sufficiently large $n$.

On the other hand, if $x  > n^{\alpha_\sigma / 2}$, then $\frac{x\sigma}{\widehat{\sigma}} > \frac{2n^{\alpha_\sigma/2}}{3}$. From lemma \ref{lemma_Se}, we may choose sufficiently large $m_1$ such that $m_1\alpha_\sigma / 2 > 2$, since $\mathbf{E}\vert\xi_1\vert^{m_1}<\infty$(Here $\xi_1$ is a normal random variable with mean $0$ and variance $\sigma^2$) is a constant for given $m_1$ and $\max_{i\in\mathcal{M}}\sum_{k=1}^r \frac{1}{\tau_i^2}c_{ik}^2\left(\frac{\lambda_k}{\lambda_k^2 + \rho_n} + \frac{\rho_n\lambda_k}{(\lambda_k^2 + \rho_n)^2}\right)^2\leq 1$, we have
\begin{equation}
\begin{aligned}
Prob\left(\max_{i\in\mathcal{M}}\frac{\vert\sum_{k=1}^r c_{ik}\left(\frac{\lambda_k}{\lambda_k^2 + \rho_n} + \frac{\rho_n\lambda_k}{(\lambda_k^2 + \rho_n)^2}\right)\xi_k\vert}{\tau_i} > \frac{2n^{\alpha_\sigma / 2}}{3}\right)\leq \frac{3^{m_1}\vert\mathcal{M}\vert\times E}{2^{m_1}n^{m_1\alpha_\sigma / 2}}\\
\Rightarrow \vert Prob^*\left(\max_{i\in\mathcal{M}}\frac{\vert \sum_{l=1}^n\sum_{k=1}^rc_{ik}\left(\frac{\lambda_k}{\lambda_k^2 + \rho_n} + \frac{\rho_n\lambda_k}{(\lambda_k^2 + \rho_n)^2}\right)p_{lk}\epsilon^*_l\vert}{\tau_i}\leq x\right) - H(x)\vert\\
\leq Prob\left(\max_{i\in\mathcal{M}}\frac{\vert\sum_{k=1}^r c_{ik}\left(\frac{\lambda_k}{\lambda_k^2 + \rho_n} + \frac{\rho_n\lambda_k}{(\lambda_k^2 + \rho_n)^2}\right)\xi_k\vert}{\tau_i} > \frac{2n^{\alpha_\sigma / 2}}{3}\right) + Prob\left(\max_{i\in\mathcal{M}}\frac{\vert\sum_{k=1}^r c_{ik}\left(\frac{\lambda_k}{\lambda_k^2 + \rho_n} + \frac{\rho_n\lambda_k}{(\lambda_k^2 + \rho_n)^2}\right)\xi_k\vert}{\tau_i} > n^{\alpha_\sigma / 2}\right)\\
\leq 2\times \frac{3^{m_1}\vert\mathcal{M}\vert\times E}{2^{m_1}n^{m_1\alpha_\sigma / 2}}
\end{aligned}
\label{Sep2}
\end{equation}
Since $H(0) = 0$, from \eqref{Sep1} and \eqref{Sep2}, for any given $a>0$ and sufficiently large $n$,
\begin{equation}
\sup_{x\geq 0}\vert Prob^*\left(\max_{i\in\mathcal{M}}\frac{\vert \sum_{l=1}^n\sum_{k=1}^rc_{ik}\left(\frac{\lambda_k}{\lambda_k^2 + \rho_n} + \frac{\rho_n\lambda_k}{(\lambda_k^2 + \rho_n)^2}\right)p_{lk}\epsilon^*_l\vert}{\tau_i}\leq x\right) - H(x)\vert < a
\label{INTZ}
\end{equation}

As a summary, for any given $a>0$, $\exists$ a constant $D_a$ such that for sufficiently large $n$, the event $\vert\widehat{\sigma}^2 - \sigma^2\vert\leq D_an^{-\alpha_\sigma}$, $\frac{1}{2}\sigma <\widehat{\sigma}<\frac{3}{2} \sigma$, $\widehat{\mathcal{N}}_{b_n} = \mathcal{N}_{b_n}$, $\Vert\widehat{\theta}\Vert_2\leq D_a\times n^{\alpha_\theta}\Rightarrow \frac{\rho_n^2\Vert\widehat{\theta}\Vert_2}{\lambda_r^3}\leq D_a^{'}n^{-\delta_1}$ for constant $D_a^{'}$ and  $\max_{i=1,2,...,p}\vert\rho_n^2\sum_{j=1}^r\frac{q_{ij}\widehat{\zeta}_j}{(\lambda_j^2 + \rho_n)^2}\vert\leq D_a \times n^{-\eta-\delta_1}$ happen with probability $1-a$. From \eqref{StarJ1}, assumption 5 and lemma \ref{CI_lemma_1}, we have for any $x\geq 0$, there exists a constant $C^{'}$ such that
\begin{equation}
\begin{aligned}
Prob^*\left(\max_{i=1,2,...,p_1}\frac{\vert\widehat{\gamma}^*_i - \widehat{\gamma}_i\vert}{\widehat{\tau}_i^*}\leq x\right) - H(x)\leq Prob^*\left( \widehat{\mathcal{N}}_{b_n}^* \neq \mathcal{N}_{b_n}\right)+ C^{'}\frac{\rho_n^2\Vert\widehat{\theta}\Vert_2}{\lambda^3_r}\times\left(1 + \sqrt{\log(\vert\mathcal{M}\vert)} + \sqrt{\vert\log(\frac{\rho_n^2\Vert\widehat{\theta}\Vert_2}{\lambda^3_r})\vert}\right)\\
+ Prob^*\left(\max_{i\in\mathcal{M}}\frac{\vert \sum_{l=1}^n\sum_{k=1}^rc_{ik}\left(\frac{\lambda_k}{\lambda_k^2 + \rho_n} + \frac{\rho_n\lambda_k}{(\lambda_k^2 + \rho_n)^2}\right)p_{lk}\epsilon^*_l\vert}{\tau_i}\leq x + \frac{\rho_n^2\Vert\widehat{\theta}\Vert_2}{\lambda^3_r}\right) - H(x + \frac{\rho_n^2\Vert\widehat{\theta}\Vert_2}{\lambda^3_r})\\
\leq a + C^{'}D_a^{'}(1+\sqrt{\log(n)})n^{-\delta_1} + C^{'}\sqrt{D_a^{'}}n^{-\delta_1/2}\sqrt{\vert \log(\frac{\rho_n^2\Vert\widehat{\theta}\Vert_2}{\lambda^3_r})\times \frac{\rho_n^2\Vert\widehat{\theta}\Vert_2}{\lambda^3_r}\vert} + Prob^*\left( \widehat{\mathcal{N}}_{b_n}^* \neq \mathcal{N}_{b_n}\right)\\
Prob^*\left(\max_{i=1,2,...,p_1}\frac{\vert\widehat{\gamma}^*_i - \widehat{\gamma}_i\vert}{\widehat{\tau}_i^*}\leq x\right) - H(x)
\geq Prob^*\left(\max_{i=1,2,...,p_1}\frac{\vert\widehat{\gamma}^*_i - \widehat{\gamma}_i\vert}{\widehat{\tau}_i^*}\leq x\cap\widehat{\mathcal{N}}_{b_n}^* = \mathcal{N}_{b_n}\right) - H(x)\\
\geq Prob^*\left(\max_{i\in\mathcal{M}}\frac{\vert \sum_{l=1}^n\sum_{k=1}^rc_{ik}\left(\frac{\lambda_k}{\lambda_k^2 + \rho_n} + \frac{\rho_n\lambda_k}{(\lambda_k^2 + \rho_n)^2}\right)p_{lk}\epsilon^*_l\vert}{\tau_i}\leq x - \frac{\rho_n^2\Vert\widehat{\theta}\Vert_2}{\lambda_r^3}\right)
-H(x - \frac{\rho_n^2\Vert\widehat{\theta}\Vert_2}{\lambda_r^3})\\ -Prob^*\left(\widehat{\mathcal{N}}_{b_n}^* \neq \mathcal{N}_{b_n}\right) - C^{'}D_a^{'}(1+\sqrt{\log(n)})n^{-\delta_1} - C^{'}\sqrt{D_a^{'}}n^{-\delta_1/2}\sqrt{\vert \log(\frac{\rho_n^2\Vert\widehat{\theta}\Vert_2}{\lambda^3_r})\times \frac{\rho_n^2\Vert\widehat{\theta}\Vert_2}{\lambda^3_r}\vert}
\end{aligned}
\end{equation}
If $0\leq x\leq \frac{\rho_n^2\Vert\widehat{\theta}\Vert_2}{\lambda_r^3}$, then $Prob^*\left(\max_{i\in\mathcal{M}}\frac{\vert \sum_{l=1}^n\sum_{k=1}^rc_{ik}\left(\frac{\lambda_k}{\lambda_k^2 + \rho_n} + \frac{\rho_n\lambda_k}{(\lambda_k^2 + \rho_n)^2}\right)p_{lk}\epsilon^*_l\vert}{\tau_i}\leq x - \frac{\rho_n^2\Vert\widehat{\theta}\Vert_2}{\lambda_r^3}\right) = H(x - \frac{\rho_n^2\Vert\widehat{\theta}\Vert_2}{\lambda_r^3}) = 0$. Therefore, for sufficiently large $n$, from \eqref{INTZ} and \eqref{StarVar}, $\exists$ a constant $C$ such that
\begin{equation}
\begin{aligned}
\sup_{x\geq 0}\vert Prob^*\left(\max_{i=1,2,...,p_1}\frac{\vert\widehat{\gamma}^*_i - \widehat{\gamma}_i\vert}{\widehat{\tau}_i^*}\leq x\right) - H(x)\vert\leq \frac{pE\widehat{\sigma}^m}{c_\lambda^mn^{m\eta}b_n^m}\times \left(2^m + \frac{2^m}{(1/c_b - 1)^m}\right) + a \\
+ C^{'}D_a^{'}(1+\sqrt{\log(n)})n^{-\delta_1} + C^{'}\sqrt{D_a^{'}}n^{-\delta_1/2}\sqrt{\sup_{x\in(0,1]}\vert x\log(x)\vert}
\leq Cn^{m(\nu_b + \alpha_p/m - \eta)}+ 2a
\end{aligned}
\end{equation}
and we prove \eqref{EstFir}.

For any given $a > 0$, from the first result, for sufficiently large $n$, we have
\begin{equation}
Prob\left(\sup_{x\geq 0}\vert Prob^*\left(\max_{i=1,2,...,p_1}\frac{\vert\widehat{\gamma}^*_i - \widehat{\gamma}_i\vert}{\widehat{\tau}_i^*}\leq x\right) - H(x)\vert \leq a\right) > 1-a
\label{FIRRES}
\end{equation}
Choose sufficiently small $a$ such that $0< 1-\alpha - 2a < 1-\alpha + 2a < 1$. If \eqref{FIRRES} happens, for any $1 >\alpha > 0$, define $c_{1-\alpha}$ as the $1-\alpha$ quantile of $H(x)$,
\begin{equation}
\begin{aligned}
Prob^*\left(\max_{i=1,2,...,p_1}\frac{\vert\widehat{\gamma}^*_i - \widehat{\gamma}_i\vert}{\widehat{\tau}_i^*}\leq c_{1-\alpha + 2a}\right) - (1-\alpha + 2a)\geq -a\Rightarrow c^*_{1-\alpha}\leq c_{1-\alpha + 2a}\\
Prob^*\left(\max_{i=1,2,...,p_1}\frac{\vert\widehat{\gamma}^*_i - \widehat{\gamma}_i\vert}{\widehat{\tau}_i^*}\leq c_{1-\alpha - 2a}\right) - (1-\alpha - 2a)\leq a\Rightarrow c^*_{1-\alpha} > c_{1-\alpha - 2a}
\end{aligned}
\end{equation}
From theorem \ref{Thm2}, we have for sufficiently large $n$,
\begin{equation}
\begin{aligned}
Prob\left(\max_{i=1,2,...,p_1}\frac{\vert\widehat{\gamma}_i - \gamma_i\vert}{\widehat{\tau}_i}\leq c^*_{1-\alpha}\right)\\
\leq Prob\left(\sup_{x\geq 0}\vert Prob^*\left(\max_{i=1,2,...,p_1}\frac{\vert\widehat{\gamma}^*_i - \widehat{\gamma}_i\vert}{\widehat{\tau}_i^*}\leq x\right) - H(x)\vert > a\right)
+ Prob\left(\max_{i=1,2,...,p_1}\frac{\vert\widehat{\gamma}_i - \gamma_i\vert}{\widehat{\tau}_i}\leq c_{1-\alpha + 2a}\right)\\
\leq a + (H(c_{1-\alpha + 2a}) + a) = 1-\alpha + 4a\\
Prob\left(\max_{i=1,2,...,p_1}\frac{\vert\widehat{\gamma}_i - \gamma_i\vert}{\widehat{\tau}_i}\leq c^*_{1-\alpha}\right)
\geq Prob\left(\max_{i=1,2,...,p_1}\frac{\vert\widehat{\gamma}_i - \gamma_i\vert}{\widehat{\tau}_i}\leq c^*_{1-\alpha}\cap \sup_{x\geq 0}\vert Prob^*\left(\max_{i=1,2,...,p_1}\frac{\vert\widehat{\gamma}^*_i - \widehat{\gamma}_i\vert}{\widehat{\tau}_i^*}\leq x\right) - H(x)\vert\leq a\right)\\
\geq Prob\left(\max_{i=1,2,...,p_1}\frac{\vert\widehat{\gamma}_i - \gamma_i\vert}{\widehat{\tau}_i}\leq c_{1-\alpha - 2a}\right) - Prob\left( \sup_{x\geq 0}\vert Prob^*\left(\max_{i=1,2,...,p_1}\frac{\vert\widehat{\gamma}^*_i - \widehat{\gamma}_i\vert}{\widehat{\tau}_i^*}\leq x\right) - H(x)\vert > a\right)\\
\geq (H(c_{1-\alpha-2a}) - a) - a = 1-\alpha - 4a
\Rightarrow \vert Prob\left(\max_{i=1,2,...,p_1}\frac{\vert\widehat{\gamma}_i - \gamma_i\vert}{\widehat{\tau}_i}\leq c^*_{1-\alpha}\right) - (1-\alpha)\vert\leq 4a
\end{aligned}
\end{equation}
For $a>0$ can be arbitrarily small, we prove \eqref{EstSec}.
\end{proof}

\section{Proofs of theorems in section \ref{BOOTPRD}}

\begin{proof}[Proof of lemma \ref{LEMMARES}]
Define the design matrix $X = (x_{ij})_{i = 1,...,n, j = 1,...,p}$, $\overline{x}_j = \frac{1}{n}\sum_{i=1}^nx_{ij}$, and $x_{ij}^{'} = x_{ij} - \overline{x}_j$. If $\widehat{\mathcal{N}}_{b_n} = \mathcal{N}_{b_n}$, for $i=1,2,...,n$,
\begin{equation}
\widehat{\epsilon}_i^{'} = \epsilon_i + \sum_{j\not\in\mathcal{N}_{b_n}}x_{ij}\theta_j - \sum_{j\in\mathcal{N}_{b_n}}x_{ij}(\widetilde{\theta}_j - \theta_j) \Rightarrow \widehat{\epsilon}_i = \epsilon_i - \frac{1}{n}\sum_{i=1}^n\epsilon_i + \sum_{j\not\in\mathcal{N}_{b_n}}x_{ij}^{'}\theta_j - \sum_{j\in\mathcal{N}_{b_n}}x_{ij}^{'}(\widetilde{\theta}_j - \theta_j)
\end{equation}
Define $\widetilde{F}(x) = \frac{1}{n}\sum_{i=1}^n\mathbf{1}_{\epsilon_i\leq x},\ x\in\mathbf{R}$. From \eqref{prop}, for any given $\psi > 0$,
\begin{equation}
\begin{aligned}
\widehat{F}(x) - F(x) = \left(\widehat{F}(x) - \widetilde{F}(x + 1/\psi)\right) + \left(\widetilde{F}(x + 1/\psi) - F(x + 1/\psi)\right) + \left(F(x + 1/\psi) - F(x)\right)\\
\leq \frac{1}{n}\sum_{i=1}^n(g_{\psi,x}(\widehat{\epsilon}_i) - g_{\psi,x}(\epsilon_i)) + \sup_{x\in\mathbf{R}}\vert \widetilde{F}(x) - F(x)\vert + \left(F(x + 1/\psi) - F(x)\right)\\
\leq g_*\psi\sqrt{\frac{1}{n}\sum_{i=1}^n(\widehat{\epsilon}_i - \epsilon_i)^2}+ \sup_{x\in\mathbf{R}}\vert \widetilde{F}(x) - F(x)\vert + \left(F(x + 1/\psi) - F(x)\right)\\
\widehat{F}(x) - F(x) = \left(\widehat{F}(x) - \widetilde{F}(x - 1/\psi)\right) + \left(\widetilde{F}(x - 1/\psi) - F(x - 1/\psi)\right) - \left(F(x) - F(x - 1/\psi)\right)\\
\geq \frac{1}{n}\sum_{i=1}^n(g_{\psi,x - 1/\psi}(\widehat{\epsilon}_i) - g_{\psi, x - 1/\psi}(\epsilon_i)) - \sup_{x\in\mathbf{R}}\vert\widetilde{F}(x) - F(x)\vert - \left(F(x) - F(x - 1/\psi)\right)\\
\geq -g_*\psi\sqrt{\frac{1}{n}\sum_{i=1}^n(\widehat{\epsilon}_i - \epsilon_i)^2} - \sup_{x\in\mathbf{R}}\vert\widetilde{F}(x) - F(x)\vert - \left(F(x) - F(x - 1/\psi)\right)\\
\Rightarrow \sup_{x\in\mathbf{R}}\vert \widehat{F}(x) - F(x)\vert\leq g_*\psi\sqrt{\frac{1}{n}\sum_{i=1}^n(\widehat{\epsilon}_i - \epsilon_i)^2} + \sup_{x\in\mathbf{R}}\vert\widetilde{F}(x) - F(x)\vert + \sup_{x\in\mathbf{R}}\vert F(x + 1/\psi) - F(x)\vert
\end{aligned}
\end{equation}
Suppose assumption 1 to 6. From \eqref{SIGMA}, \eqref{SIGMA2}, \eqref{XTHETA} and $\frac{1}{n}\sum_{i=1}^n\epsilon_i = O_p(1/\sqrt{n})$, for any $0<a<1$, $\exists$ a constant $C_a$ such that with probability at least $1-a$
\begin{equation}
\begin{aligned}
\frac{1}{n}\sum_{i=1}^n(\widehat{\epsilon}_i - \epsilon_i)^2 =
\frac{1}{n}\sum_{i=1}^n\left(\sum_{j\not\in\mathcal{N}_{b_n}}x_{ij}^{'}\theta_j - \sum_{j\in\mathcal{N}_{b_n}}x_{ij}^{'}(\widetilde{\theta}_j - \theta_j) - \frac{1}{n}\sum_{j=1}^n\epsilon_j\right)^2\\
\leq \frac{3}{n}\sum_{i=1}^n\left(\sum_{j\not\in\mathcal{N}_{b_n}}x_{ij}^{'}\theta_j\right)^2 + \frac{3}{n}\sum_{i=1}^n\left(\sum_{j\in\mathcal{N}_{b_n}}x_{ij}^{'}(\widetilde{\theta}_j - \theta_j)\right)^2 + 3\left(\frac{1}{n}\sum_{j=1}^n\epsilon_j\right)^2\\
\leq \frac{6}{n}\sum_{i=1}^n\left(\sum_{j\not\in\mathcal{N}_{b_n}}x_{ij}\theta_j\right)^2 + 6\left(\sum_{j\not\in\mathcal{N}_{b_n}}\overline{x}_j\theta_j\right)^2 + \frac{6}{n}\sum_{i=1}^n\left(\sum_{j\in\mathcal{N}_{b_n}}x_{ij}(\widetilde{\theta}_j - \theta_j)\right)^2 + 6\left(\sum_{j\in\mathcal{N}_{b_n}}\overline{x}_j(\widetilde{\theta}_j-\theta_j)\right)^2 + 3\left(\frac{1}{n}\sum_{j=1}^n\epsilon_j\right)^2\\
\leq C_an^{-\alpha_\sigma} + \frac{6}{n^2}\left(\sum_{i=1}^n\sum_{j\not\in\mathcal{N}_{b_n}}x_{ij}\theta_j\right)^2 + C_a\vert\mathcal{N}_{b_n}\vert(n^{2\alpha_\theta -4\delta} + n^{-2\eta})+ \frac{6}{n^2}\left(\sum_{i=1}^n\sum_{j\in\mathcal{N}_{b_n}}x_{ij}(\widetilde{\theta}_j-\theta_j)\right)^2 + \frac{C_a}{n}\\
\leq C_an^{-\alpha_\sigma} + \frac{6}{n}\sum_{i=1}^n\left(\sum_{j\not\in\mathcal{N}_{b_n}}x_{ij}\theta_j\right)^2 + C_a\vert\mathcal{N}_{b_n}\vert(n^{2\alpha_\theta -4\delta} + n^{-2\eta})+ \frac{6}{n}\sum_{i=1}^n\left(\sum_{j\in\mathcal{N}_{b_n}}x_{ij}(\widetilde{\theta}_j-\theta_j)\right)^2 + \frac{C_a}{n}\\
\Rightarrow \sqrt{\frac{1}{n}\sum_{i=1}^n(\widehat{\epsilon}_i - \epsilon_i)^2} = O_p(n^{-\alpha_\sigma / 2})
\end{aligned}
\label{DTS}
\end{equation}
According to Gilvenko-Cantelli lemma, $\sup_{x\in\mathbf{R}}\vert\widetilde{F}(x) - F(x)\vert\to 0$ almost surely. Therefore, for any $a > 0$ and sufficiently large $n$, $Prob\left(\sup_{x\in\mathbf{R}}\vert\widetilde{F}(x) - F(x)\vert\leq a\right)>1-a$. Choose sufficiently small $a$ and $\psi = 1/a$, from assumption 8. and \eqref{DTS}, we prove \eqref{PrdEq}.
\end{proof}

\begin{proof}[Proof of theorem \ref{THMPD}]
Define $X_f = (x_{f,ij})_{i=1,...,p_1,j=1,...,p}$. From theorem \ref{THM_CON}, since $p_1 = O(1)$,
\begin{equation}
\max_{i=1,2,...,p_1}\vert \sum_{j=1}^px_{f,ij}\widehat{\theta}_j - \sum_{j=1}^p x_{f,ij}\beta_j\vert = O_p(n^{-\eta})
\end{equation}
For any given $0<a<1$, choose a constant $C_a$ such that $Prob\left(\max_{i=1,2,...,p_1}\vert \sum_{j=1}^px_{f,ij}\widehat{\theta}_j - \sum_{j=1}^p x_{f,ij}\beta_j\vert\leq C_an^{-\eta}\right) \geq 1 - a$ for any $n=1,2,...$. Define $F^{-}(x) = \lim_{y<x,y\to x}F(y)$ for any $x\in\mathbf{R}$, and $G(x) = Prob\left(\max_{i=1,2,...,p_1}\vert\epsilon_{f,i}\vert\leq x\right) = (F(x) - F^{-}(-x))^{p_1}$ for $x\geq 0$. $G$ is continuous according to assumption 8. With probability at least $1-a$

\begin{equation}
\begin{aligned}
\sup_{x\geq 0}\vert Prob^*\left(\max_{i=1,2,...,p_1}\vert y_{f,i} - \sum_{j=1}^p x_{f,ij}\widehat{\theta}_j\vert\leq x\right) - G(x)\vert\\
\leq \sup_{x\geq 0}\vert Prob^*\left(\max_{i=1,2,...,p_1}\vert\epsilon_{f,i}\vert\leq x + \max_{i=1,2,...,p_1}\vert\sum_{j=1}^p x_{f,ij}(\beta_j - \widehat{\theta}_j)\vert\right) - G(x)\vert\\
+ \sup_{x\geq 0}\vert Prob^*\left(\max_{i=1,2,...,p_1}\vert\epsilon_{f,i}\vert\leq x - \max_{i=1,2,...,p_1}\vert\sum_{j=1}^p x_{f,ij}(\beta_j - \widehat{\theta}_j)\vert\right) - G(x)\vert\\
\leq \sup_{x\geq 0}\vert G(x + C_an^{-\eta}) - G(x)\vert + \sup_{x\geq 0}\vert G(x) - G(\max(0, x - C_an^{-\eta}))\vert
\end{aligned}
\label{PDF}
\end{equation}
For any $\delta > 0$ and any $x\geq 0$
\begin{equation}
\begin{aligned}
G(x+\delta) - G(x) = \sum_{i = 1}^{p_1} (F(x+\delta) - F(-x-\delta))^{i-1}\times(F(x) - F(-x))^{p_1-i}\times(F(x+\delta) - F(-x-\delta) - F(x) + F(-x))\\
\leq 2p_1\times\sup_{x\in\mathbf{R}}(F(x+\delta) - F(x)) \Rightarrow \sup_{x\geq 0}\left( G(x+\delta) - G(x)\right)\leq 2p_1\times\sup_{x\in\mathbf{R}}(F(x+\delta) - F(x))
\end{aligned}
\end{equation}
From \eqref{PDF} and assumption 8
\begin{equation}
\begin{aligned}
\sup_{x\geq 0}\vert Prob^*\left(\max_{i=1,2,...,p_1}\vert y_{f,i} - \sum_{j=1}^p x_{f,ij}\widehat{\theta}_j\vert\leq x\right) - G(x)\vert = o_p(1)
\end{aligned}
\label{PBG}
\end{equation}
If $\widehat{\mathcal{N}}_{b_n} = \mathcal{N}_{b_n}$, $\frac{\sigma}{2}<\widehat{\sigma} < \frac{3\sigma}{2}$, and $\Vert\widehat{\theta}\Vert_2\leq C\times n^{\alpha_\theta}$, then
\begin{equation}
\max_{i=1,2,...,p}\vert\rho_n^2\sum_{j=1}^r\frac{q_{ij}\widehat{\zeta}_j}{(\lambda_j^2 + \rho_n)^2}\vert\leq C n^{-\eta-\delta_1},\ \text{and } \max_{i=1,2,...,p}\vert \sum_{j=1}^rq_{ij}\left(\frac{\lambda_j}{\lambda_j^2 + \rho_n} + \frac{\rho_n\lambda_j}{(\lambda_j^2 + \rho_n)^2}\right)\sum_{l=1}^np_{lj}\epsilon_l\vert\leq Cn^{\alpha_p/m - \eta}
\label{ZZZ}
\end{equation}
for some constant $C$. Here $2\delta = \eta + \alpha_\theta +\delta_1$. From \eqref{StarVar}, $\exists$ a constant $E$ such that
\begin{equation}
Prob^*\left(\widehat{\mathcal{N}}_{b_n}^*\neq\mathcal{N}_{b_n}\right)\leq \frac{Ep}{n^{m\eta}b_n^m}
\label{SELE}
\end{equation}
If $\widehat{\mathcal{N}}_{b_n}^* = \mathcal{N}_{b_n}$,
\begin{equation}
\begin{aligned}
\vert\sum_{j=1}^px_{f,ij}\widehat{\theta}^*_j - \sum_{j=1}^px_{f,ij}\widehat{\theta}_j\vert = \vert\sum_{j\in\mathcal{N}_{b_n}}x_{f,ij}(\widetilde{\theta}^*_j-\widehat{\theta}_j)\vert \leq \rho_n^2\vert\sum_{k=1}^r\frac{c_{ik}\widetilde{\zeta}_k}{(\lambda_k^2 + \rho_n)^2}\vert + \vert\sum_{k=1}^r\sum_{l=1}^nc_{ik}\left(\frac{\lambda_k}{\lambda_k^2 + \rho_n} + \frac{\rho_n\lambda_k}{(\lambda_k^2 + \rho_n)^2}\right)p_{lk}\epsilon^*_l\vert\\
\leq \rho_n^2\frac{\sqrt{C_\mathcal{M}}\Vert\widehat{\theta}\Vert_2}{\lambda_r^4} + \vert\sum_{k=1}^r\sum_{l=1}^nc_{ik}\left(\frac{\lambda_k}{\lambda_k^2 + \rho_n} + \frac{\rho_n\lambda_k}{(\lambda_k^2 + \rho_n)^2}\right)p_{lk}\epsilon^*_l\vert
\end{aligned}
\end{equation}
Form \eqref{Delta2} and lemma \ref{lemma_Se}, $\exists$ a constant $E$ which only depends on $m$, and for any $1 > a>0$ with a sufficiently large $C_a>0$,
\begin{equation}
Prob^*\left(\max_{i=1,2,...,p_1}\vert\sum_{k=1}^r\sum_{l=1}^nc_{ik}\left(\frac{\lambda_k}{\lambda_k^2 + \rho_n} + \frac{\rho_n\lambda_k}{(\lambda_k^2 + \rho_n)^2}\right)p_{lk}\frac{\epsilon^*_l}{\widehat{\sigma}}\vert > \frac{C_an^{-\eta}}{\widehat{\sigma}}\right)\leq \frac{p_1E\widehat{\sigma}^m}{n^{m\eta}C_a^m n^{-m\eta}} < a
\end{equation}
Combine with \eqref{SELE}, there exists a constant $C_a$, with conditional probability at least $1-a$
\begin{equation}
\begin{aligned}
\max_{i=1,2,...,p_1}\vert\sum_{j=1}^px_{f,ij}\widehat{\theta}^*_j - \sum_{j=1}^px_{f,ij}\widehat{\theta}_j\vert\leq C_an^{-\eta}\\
\Rightarrow Prob^*\left(\max_{i=1,2,...,p_1}\vert y^*_{f,i} - \widehat{y}^*_{f,i}\vert\leq x\right) - G(x)\leq a + Prob^*\left(\max_{i=1,2,...,p_1}\vert\epsilon^*_{f,i}\vert\leq x + C_an^{-\eta}\right) - G(x)\\
\leq a + \sup_{x\geq 0}\vert Prob^*\left(\max_{i=1,2,...,p_1}\vert\epsilon^*_{f,i}\vert\leq x\right) - G(x)\vert + 2p_1\sup_{x\in\mathbf{R}}(F(x+C_an^{-\eta}) - F(x))\\
Prob^*\left(\max_{i=1,2,...,p_1}\vert y^*_{f,i} - \widehat{y}^*_{f,i}\vert\leq x\right) - G(x)\geq -a + Prob^*\left(\max_{i=1,2,...,p_1}\vert\epsilon^*_{f,i}\vert\leq x - C_an^{-\eta}\right) - G(x)\\
\geq -a + Prob^*\left(\max_{i=1,2,...,p_1}\vert\epsilon^*_{f,i}\vert\leq x - C_an^{-\eta}\right) - G(x - C_an^{-\eta}) - 2p_1\sup_{x\in\mathbf{R}}(F(x+C_an^{-\eta}) - F(x))
\end{aligned}
\end{equation}
Since $G(x) = 0$ and  $ Prob^*\left(\max_{i=1,2,...,p_1}\vert\epsilon^*_{f,i}\vert\leq x\right) = 0 $ if $x < 0$, we have
\begin{equation}
\sup_{x\geq 0}\vert Prob^*\left(\max_{i=1,2,...,p_1}\vert y^*_{f,i} - \widehat{y}^*_{f,i}\vert\leq x\right) - G(x)\vert\leq a + \sup_{x\geq 0}\vert Prob^*\left(\max_{i=1,2,...,p_1}\vert\epsilon^*_{f,i}\vert\leq x\right) - G(x)\vert + 2p_1\sup_{x\in\mathbf{R}}(F(x+C_an^{-\eta}) - F(x))
\end{equation}
From lemma \ref{LEMMARES}, for any $x\geq 0$,
\begin{equation}
\begin{aligned}
\vert Prob^*\left(\max_{i=1,2,...,p_1}\vert\epsilon^*_{f,i}\vert\leq x\right) - G(x)\vert = \vert\left(\widehat{F}(x) - \widehat{F}^{-}(-x)\right)^{p_1} - (F(x) - F(-x))^{p_1}\vert\\
\leq \sum_{i=1}^{p_1}\vert \widehat{F}(x) - \widehat{F}^{-}(-x)\vert^{i - 1}\times\vert F(x) - F(-x)\vert^{p_1 - i}\times\left(\vert\widehat{F}(x) - F(x)\vert + \vert \widehat{F}^{-}(-x) - F^{-}(-x)\vert\right)
\leq 2p_1\sup_{x\in\mathbf{R}}\vert \widehat{F}(x) - F(x)\vert\to_p 0
\end{aligned}
\end{equation}
as $n\to\infty$. From theorem \ref{thm1} and \eqref{Var1} to \eqref{Var3}, for any $1>a>0$, with probability at least $1-a$ $\exists$ a constant $C_a > 0$ such that for sufficiently large $n$, \eqref{ZZZ} happens with $C = C_a$ and $\sup_{x\geq 0}\vert Prob^*\left(\max_{i=1,2,...,p_1}\vert\epsilon^*_{f,i}\vert\leq x\right) - G(x)\vert < a$. Correspondingly for sufficiently large $n$, with probability at least $1-a$,
\begin{equation}
\begin{aligned}
\sup_{x\geq 0}\vert Prob^*\left(\max_{i=1,2,...,p_1}\vert y^*_{f,i} - \widehat{y}^*_{f,i}\vert\leq x\right) - Prob^*\left(\max_{i=1,2,...,p_1}\vert y_{f,i} - \widehat{y}_{f,i}\vert\leq x\right)\vert\\
\leq \sup_{x\geq 0}\vert Prob^*\left(\max_{i=1,2,...,p_1}\vert y^*_{f,i} - \widehat{y}^*_{f,i}\vert\leq x\right) - G(x)\vert + \sup_{x\geq 0}\vert Prob^*\left(\max_{i=1,2,...,p_1}\vert y_{f,i} - \widehat{y}_{f,i}\vert\leq x\right) - G(x) \vert\\
\leq a + \sup_{x\geq 0}\vert Prob^*\left(\max_{i=1,2,...,p_1}\vert\epsilon^*_{f,i}\vert\leq x\right) - G(x)\vert + 2p_1\sup_{x\in\mathbf{R}}(F(x+C_an^{-\eta}) - F(x)) + a\leq 4a
\end{aligned}
\label{Double}
\end{equation}
and we prove \eqref{PDFIR}.

For given $0<\alpha < 1$ and sufficiently small $a>0$ such that $0 < 1-\alpha-a < 1-\alpha + a < 1$, define $c_{1-\alpha}$ as the $1-\alpha$ quantile of $G(x)$. For $G(x)$ is continuous, $G(c_{1-\alpha}) = 1-\alpha$. With probability at least $1-a$, $\sup_{x\geq 0}\vert Prob^*\left(\max_{i=1,2,...,p_1}\vert y^*_{f,i} - \widehat{y}^*_{f,i}\vert\leq x\right) - G(x)\vert<a/2$. Correspondingly with probability at least $1-a$
\begin{equation}
\begin{aligned}
Prob^*\left(\max_{i=1,2,...,p_1}\vert y^*_{f,i} - \widehat{y}^*_{f,i}\vert\leq c_{1-\alpha + a}\right)\geq 1-\alpha+a/2\Rightarrow c^*_{1-\alpha}\leq c_{1-\alpha + a}\\
Prob^*\left(\max_{i=1,2,...,p_1}\vert y^*_{f,i} - \widehat{y}^*_{f,i}\vert\leq c_{1-\alpha - a}\right)\leq 1-\alpha -a/2\Rightarrow c^*_{1-\alpha}\geq c_{1-\alpha -a} \\
\end{aligned}
\end{equation}
From \eqref{PBG}, for sufficiently large $n$, with probability at least $1-a$
\begin{equation}
\begin{aligned}
Prob^*\left(\max_{i=1,2,...,p_1}\vert y_{f,i} - \widehat{y}_{f,i}\vert\leq c^*_{1-\alpha}\right)\leq Prob^*\left(\max_{i=1,2,...,p_1}\vert y_{f,i} - \widehat{y}_{f,i}\vert\leq c_{1-\alpha + a}\right)\leq 1-\alpha + 2a\\
Prob^*\left(\max_{i=1,2,...,p_1}\vert y_{f,i} - \widehat{y}_{f,i}\vert\leq c^*_{1-\alpha}\right)\geq Prob^*\left(\max_{i=1,2,...,p_1}\vert y_{f,i} - \widehat{y}_{f,i}\vert\leq c_{1-\alpha - a}\right)\geq 1-\alpha - 2a
\end{aligned}
\end{equation}
For $a>0$ can be arbitrarily small, we prove \eqref{PDFCI}.
\end{proof}
\end{document}